\documentclass[10pt]{amsart}
\numberwithin{equation}{section}

\usepackage{amsmath}
\usepackage{amscd,amsthm,amssymb,amsfonts}
\usepackage{mathrsfs}
\usepackage{dsfont}
\usepackage{stmaryrd}
\usepackage{euscript}
\usepackage{expdlist}
\usepackage{enumerate}

\input xy

\newtheorem{thm}{Theorem}[section]
\newtheorem{thmA}{Theorem}

\newtheorem{assumption}[thm]{Assumption}

\newtheorem*{thm*}{Theorem}

\newtheorem{lm}[thm]{Lemma}

\newtheorem*{cor*}{Corollary}
\newtheorem{prop}[thm]{Proposition}
\newtheorem*{conj*}{Conjecture}

\theoremstyle{Remark}

\theoremstyle{definition}
\newtheorem*{defn*}{Definition}

\newtheorem{I_Remark*}{Remark}

\newcommand{\nc}{\newcommand}

\newcommand{\beq}{\begin{equation}}
\newcommand{\eeq}{\end{equation}}
\newcommand{\bpmx}{\begin{pmatrix}}
\newcommand{\epmx}{\end{pmatrix}}
\newcommand{\bbmx}{\begin{bmatrix}}
\newcommand{\ebmx}{\end{bmatrix}}

\def\parref#1{\ref{#1}}
\def\thmref#1{Theorem~\parref{#1}}

\def\propref#1{Proposition~\parref{#1}}
     
\def\secref#1{\S\parref{#1}}

\def\lmref#1{Lemma~\parref{#1}}

\def\makeop#1{\expandafter\def\csname#1\endcsname
  {\mathop{\rm #1}\nolimits}\ignorespaces}

\makeop{Hom}   \makeop{End}   \makeop{Aut}   
\makeop{Pic} \makeop{Gal}       \makeop{Div} \makeop{Lie}
\makeop{PGL}   \makeop{Corr} \makeop{PSL} \makeop{sgn} \makeop{Spf}
 \makeop{Tr} \makeop{Nr} \makeop{Fr} \makeop{disc}
\makeop{Proj} \makeop{supp} \makeop{ker}   \makeop{Im} \makeop{dom}
\makeop{coker} \makeop{Stab} \makeop{SO} \makeop{SL} \makeop{SL}
\makeop{Cl}    \makeop{cond} \makeop{Br} \makeop{inv} \makeop{rank}
\makeop{id}    \makeop{Fil} \makeop{Frac}  \makeop{GL} \makeop{SU}
\makeop{Trd}   \makeop{Sp} \makeop{Tr}    \makeop{Trd} \makeop{Res}
\makeop{ind} \makeop{depth} \makeop{Tr} \makeop{st} \makeop{Ad}
\makeop{Int} \makeop{tr}    \makeop{Sym} \makeop{can} \makeop{SO}
\makeop{torsion} \makeop{GSp} \makeop{Tor}\makeop{Ker} \makeop{rec}
\makeop{Ind} \makeop{Coker}
 \makeop{vol} \makeop{Ext} \makeop{gr} \makeop{ad}
 \makeop{Gr}\makeop{corank} \makeop{Ann}
\makeop{Hol} 
\makeop{Fitt} \makeop{Mp} \makeop{CAP}

\def\makebb#1{\expandafter\def
  \csname bb#1\endcsname{{\mathbb{#1}}}\ignorespaces}
\def\makebf#1{\expandafter\def\csname bf#1\endcsname{{\bf
      #1}}\ignorespaces}
\def\makegr#1{\expandafter\def
  \csname gr#1\endcsname{{\mathfrak{#1}}}\ignorespaces}
\def\makescr#1{\expandafter\def
  \csname scr#1\endcsname{{\EuScript{#1}}}\ignorespaces}
\def\makecal#1{\expandafter\def\csname cal#1\endcsname{{\mathcal
      #1}}\ignorespaces}

\def\doLetters#1{#1A #1B #1C #1D #1E #1F #1G #1H #1I #1J #1K #1L #1M
                 #1N #1O #1P #1Q #1R #1S #1T #1U #1V #1W #1X #1Y #1Z}
\def\doletters#1{#1a #1b #1c #1d #1e #1f #1g #1h #1i #1j #1k #1l #1m
                 #1n #1o #1p #1q #1r #1s #1t #1u #1v #1w #1x #1y #1z}
\doLetters\makebb   \doLetters\makecal  \doLetters\makebf
\doLetters\makescr
\doletters\makebf   \doLetters\makegr   \doletters\makegr

    \def\setminus{\smallsetminus}

\normalsize

\makeop{Ram} \makeop{Rep} \makeop{mass}

\makeop{Bl}

\def\cR{{\mathcal R}}

\def\cS{{\mathcal S}}

\def\cW{{\mathcal W}}

\def\cV{{\mathcal V}}

\def\sR{\mathscr R}

\newcommand{\C}{\mathbf C}
\newcommand{\A}{\mathbf A}    

\def\bbE{{\mathbb E}}

\def\etale{{\'{e}tale }}

\def\ot{\otimes}

\def\hookto{\hookrightarrow}
\def\longto{\longrightarrow}
  \nc{\opp}{\mathrm{opp}} \nc{\ul}{\underline}

\def\XYmatrix{\xymatrix@M=8pt} 
\def\ncmd{\newcommand}
\ncmd{\xysubset}[1][r]{\ar@<-2.5pt>@{^(-}[#1]\ar@<2.5pt>@{_(-}[#1]}
\ncmd{\XYmatrixc}[1]{\vcenter{\XYmatrix{#1}}}
\ncmd{\xyto}[1][r]{\ar@{->}[#1]}
\ncmd{\xyinj}[1][r]{\ar@{^(->}[#1]}
\ncmd{\xysurj}[1][r]{\ar@{->>}[#1]}
\ncmd{\xyline}[1][r]{\ar@{-}[#1]}
\ncmd{\xydotsto}[1][r]{\ar@{.>}[#1]}
\ncmd{\xydots}[1][r]{\ar@{.}[#1]}
\ncmd{\xyleadsto}[1][r]{\ar@{~>}[#1]}
\ncmd{\xyeq}[1][r]{\ar@{=}[#1]} \ncmd{\xyequal}[1][r]{\ar@{=}[#1]}
\ncmd{\xyequals}[1][r]{\ar@{=}[#1]}
\ncmd{\xymapsto}[1][r]{l\ar@{|->}[#1]}\ncmd{\xyimplies}[1][r]{\ar@{=>}[#1]}
\ncmd{\xyiso}{\ar[r]_-{\sim}}
\def\injxy{\ar@{^(->}}

\newcommand{\pMX}[4]{\begin{pmatrix}
{#1}& {#2}\\
{#3}&{#4}\end{pmatrix} }

\newcommand{\seesaw}[4]{{#1}\ar@{-}[rd]\ar@{-}[d]&{#2}\ar@{-}[d]\\
{#3}\ar@{-}[ru]&{#4}}

\def\x{{\times}}

\newcommand\stt[1]{\left\{#1\right\}}

\renewcommand\Re{\text{Re}\,}

\usepackage{color}
\usepackage{hyperref}
\usepackage{MnSymbol}
\usepackage{mathdots}
\usepackage{tikz-cd}
\usepackage{graphicx}

\def\G{{\rm G}}
\def\U{{\rm U}}

\def\M{{\rm Mat}}

\def\t{\tilde}
\def\b{\bar}
\def\u{\underline}
\def\bt{\boxtimes}
\def\h{\hat}

\title{On gamma factors of generic representations of $\U_{2n+1}\x{\rm Res}_{E/F}\GL_r$}
\author{Yao Cheng and Chian-Jen Wang}
\date{\today}
\address{No. 151, Yingzhuan Road, Tamsui District, New Taipei City 251, Taiwan (R.O.C),  Lui-Hsien Memorial 
Science Hall.}
\email{briancheng@o365.tku.edu.tw}
\email{142888@o365.tku.edu.tw}

\begin{document}
\maketitle
\begin{abstract}
In this article we prove that the gamma factors attached to generic representations of $\U_{2n+1}\x{\rm Res}_{E/F}\GL_r$
over a local field $F$ defined by various approaches coincide when $F$ is non-archimedean and under additional 
assumptions on $n,r$ and $E$ when $F$ is archimedean.
\end{abstract}

\section{Introduction}

\subsection{Main Results}
Let $F$ be a local field of characteristic zero, $E$ an \etale $F$-algebra of rank $2$ and
$\psi$ be a non-trivial additive character of $F$.
Denote by $\G_N={\Res}_{E/F}\GL_N$ the Weil 
restriction of $\GL_N$ from $E$ to $F$, and $\U_N\subset\G_N$ a quasi-split unitary group of $N$-variables.
Let $\pi$ and $\tau$ be irreducible generic (complex) representations of $\U_{2n+1}(F)$ and $\G_r(F)$ 
respectively with $n\ge 0$ and $r\ge 1$. To $\pi, \tau$ and $\psi$, one can define the gamma factors through the following 
approaches:

\begin{itemize}
\item (WD) the associated Weil-Delinge representation (\cite{Mok2015}, \cite{Tate1979}); 
\item (LS) the Langlands-Shahidi method (\cite{Shahidi1990});
\item (RS) the Rankin-Selberg integrals (\cite{Ben-ArtziSoudry2009}, \cite{MorimotoSoudry2020}).
\end{itemize}
One then expects that these approaches give essentially the same (possibly up to an exponential) gamma factors. 
The main theme of this article is to verify (resp. partially verify) this expectation when $F$ is non-archimedean (resp. 
archimedean). More concretely, let us denote by $\gamma^*(s,\pi\x\tau,\psi)$ the gamma factors defined by the associated 
approach, where $*={\rm WD}, {\rm LS}$ or ${\rm RS}$. Then our first result is following.

\begin{thm}\label{T:main}
We have
\[
\gamma^{{\rm LS}}(s,\pi\x\tau,\psi)
=
\gamma^{{\rm WD}}(s,\pi\x\tau,\psi).
\]
\end{thm}

We must point out that when $F$ is archimedean, or $E=F\,\x\,F$, or $F$ is non-archimedean and $\pi$, $\tau$ are 
unramified, \thmref{T:main} was already obtained by Shahidi (\cite{Shahidi1985}, \cite{Shahidi1990}). In fact, 
\thmref{T:main} is proved by combining the results in op. cit. and works of Mok (\cite{Mok2015}), 
Adrian-Henniart-Kaplan-Oi (\cite{AHKO}) together with standard arguments 
(cf. \cite{JiangSoudry2004}, \cite{CKPSS2004}, \cite{KK2005}). Our next result concerns the relation between 
$\gamma^{{\rm RS}}(s,\pi\x\tau,\psi)$ and $\gamma^{{\rm LS}}(s,\pi\x\tau,\psi)$; however, since we need to impose some 
assumptions when $F$ is archimedean, we only state our results when $F$ is non-archimedean in the following, in order 
to make the statements more concise. The results for archimedean local fields can be found in the main body of the article 
(see \propref{P:main'}).

\begin{thm}\label{T:main'}
Suppose that $F$ is non-archimedean. Then we have
\[
\gamma^{{\rm RS}}(s,\pi\x\tau,\psi)
=
\gamma^{{\rm LS}}(s,\pi\x\tau,\psi).
\]
\end{thm}

As usual, \thmref{T:main'} is a consequence of the multiplicativity of $\gamma^{{\rm RS}}(s,\pi\x\tau,\psi)$, which will present
in the next two theorems. To unify the statement,  if $E=F\,\x\, F$ and $\sigma=\sigma_1\boxtimes\sigma_2$ is an irreducible
generic representation of $\G_k(F)=\GL_k(F)\x\GL_k(F)$, we denote
\[
\gamma^{{\rm WD}}(s,\sigma\x\eta,\psi)
=
\gamma^{{\rm WD}}(s,\sigma_1\x\eta,\psi)
\gamma^{{\rm WD}}(s,\sigma_2\x\eta,\psi)
\]
where $\eta$ is an irreducible generic representation of $\GL_m(F)$. The following theorem proves the multiplicativity of 
the Rankin-Selberg gamma factors with respect to the "1st variables" when $F$ is non-archimedean. 

\begin{thm}\label{T:main''}
Suppose that $F$ is non-archimedean. Then if $\pi$ is the generic quotient of the 
representation of $\U_{2n+1}(F)$, which is parabolically induced from an irreducible generic representation 
$\sigma\boxtimes\pi_0$ of $\G_k(F)\x\U_{2n_0+1}(F)$ for some integers $k\ge 1$ and $n_0\ge 0$ such that $k+n_0=n$, 
then 
\[
\gamma^{{\rm RS}}(s,\pi\x\tau,\psi)
=
\gamma^{{\rm RS}}(s,\pi_0\x\tau,\psi)
\gamma^{{\rm WD}}(s,\sigma\x\tau_1,\psi)
\gamma^{{\rm WD}}(s,\tilde{\sigma}\x\tau_2,\psi)
\]
where $\tilde{\sigma}$ is the contragredient of $\sigma$; $\tau_1$, $\tau_2$ are irreducible generic representations of 
$\GL_r(F)$ such that $\tau=\tau_1\boxtimes\tau_2$ when $E=F\,\x\, F$, and we understand that $\tau=\tau_1=\tau_2$ when 
$E$ is a field.
\end{thm}

The gamma factors $\gamma^{{\rm WD}}(s,\sigma\x\tau,\psi)$ and $\gamma^{{\rm WD}}(s,\tilde{\sigma}\x\tau,\psi)$
in \thmref{T:main''} are the ones defined by the associated Weil-Deligne representations through the local Langlands 
correspondence for the $F$-groups $\G_k$ and $\G_r$ (\cite{Langlands1989}, \cite{HarrisTaylor2001}, \cite{Henniart2000}, 
\cite{Scholze2013}), which are the same with those defined by the Langlands-Shahidi method (\cite{Shahidi1981},
\cite{Shahidi1985}). Up to powers of Langlands' $\lambda$-function attached to $E/F$ and $\psi$, they also equal to those 
defined by the Rankin-Selberg integrals (\cite{JPSS1983}, \cite{Jacquet2009}).\\
 
Our least main result verifies the multiplicativity of the Rankin-Selberg gamma factors with respect to the "2nd variables."

\begin{thm}\label{T:main'''}
Suppose that $F$ is non-archimedean. Then if $\tau$ is the generic quotient of the 
induced representation of $\G_r(F)$, which is parabolically induced from an irreducible generic representation 
$\tau'\boxtimes\tau''$ of $\G_{r'}(F)\,\x\,\G_{r''}(F)$ for some integers $r'>0$ and $r''>0$ such that $r'+r''=r$, then
\[
\gamma^{{\rm RS}}(s,\pi\x\tau,\psi)
=
\gamma^{{\rm RS}}(s,\pi\x\tau',\psi)
\gamma^{{\rm RS}}(s,\pi\x\tau'',\psi).
\]
\end{thm}


As mentioned, all of these results are expected. In fact, the theory of Rankin-Selberg convolutions for $G\,\x\GL_r$ with
$G$ a classical group has long and rich history, starting from the classical works of Rankin and Selberg. Their results (for 
$G=\GL_2$ and $r=2$) were later reformulated by Jacquet in the representation-theoretic framework in \cite{JLBook2}. 
Together with Piatetski-Shapiro and Shalika, such a theory was extended to $\GL_n\x\GL_r$ for any $n, r$ in 
\cite{JPSS1983}, \cite{Jacquet2009}. In \cite[Part B]{GPSR1987}, Gelbart and Piatetski-Shapiro developed the theory when 
$G$ is a split classical of type $B_r$, $C_r$ or $D_r$, which generalized low rank results in the literature. Since then, their 
constructions were extended by many authors in various settings including (to mention a few) Ginzburg 
(\cite{Ginzburg1990}), Tamir (\cite{Tamir1991}), Soudry (\cite{Soudry1993}, \cite{Soudry1995}, \cite{Soudry2000}), 
Watanabe (\cite{Watanabe2000}), Ben-Artzi--Soudry (\cite{Ben-ArtziSoudry2009}), Kaplan (\cite{Kaplan2010}, 
\cite{Kaplan2013b}, \cite{Kaplan2015}), Morimoto-Soudry (\cite{MorimotoSoudry2020}) and Morimoto (\cite{Morimoto}). 
Similar results to ours for other classical groups have been established in the aforementioned references. Furthermore, it is 
expected that the techniques will also be applicable to our setting. Our primary motivation for documenting these results, 
is rooted in our desire to eliminate the assumption made in \cite{YCheng2}. Another reason is that we hope 
this article, despite its incomplete results, can serve as a convenient reference.

\subsection{Outline of this article}
In \secref{S:gamma}, we will recall the definitions of the $\gamma$-factors via the associated Weil-Deligne representations 
as well as the Rankin-Selberg integrals. Then in \secref{S:proof of main}, we will prove \thmref{T:main}, and we will verify  
\thmref{T:main'} in \S\ref{S:proof of main'} assuming the validity of \thmref{T:main''}, \thmref{T:main'''}, minimal cases 
(i.e. $n=0$ and $r=1$), as well as the results for archimedean local fields. \thmref{T:main''} and \thmref{T:main'''} will be 
verified in \S\ref{S:1st} and \S\ref{S:2nd} respectively. In \S\ref{S:mini}, we will obtain some partial results when $F$ is 
archimedean, and establish \thmref{T:main'} when $n=0$ and $r=1$. Finally, in the Appendix, we will prove a multiplicity 
one result for certain Hom-space.

\subsection{Notation and conventions}

\subsubsection{Fields}
In this article, $F$ will always be a local field of characteristic zero, and $E$ an \etale $F$-algebra of rank $2$.  Thus $E$ is 
either a quadratic field extension of $F$ or $E=F\,\x\, F$. In the latter case, we always identify $F$ as a subfield of $E$ via 
the diagonal embedding. Denote by $\theta$ the generator of ${\rm Aut}_F(E)\cong\mathbb{Z}_2$. In particular, if $z\in E$, 
then both $z+\theta(z)$ and $z\theta(z)$ are contained in $F$.  Let $|\cdot|_F$ be the usual absolute on $F$, so that 
$|x|_\mathbb{R}={\rm max}\stt{x,-x}$,  $|x|_\mathbb{C}=x\bar{x}$ and when $F$ is non-archimedean,
it is characterized by $|\varpi|_F=q^{-1}$, where $\varpi\in F$ is a prime element and $q$ is the size of the residue field of 
$F$. The absolute value $|\cdot|_E$ on $E$ is then given by $|z|_E=|z\theta(z)|_F$. Throughout this article, we fix a 
non-trivial additive character $\psi$ of $F$ and define $\psi_E$ to be the non-trivial additive character of $E$ given by 
$\psi_E(x)=\psi(x+\theta(x))$ when $E$ is a filed. If $c\in F^\x$, we write $\psi_c$ for the character of $F$ defined by 
$\psi_c(x)=\psi(cx)$. We also fix an element $\delta\in E^\x$ with $\theta(\delta)=-\delta$ and put $\Delta=\delta^2\in F^\x$.
When $E=F\,\x\,F$, we simply take $\delta=(1,-1)$, so that $\Delta=1$.

\subsubsection{Matrices}
If $R$ is a ring, then ${\rm Mat}_{m, k}(R)$ will be the ring of $m$ by $k$ matrices with entries in $R$. If
$A\in{\rm Mat}_{m, k}(R)$, then we write $A_{ij}$ for the $(i,j)$-entry of $A$, and ${}^tA\in{\rm Mat}_{k, m}(R)$ for the 
transpose of $A$. The automorphism $\theta$ and its notation extends naturally to an involution on ${\rm Mat}_{m, k}(E)$, 
namely, we have $\theta(A)_{ij}=\theta(A_{ij})$ for every $A\in{\rm Mat}_{m, k}(E)$. When $E=F\,\x\, F$, we identify 
$\M_{m,k}(F)$ as a subring of $\M_{m,k}(E)=\M_{m,k}(F)\,\x\,\M_{m,k}(F)$ through the diagonal embedding.

\subsubsection{Groups}
Let $N$ be a positive integer. Denote by $\G_N={\Res}_{E/F}\GL_N$ the Weil restriction of $\GL_N$ from $E$ to $F$, so that 
$\G_N(F)=\GL_N(E)$ if $E$ is a field, whereas $\G_N(F)=\GL_N(F)\,\x\,\GL_N(F)$ if $E=F\,\x\,F$. 
Denote by $\U_N\subset\G_N$ a quasi-split unitary group over $F$ of $N$-variables When $E=F\,\x\, F$, we always identify 
$\U_N(F)$ with $\GL_N(F)$ through the projection onto the first component. Let $G$ be a connected reductive algebraic 
group over $F$, and $P\subset G$ be a parabolic subgroup. We denote by $\delta_P$ the modulus character of $P$. 
When $G$ is quasi-split, we usually need to fix a Borel subgroup of $G$ together with its Levi decomposition; such a Borel 
subgroup will be denoted by $B_G$ with the Levi decomposition $B_G=T_G\ltimes V_G$, where $T_G$ is a maximal torus 
in $B_G$ and $V_G$ is the unipotent of $B_G$. Then as usual, a parabolic subgroup $P$ of $G$ is said to be standard if 
$P\supseteq B_G$; we will always write $P=M_P\ltimes N_P$ for its unique Levi decomposition with $M_P\supseteq T_G$ 
and $N_P\subseteq V_G$. 

\subsubsection{Representations}
Let $G$ be a connected reductive algebraic group over $F$. In this article, by a representation of $G(F)$ we will 
always means a $smooth$ $complex$ representation of $finite$ $length$. The smoothness has its usual 
meaning when $F$ is non-archimedean, namely, every element has an open stabilizer. When $F$ is archimedean, it means 
a smooth admissible Fr\'echet representation of moderate growth in the sense of Casselman-Wallach 
(\cite{Casselman1989}). Let $\pi$ be a representation of $G(F)$, we usually denote $\cV_\pi$ as its (abstract) underlying 
space and $\omega_\pi$ as its central character (if exists). We write $\tilde{\pi}$ for the contragredient representation of 
$\pi$; when $F$ is archimedean, it can be defined as the Casselman-Wallach globalization of the contragredient of the 
Harish-Chandra module underlying $\pi$. If $G'$ is another group over $F$ and $\pi'$ is a representation of $G'(F)$, then 
$\pi\boxtimes\pi'$ stands for the usual external tensor product representation of $G(F)\,\x\,G'(F)$ on $\cV_\pi\ot\cV_{\pi'}$ 
when $F$ is non-archimedean; while when $F$ is archimedean, it represents the representation of $G(F)\,\x\,G'(F)$ on the
completed projective tensor product $\cV_\pi\widehat{\ot}\cV_{\pi'}$, which is actually the Casselman-Wallach completion of 
the tensor product of the Harish-Chendra modules underlying $\pi$ and $\pi'$. If $G=\GL_N$,  $\pi$ is a representation 
of $\GL_N(F)$ and $\mu$ is a character of $F^\x$, 
we denote by $\pi\ot\mu$ the representation of $\GL_N(F)$ whose underlying space is $\cV_\pi$ with the 
action $(\pi\ot\mu)(a)=\pi(a)\mu(\det(a))$. Similar convention applys with $F$ replaced by $E$.

\subsubsection{Generic representations}
Let $G$ be a quasi-split, connected reductive algebraic group over $F$, $V$ a maximal unipotent subgroup of $G$,
$\chi$ a non-degenerate character of $V(F)$ and $\pi$ a representation of $G(F)$. We say the $\pi$ is $generic$
(with respect to $\chi$) if the space ${\rm Hom}_{V(F)}(\pi,\chi)$ of Whittaker functionals is non-zero. When $F$ is
archimedean, we further require that Whittaker functionals are continuous. When $\pi$ is irreducible, it's known that the 
space of Whittaker functionals is at most one-dimensional. Suppose that $\pi$ is generic and admits a unique (up to scalars)
non-zero Whittaker functional $\lambda$, then space of complex-valued functions $W_v(g):=\lambda(\pi(g)v)$ on $G(F)$ 
for $v\in\cV_\pi$ is denoted by $\cW(\pi,\chi)$, which gives rise to a representation of $G(F)$ via right-translation; it becomes 
a quotient of $\pi$ unless $\pi$ is irreducible, in which case it is the usual Whittaker model of $\pi$.

\subsubsection{Induced representations}
Let $G$ be a connected reductive algebraic group over $F$, and $P\subset G$ be a parabolic subgroup with a Levi 
decomposition $P=M_P\ltimes N_P$. Let $\sigma$ be a representation of $M_P(F)$, which can be regarded 
as a representation of $P(F)$ via the natural projection $P(F)\twoheadrightarrow P(F)/N_P(F)\cong M_P(F)$. Denote by
$I_P^G(\sigma)={\rm Ind}_{P(F)}^{G(F)}(\sigma)$ the normalized induced representation of $G(F)$ inducing from $\sigma$.
Its underlying space $\cV_P^G(\sigma)$ consists of smooth functions $\xi:G(F)\to\cV_\sigma$ satisfying 
$\xi(pg)=\delta_P^{1/2}(p)\sigma(p)\xi(g)$ for $p\in P(F)$ and $g\in G(F)$. The group $G(F)$ acts on $\cV_P^G(\sigma)$
via the right translation $\rho$.

\section{Gamma factors via two approaches}\label{S:gamma}
In this section, we recall the definitions of the gamma factors through the approaches of the associated Weil-Deligne 
representations and the Rankin-Selberg integrals. 

\subsection{Gamma factors via Weil-Deligne representations}
\subsubsection{The $L$-groups}
Let $WD_F$ be the Weil-Delgine group of $F$, so that $WD_F=W_F$ if $F$ is archimedean whereas 
$WD_F=W_F\,\x\,{\rm SL}_2(\mathbb{C})$ if $F$ is non-archimedean, where $W_F$ is the Weil group of $F$. As usual, if 
$G$ is a quasi-split, connected reductive group over $F$, the $L$-group of $G$ is a semi-direct product 
\[
{}^LG=\widehat{G}\rtimes{\rm Gal}(K/F)
\]
where $\widehat{G}$ is the complex dual group and $K$ is a splitting field for $G$, with ${\rm Gal}(K/F)$ acting on 
$\widehat{G}$ via pinned automorphism. In our setting, we have $\widehat{\U}_N=\GL(\cV_N)$, 
$\widehat{\G}_N=\GL(\cV_N)\x\GL(\cV_N)$ and $K=E$ or $F$ according to $E$ is a field or not; here $\cV_N$ is the 
$N$-dimensional complex vector space. 

\subsubsection{$E$ is a field}
Suppose that $E$ is a field. Denote by $g\mapsto g^\theta$ the involution on $\widehat{\U}_N$ or $\widehat{\G}_N$ induced 
by $\theta$; we then have a natural $L$-embedding 
\[
{}^L\U_N\hookto{}^L\G_N;\quad g\rtimes c\mapsto (g,g^\theta)\rtimes c.
\]
Let $M\ge 1$ be another integer; the $L$-group of $\U_N\,\x\,\G_M$ is 
\[
(\widehat{\U}_N\,\x\,\widehat{\G}_M)\rtimes{\rm Gal}(E/F)
\]
with $\theta$ acting on the components of $\widehat{\U}_N\,\x\,\widehat{\G}_M$. One can replace $\U_N$ with $\G_N$ to 
obtain the $L$-group of $\G_N\,\x\,\G_M$; then the aforementioned $L$-embedding induces an embedding 
\[
\iota=\iota_{N,M}:{}^L(\U_N\,\x\,\G_M)\hookto{}^L(\G_N\,\x\,\G_M)
\]
in an obvious way. Define a representation $r=r_{N,M}:{}^L(\G_N\,\x\,\G_M)\to\GL(\cV_N\otimes\cV_M\oplus\cV_N\ot\cV_M)$ 
by the rules 
\[
r((g_1,g_2; h_1, h_2)\rtimes 1)=\pMX{g_1\ot h_1}{}{}{g_2\ot h_2};\quad
r((g_1,g_2; h_1, h_2)\rtimes\theta)=\pMX{}{g_1\ot h_1}{g_2\ot h_2}{}
\]
for $g_1,g_2\in\GL(\cV_N)$ and $h_1, h_2\in\GL(\cV_M)$.\\

Let $WD_F\twoheadrightarrow W_F\twoheadrightarrow W_F/W_E\cong{\rm Gal}(E/F)$ be the natural projection, which 
sending $w\in WD_F$ to $w'\in{\rm Gal}(E/F)$. Let $\phi:WD_F\to{}^L\U_N$ and $\eta:WD_F\to{}^L\G_M$ be two
$L$-homomorphisms; they can be written as  $\phi(w)=\hat{\phi}(w)\rtimes w'$ and $\eta(w)=\hat{\eta}(w)\rtimes w'$ for 
$w\in WD_F$, where $\hat{\phi}(w)$ and $\hat{\eta}(w)$ are contained in $\widehat{\U}_N$ and $\widehat{\G}_M$, 
respectively. Define the $L$-homomorphism
\[
\phi\ot\eta: WD_F\to{}^L(\U_N\,\x\,\G_M);\quad
w\mapsto (\hat{\phi}(w);\hat{\eta}(w))\rtimes w'
\]
which, after composing with $\iota$, gives rise to an $L$-homomorphism $\iota(\phi\ot\eta):WD_F\to{}^L(\G_N\,\x\,\G_M)$.
Suppose that $\pi$ (resp. $\tau$) is an irreducible representations of $\U_N(F)$ (resp. $\G_M(F)$) whose $L$-parameter is
$\phi$ (resp. $\eta$), and that $\psi$ is a non-trivial additive character of $F$. Then we have
\[
\gamma^{{\rm WD}}(s,\pi\x\tau,\psi)
=
\gamma(s,\iota(\phi\ot\eta), r,\psi).
\]
The definition of $\gamma^{{\rm WD}}(s,\pi\x\tau,\psi)$ is a little bit complicated; nevertheless, it can be described easily 
through the standard base change lift of $\pi$. More precisely, let $\phi_E$ denote the restriction of 
$\phi$ to $WD_E$, and $\eta_E$ denote the restriction of $\eta$ to $WD_E$, and then compose with the natural projection 
$\widehat{\G}_M\twoheadrightarrow\GL(\cV_M)$ onto its first component. In particular, $\phi_E$ (reps. $\eta_E$) becomes 
an $N$-dimensional (resp. $M$-dimensional) representation of $WD_E$, so that we can form their tensor product 
$\phi_E\ot\eta_E$. Then we have following lemma, which is not hard to verify.

\begin{lm}\label{L:L-isom}
We have
\[
{\rm Ind}_{WD_E}^{WD_F}(\phi_E\ot\eta_E)\cong r\circ\iota(\phi\ot\eta).
\]
\end{lm}

Now let $\pi$ be an irreducible generic representation of $\U_N(F)$ with the $L$-parameter $\phi$, 
${\rm BC}(\pi)$ be the standard base change of $\pi$, which is an irreducible representation of $\GL_N(E)$ with 
the $L$-parameter $\phi_E$, and $\tau$ be an irreducible generic representation of $\G_M(F)=\GL_M(E)$. 
Then the above lemma immediately yields 
\[
\gamma^{{\rm WD}}(s,\pi\x\tau,\psi)
=
\lambda_{E/F}(\psi)^{NM}\gamma^{{\rm WD}}(s,{\rm BC}(\pi)\x\tau,\psi_E)
\]
where $\lambda_{E/F}(\psi)$ is Langlands' $\lambda$-function (cf. \cite[Lemma 1.2]{JLbook}, \cite{Rohrlich1994}).

\subsubsection{$E=F\,\x\, F$}
Let us briefly explain the definition of $\gamma^{{\rm WD}}(s,\pi,\tau,\psi)$ when $E=F\,\x\,F$. In this setting, we have
$\U_N(F)=\GL_N(F)$ and $\G_N(F)=\GL_N(F)\x\GL_N(F)$, so that 
\[
{}^L\U_N=\widehat{\U}_N=\GL(\cV_N)
\quad\text{and}\quad
{}^L\G_N=\widehat{\G}_N=\GL(\cV_N)\x\GL(\cV_N).
\]
The embedding ${}^L\U_N\hookto{}^L\G_N$ is given by $g\mapsto(g, {}^tg^{-1})$, which induces an embedding 
\[
\iota:{}^L(\U_N\,\x\,\G_M)
=\GL(\cV_N)\x\GL(\cV_M)\x\GL(\cV_M)\hookto{}^L(\G_N\,\x\,\G_M)
=\GL(\cV_N)\x\GL(\cV_N)\x\GL(\cV_M)\x\GL(\cV_M).
\]
On the other hand, the representation 
\[
r:{}^L(\G_N\,\x\,\G_M)\to\GL(\cV_N\ot\cV_M\oplus\cV_N\ot\cV_M)
\]
in this case is simply defined by 
\[
r(g_1,g_2; h_1, h_2)=\pMX{g_1\ot h_1}{}{}{g_2\ot h_2}
\]
for $g_1,g_2\in\GL(\cV_N)$ and $h_1, h_2\in\GL(\cV_M)$.
Suppose that $\pi$ (resp. $\tau=\tau_1\boxtimes\tau_2$) is an irreducible representations of $\GL_N(F)$ (resp. 
$\GL_M(F)\x\GL_M(F)$) whose $L$-parameter is $\phi$ (resp. $\eta$). Then we put
\[
\gamma^{{\rm WD}}(s,\pi\x\tau,\psi)
=
\gamma(s,\iota(\phi\ot\eta), r,\psi).
\]
One checks easily that 
\[
\gamma^{{\rm WD}}(s,\pi\x\tau,\psi)
=
\gamma^{{\rm WD}}(s,\pi\x\tau_1,\psi)
\gamma^{{\rm WD}}(s,\tilde{\pi}\x\tau_2,\psi).
\]

\subsection{Gamma factor via Rankin-Selberg integrals}

\subsubsection{Unitary groups and embeddings}\label{SSS:embedding}
To define the Rankin-Selberg integrals, it will be more convenient to have an explicit realization for the unitary groups $\U_N$.
For this, let $n=\lfloor\frac{N}{2}\rfloor$, and $S_N\in\GL_N(F)$ be the matrix given by $S_N=J_N$ if $N$ is even, and 
\[
S_N
=
\begin{pmatrix}
&&J_n\\
&-2&\\
J_n
\end{pmatrix}
\]
if $N$ is odd, where for an integer $k\ge 1$, the matrix $J_k\in\GL_k(F)$ is defined inductively by 
\[
J_1=(1)
\quad\text{and}\quad
J_k
=
\begin{pmatrix}
&J_{k-1}\\
1&
\end{pmatrix}.
\]
Then $\U_N$ is realized as
\[
\U_N(F)
=
\stt{g\in\G_N(F)\mid {}^t\theta(g)S_N g=S_N}
\]
and we have two embeddings
\[
\jmath_{n,r}:\U_{2r}(F)\hookto\U_{2n+1}(F)\quad\text{and}\quad
\jmath^{n,r}:\U_{2n+1}(F)\hookto\U_{2r}(F)
\]
depending on $n\ge r$ and $n<r$, respectively. The first one is easier to describe, and is given by  
\[
\jmath_{n,r}\left(\pMX{a}{b}{c}{d}\right)
=
\begin{pmatrix}
a&&b\\
&I_{2(n-r)+1}\\
c&&d
\end{pmatrix}
\]
with $a,b,c,d\in{\rm Mat}_{r\x r}(K)$, where $K=E$ or $F$ according to $E$ is a field or not. Suppose that $n<r$ and put
\[
e={}^t(\underbrace{0,\cdots,0}_{n}, 1,\underbrace{0,\cdots, 0}_{2(r-n-1)},1,\underbrace{0,\cdots,0}_{n}).
\]
Then the subgroup of $\U_{2r}(F)$ consisting of stabilizers of $e$ is isomorphic to $\U_{2n+1}(F)$, and $\jmath^{n,r}$ is 
obtained via an isomorphism from $\U_{2n+1}(F)$ onto this subgroup. More specifically, if we put
\[
Q
=
\begin{pmatrix}
I_n&&\\
&1&&1\\
&&I_{2(r-n-1)}\\
&1&&-1\\
&&&&I_n
\end{pmatrix}\in\GL_{2r}(F)
\]
then 
\[
\jmath^{n,r}
\left(\pMX{a}{b}{c}{d}\right)
=
Q
\begin{pmatrix}
a&&&b\\
&1&&0\\
&&I_{2(r-n-1)}\\
c&0&&d
\end{pmatrix}
Q^{-1}
\]
with $a\in{\rm Mat}_{n\x n}(K)$ and $d\in{\rm Mat}_{n+1\x n+1}(K)$, where again, $K=E$ or $F$ depending on $E$ is 
a field or not. 

\subsubsection{Non-degenerate characters}
Let $G=\GL_N, \U_N$ or $\G_N$, and $\psi=\psi_F$ be a non-trivial additive character of $F$. Denote by $B_G\subset G$ 
the Borel subgroup consisting of upper triangular matrices, with the Levi decomposition $B_G=T_G\ltimes V_G$, where 
$T_G\subset B_G$ is the maximal tours whose elements are diagonal matrices. They can be viewed as algebraic groups 
defined over $F$; in particular, we have $V_{\G_N}(F)=V_{\GL_N}(E)$. Define the non-degenerate characters 
$\psi_{\GL_N}$, $\psi_{\U_{2n+1}}$ and $\psi_{\G_r}$ of $V_{\GL_N}(K)$, $V_{\U_{2n+1}}(F)$ and $V_{\G_r}(F)$, 
respectively, where $K=E$ or $F$ is a field, as follows. First, we define
\[
\psi_{\GL_N}(z)=\psi_K(z_{12}+z_{23}+\dots+z_{N-1,N}).
\]
Next assume that $E$ is a field, so that $\G_r(F)=\GL_r(E)$. Then we define
\[
\psi_{\U_{2n+1}}(z)
=
\psi_E(z_{12}+z_{23}+\dots+z_{n-1,n}+2^{-1}z_{n,n+1})
\quad\text{and}\quad
\psi_{\G_r}
=
\psi_{\GL_r}
\]
Finally, suppose that $E=F\,\x\,F$, so that $V_{\U_{2n+1}}(F)=V_{\GL_{2n+1}}(F)$ and 
$V_{\G_r}(F)=V_{\GL_r}(F)\,\x\,V_{\GL_r}(F)$. In this case, we define
\[
\psi_{\U_{2n+1}}(z)
=
\psi(z_{12}+z_{23}+\dots+z_{n-1,n}+2^{-1}z_{n,n+1}+z_{n+1,n+2}-z_{n+2,n+3}-\dots-z_{2n,2n+1})
\]
and
\[
\psi_{\G_r}(z, z')
=
\psi_{{\GL_r}}(z)\psi^{-1}_{{\GL_r}}(z').
\]
We emphasis that when $E=F\,\x\, F$, the characters $\psi_{\U_{2n+1}}$ and $\psi_{\GL_{2n+1}}$ are different, although the
groups $\U_{2n+1}(F)$ and $\GL_{2n+1}(F)$ are isomorphic. We hope this will not cause a serious confusion.
We also point out that these non-degenerate characters are unique up to conjugations of torus elements. 

\subsubsection{Induced representations of $\U_{2r}$}\label{SSS:rho_tau}
Here and after, we denote $a^*=J_N{}^t\theta(a)^{-1}J_N$ if $a\in\GL_N(E)$, and $a^*=J_N{}^ta^{-1}J_N$ if 
$a\in\GL_N(F)$. If $\eta$ is a representation of $\GL_N(E)$ or $\GL_N(F)$, then $\eta^*$ denotes the representation 
of $\GL_N(E)$ or $\GL_N(F)$ on the same space $\cV_\eta$ with the action $\eta^*(a)=\eta(a^*)$. Moreover, if $s$ is a 
complex number, then we write $\eta_s$ for $\eta\ot|\det|_K^{s-1/2}$, where $K=E$ or $F$. We understand that 
$\eta^*_s=\eta^*\ot|\det|^{s-1/2}_K$.\\

Let $Q_{2r}$ be the standard parabolic subgroup of $\U_{2r}$ whose Levi subgroup $M_{Q_{2r}}\cong\G_r$. 
Explicitly, we have 
\[
M_{Q_{2r}}(F)
=
\stt{m(a)=\pMX{a}{}{}{a^*}\mid a\in\G_r(F)}
\]
if $E$ is a filed, where $a^*=J_r{}^t\theta(a)^{-1}J_r$, and 
\[
M_{Q_{2r}}(F)
=
\stt{m(a,b)=\pMX{a}{}{}{b}\mid (a,b)\in\G_r(F)}
\]
if $E=F\,\x\,F$. Let $\tau$ be an irreducible generic representation of $\G_r(F)$. When $E=F\,\x\,F$, $\tau$ is of the form 
$\tau_1\boxtimes\tau_2$, where $\tau_1,\tau_2$ are irreducible generic representations of $\GL_r(F)$. In this case, we 
define $\tau_s=\tau_{1,s}\boxtimes\tau^*_{2,1-s}$. The induced representations appear in the definition of the Rankin-Selberg 
integrals are 
\[
\rho_{\tau,s}=I_{Q_{2r}}^{\U_{2r}}(\tau_s)={\rm Ind}_{Q_{2r}(F)}^{\U_{2r}(F)}(\tau_s).
\]
To define the Rankin-Selberg integrals, we actually need scalar-valued functions. For this, let $\lambda_{\tau,\psi^{-1}}$ be a 
non-zero Whittaker functional of $\tau$ with respect to the generic character $\psi^{-1}_{\G_r}$. Then for a given 
$\xi_s\in \cV_{Q_{2r}}^{\U_{2r}}(\tau_s)$, we put
\begin{equation}\label{E:f_xi}
f_{\xi_s}(h;a)=\lambda_{\tau,\psi^{-1}}(\tau(a)\xi_s(h))
\end{equation}
where $a\in\G_r(F)$ and $h\in\U_{2r}(F)$. The function $f_{\xi_s}:\U_{2r}(F)\,\x\,\G_r(F)\to\mathbb{C}$ is smooth, and for 
each fixed $h$ in $\U_{2r}(F)$, the assignment $a\mapsto f_{\xi_s}(h;a)$ gives rise to an element in 
$\cW(\tau,\psi^{-1}_{\G_r})$. 

\subsubsection{Intertwining maps}\label{SSS:intertwining map}
We introduce an intertwining map on $\cV_{Q_{2r}}^{\U_{2r}}(\tau_s)$, which is necessary in defining the gamma factors 
through the Rankin-Selberg integrals. To do so, let 
\[
w_{r,r}
=
\pMX{}{I_r}{I_r}{}\in\U_{2r}(F).
\]
Then we have the intertwining map 
$A(w_{r,r},\tau,s):\cV_{Q_{2r}}^{\U_{2r}}(\tau_s)\to \cV_{Q_{2r}}^{\U_{2r}}(\tau^*_{1-s})$ given by the integral
\[
A(w_{r,r},\tau,s)\xi_s(h)
=
|\Delta|_F^{\frac{r}{2}}\int_{N_{Q_{2r}}(F)}\xi_s(w_{r,r}^{-1}uh)du
\]
when $\Re(s)\gg 0$, and by the meromorphic continuation of the integral in general, where we define
$\tau^*=\tau_2^*\boxtimes\tau_1^*$ when $E=F\,\x\,F$. 
We in fact need a normalized version $A_{\psi,\delta}(w_{r,r},\tau,s)$ of $A(w_{r,r},\tau,s)$. 
By the works in \cite{Shahidi1981}, 
\cite{Shahidi1985}, \cite{Casselman1989} and \cite{CHM2000}, it can be defined to satisfy the following identity
\[
\int_{N_{Q_{2r}}(F)}f_{\xi_s}(w_{r,r}uh; d_r)\psi'^{-1}(u_{r,r+1})du
=
\int_{N_{Q_{2r}}(F)}f_{A_{\psi,\delta}(w_{r,r},\tau,s)\xi_s}(w_{r,r}uh; I'_r)\psi'^{-1}(u_{r,r+1})du.
\]
Here $d_r\in\G_r(F)$ is the matrix given inductively via the formulas
\begin{equation}\label{E:d_r}
I'_1=(1)
\quad\text{and}\quad
I'_r
=
\pMX{d_{r-1}}{}{}{(-1)^{r-1}}
\end{equation}
and $\psi'(u_{r, r+1})=\psi_E(\delta u_{r,r+1})$ if $E$ is a field, whereas 
$\psi'(u_{r,r+1})=\psi(2u_{r,r+1})$ if $E=F\,\x\, F$.\\

The Haar measure $du$ on $N_{Q_{2r}}(F)$ is defined as a product measure, where each root group is isomorphic to either 
the fields $E$ or $F$, and we take the self-dual measure on them with respect to $\psi_E$ or $\psi'$ accordingly. More 
precisely, to identify those root groups in $N_{Q_{2r}}(F)$ which are isomorphic to $F$ (there are $r$ of them) when $E$ is 
a field, we must fix an element $\delta\in E^\x$ with $\theta(\delta)=-\delta$; the factor $|\Delta|_F^{\frac{r}{2}}$ is included so 
that the integral is independent of the choice of $\delta$.

\subsubsection{Rankin-Selberg integrals: setups}
We introduce the setups in defining the Rankin-Selberg integral for $\U_{2n+1}\,\x\,\G_r$. Let $\pi$ and $\tau$ be irreducible 
generic representations of $\U_{2n+1}(F)$ and $\G_r(F)$, respectively. Let $\psi$ be a non-trivial additive character of $F$.
Fix a non-zero Whittaker functional $\lambda_{\pi,\psi}$ (resp. $\lambda_{\tau,\bar{\psi}}$) of $\pi$ (resp. $\tau$) with respect 
to the non-degenerate character $\psi_{\U_{2n+1}}$ (resp. $\bar{\psi}_{{\G_r}}$). We also need certain unipotent subgroups 
in $\U_{2n+1}$ or $\U_{2r}$ according to the sizes of $n$ and $r$. When $n\ge r$, such a unipotent subgroup is denoted by 
$\b{X}_{n,r}$; on the other hand, it will be denoted by $\b{X}^{n,r}$ if $n<r$.  More concretely, let's write $\ell=r-n-1$ when 
$n<r$. Then they are given by
\[
\b{X}_{n,r}(F)
=
\stt{
\begin{pmatrix}
I_r&&&&\\
A&I_{n-r}&&\\
&&1\\
&&&I_{n-r}&\\
&&&B&I_r
\end{pmatrix}
\in\U_{2n+1}(F)
\mid
A\in{\rm Mat}_{n-r,r}(K),\,B\in{\rm Mat}_{r,n-r}(K)
}
\]
and
\[
\b{X}^{n,r}(F)
=
\stt{
\bar{u}(A,B,C)
=
\begin{pmatrix}
I_{n+1}&&&\\
&I_{\ell}&&\\
A&B&I_{\ell}\\
&C&&I_{n+1}
\end{pmatrix}
\in\U_{2r}(F)
\mid
A\in{\rm Mat}_{\ell, n+1}(K),\, B\in{\rm Mat}_{\ell,\ell}(K),\, C\in{\rm Mat}_{n+1,\ell}(K)
}
\]
where $K=E$ or $F$ according to $E$ is a field or not. Notice that when $E$ is a field, $B=-J_r{}^t\theta(A) J_{n-r}$ in the 
definition of $\b{X}_{n,r}(F)$, while $C=-J_{\ell}{}^t\theta(A)J_{n+1}$ and $B\in{\rm Mat}_{\ell\x\ell}(E)$ is such that 
${}^t\theta(B)=-J_{\ell}BJ_{\ell}$ in the definition of $\b{X}^{n,r}(F)$. Therefore in this case, we may write 
$\b{u}(A,B,C)=\b{u}(A,B)$.
Define a character $\psi_{\b{X}^{n,r}}$ of $\b{X}^{n,r}(F)$ by 
\[
\psi_{\b{X}^{n,r}}(\b{u}(A,C))
=
\psi_E(A_{n+1,\ell})
\quad
\text{or}
\quad
\psi_{\b{X}^{n,r}}(\b{u}(A,B,C))
=
\psi(A_{n+1,\ell}-C_{11})
\]
depending on $E$ is a field or not.
\subsubsection{Rankin-Selberg integrals}\label{SSS:RS}
The Rankin-Selberg integral $\Psi_{n,r}(v\ot\xi_s)$ attached to $v\in\cV_\pi$ and $\xi_s\in I_r(\tau,s)$ is given by 
\begin{equation}\label{E:RS int n>r}
\Psi_{n,r}(v\ot\xi_s)
=
\int_{N_{\U_{2r}}(F)\backslash\U_{2r}(F)}\int_{\bar{X}_{n,r}(F)}
W_v(\b{u}\jmath_{n,r}(h))f_{\xi_s}(h; I_r)d\b{u}dh
\end{equation}
if $n\ge r$, and 
\begin{equation}\label{E:RS int n<r}
\Psi_{n,r}(v\ot\xi_s)
=
\int_{N_{\U_{2n+1}}(F)\backslash\U_{2n+1}(F)}\int_{\b{X}^{n,r}(F)}
W_v(g)f_{\xi_s}(\b{u}\jmath^{n,r}(g); I_r)\psi^{-1}_{\b{X}^{n,r}}(\b{u})d\b{u}dg
\end{equation}
if $n<r$.
By the results of \cite{Ben-ArtziSoudry2009} (for $n\ge r$) and \cite{MorimotoSoudry2020} (for $n<r$), these integral 
converge absolutely for $\Re(s)\gg 0$, and there exist $v$ and $\xi_s$ such that $\Psi_{n,r}(v\ot\xi_s)\equiv 1$ when $F$ is 
non-archimedean, and such that $\Psi_{n,r}(v\ot\xi_s)$ is holomorphic and is non-zero at a given $s_0$ when $F$ is 
archimedean. Moreover, except for $n<r$ and $F$ is archimedean, these integrals admit the meromorphic continuation to 
the whole complex plane. When $F$ is non-archimedean, they even give rise to elements in $\mathbb{C}(q^{-s})$.

\subsubsection{Functional equations and gamma factors}
Now we can describe the functional equations for the Rankin-Selberg integrals as well as the associated gamma factors.
Before we do so, notice that $\rho_{\tau,s}$ is irreducible except for countable many $s$. Assume first that $F$ is 
non-archimedean. Then the equivariant property (cf. \eqref{E:equiv n>r}, \eqref{E:equiv n<r}) and the fact that 
$\Psi_{n,r}(v\ot\xi_s)$ admits the meromorphic continuation to whole the complex plane imply that the Rankin-Selberg 
integrals give rise to elements in certain Hom-spaces (cf. \eqref{E:hom-space n>r}, \eqref{E:hom-space n<r}) (with $\rho$ 
being replaced with $\rho_{\tau,s}$), except for countable many $s$. Since the integrals
$\Psi_{n,r}(v\ot A_{\psi,\delta}(w_r,\tau,s)\xi_s)$ process the similar properties and the Hom-spaces are at most 
one-dimensional, we conclude that there exists a rational function $\Gamma(s,\pi\x\tau,\psi)$ in $q^{-s}$, which depends 
only on $\pi$, $\tau$ and $\psi$, such that 
\begin{equation}\label{E:FE}
\Psi_{n,r}(v\ot A_{\psi,\delta}(w_r,\tau,s)\xi_s)
=
\Gamma_\delta(s,\pi\x\tau,\psi)
\Psi_{n,r}(v\ot\xi_s)
\end{equation}
for every $v\in\cV_\pi$ and $\xi_s\in I_r(\tau,s)$. Notice that $\Gamma_\delta(s,\pi\x\tau,\psi)$ is non-zero, as we can aways 
find $v$ and $\xi_s$ such that $\Psi_{n,r}(v\ot\xi_s)$ is non-vanishing. We then define
\begin{equation}\label{E:RS gamma}
\gamma^{{\rm RS}}(s,\pi\x\tau,\psi)
=
\omega_{\pi}(-1)^r\omega_{\tau}(-1)^n\omega_\tau(\delta)^{-r}|\delta|_E^{-r\left(s-\frac{1}{2}\right)}\omega_{E/F}(-1)^{nr}\Gamma_\delta(s,\pi\x\tau,\psi)
\end{equation}
where $\omega_{E/F}$ is the quadratic character (possibly trivial) of $F^\x$ associated to $E/F$ by the local class field 
theory, and $\omega_\tau=\omega_{\tau_1}\omega_{\tau_2}$, 
$\omega_{\tau}(\delta)=\omega_{\tau_1}(1)\omega_{\tau_2}(-1)=\omega_{\tau_2}(-1)$ if $E=F\,\x\, F$.\\

Now, assume that $F$ is archimedean. The same arguments as above would give the gamma factor attached to $\pi$, 
$\tau$, and $\psi$, as defined via the Rankin-Selberg integrals in the archimedean case. However, unlike the 
non-archimedean case, we need to ensure the continuity of the Rankin-Selberg integrals on $\cV_\pi\widehat{\ot}I_r(\tau,s)$. 
When $n\ge r-1$, we can apply the results and their proofs of Soudry in \cite{Soudry1995} to deduce the desired continuity 
condition. More specifically, if $n\ge r$, then we can apply the results in \cite[Section 4]{Soudry1995} and the proofs in 
\cite[Section 5]{Soudry1995} to show that our Rankin-Selberg integrals admit meromorphic continuations, which give rise to 
elements in the desired Hom-spaces (cf. \eqref{E:hom-space n>r}). When $n=r-1$, the unipotent subgroup $\b{X}^{n,r}$ is 
trivial, as in the case when $n=r$. Consequently, the aforementioned results and proofs are still adaptable to the case 
$n=r-1$, and hence we can define the gamma factors $\gamma^{{\rm RS}}(s,\pi\x\tau,\psi)$ through \eqref{E:FE} and 
\eqref{E:RS gamma} when $n\ge r-1$.\\
 
In contrast, it seems that the results and proofs in \cite{Soudry1995} are not readily applicable to the present settings when 
$n<r-1$. More concretely, we do not have analogous identities to \cite[(6.10), (6.11)]{Soudry1995}, which establish a duality 
between the Rankin-Selberg integrals for $\SO_{2n+1}\,\x\,\GL_r$ and $\SO_{2n}\,\x\,\GL_r$. This duality allowed 
Soudry to adapt the results and proofs for ${\rm SO}_{2n+1}\,\x\,\GL_r$ to $\SO_{2n}\,\x\,\GL_r$. It's worth noting that 
both integrals for $\SO_{2n+1}\x\GL_r$ and $\SO_{2n}\x\GL_r$ belong to the Gelfand-Graev type. A critical reason behind 
the establishment of such a duality, at least from our perspective, is that (split) $\SO_2$ is abelian. However, as $\U_2$ is 
non-abelian, and more importantly, the Rankin-Selberg integrals for $\U_{2n}\,\x\,\GL_r$ fall under the category of 
Fourier-Jacobi type, we cannot expect that similar identities apply to our settings.
For a more in-depth understanding of the issue we've indicated, one could compare the identity \cite[(6.10)]{Soudry1995} with 
the one obtained in \cite[Theorem 7.2]{MorimotoSoudry2020}, even though the latter is only established for unramified 
representations (and, thus, in the context of non-archimedean local fields).
In conclusion, we cannot affirm whether these integrals possess a meromorphic continuation when $n<r-1$. Consequently, 
we are unable to define the gamma factors in this particular case.

\subsubsection{Remarks}
We end this section with a few remarks. First, our formulation of $\Psi_{n,r}(v\ot\xi_s)$ is slightly different from 
(but equivalent to) that of \cite{Ben-ArtziSoudry2009} and \cite{MorimotoSoudry2020} because we want to emphasis the 
resemblance between these integrals and those attached to generic representations of special orthogonal groups 
(cf. \cite{Soudry1993}). This is due to our intention to follow the arguments in op. cit. and \cite{Soudry2000}. 
Second, when $F$ is non-archimedean, the irreducible assumption can actually be relaxed, namely, the results for 
$\Psi_{n,r}(v\ot\xi_s)$ stated in \S\ref{SSS:RS} remain valid if we assume that (i) $\tau$ is of the form 
$\tau_1\boxtimes\tau_2$ for some representations $\tau_1,\tau_2$ of $\GL_r(F)$ when $E=F\,\x\, F$, and (ii) all the 
involved representations are of $Whittaker$ $type$. Here we call a representation of a quasi-split connected reductive group
over $F$ of Whittaker type if it admits a unique (up to scalar) Whittaker functional (with respect to some non-degenerate 
character). To prove the existence of the factor $\Gamma_\delta(s,\pi\x\tau,\psi)$ attached to these less restricted 
representations, we must show that the spaces \eqref{E:hom-space n>r} and \eqref{E:hom-space n<r}, with $\rho$ being 
replaced with $\rho_{\tau,s}$, are of one-dimensional except for countable many $s$. When $n<r$, this is proved by 
Morimoto-Soudry in \cite[Theorem 6.3]{MorimotoSoudry2020}; when $n\ge r$, this will be proved in the Appendix.

\section{Proof of \thmref{T:main}}\label{S:proof of main}

\subsection{Notaiton}\label{SS:notation}
For simplicity, we introduce the following notation.
Let $\tau_j$ be a representation of $\GL_{r_j}(K)$ for $j=1,\ldots, m$, where $K$ is a local 
field of characteristic zero. We denote by 
\[
I^{\GL_r(K)}(\tau_1,\cdots,\tau_m)
\]
the normalized induced representation of $\GL_{r_1+r_2+\dots+r_m}(K)$, which is induced from the representation 
$\tau_1\boxtimes\cdots\boxtimes\tau_m$ of the standard parabolic subgroup of 
$\GL_{r_1+r_2+\dots+r_m}(K)$, whose Levi subgroup is isomorphic to 
\[
\GL_{r_1}(K)\x\cdots\x\GL_{r_m}(K).
\]
Next, if $\pi$ is a representation of $\U_{2n+1}(F)$, and $\tau_j$ is a representations of 
$\G_{r_j}(F)$ for $j=1,\ldots, m$, with $\tau_j=\tau'_j\boxtimes\tau''_j$ when $E=F\,\x\, F$, where 
$\tau'_j$,  $\tau''_j$ are representations of $\GL_{r_j}(F)$, then
\[
I^{\U_{2(n+r)+1}(F)}(\tau_1,\dots,\tau_m;\pi)
\]
stands for the normalized induced representation of $\U_{2(n+r)+1}(F)$, which is induced from the 
representation $\tau_1\boxtimes\cdots\boxtimes\tau_m\boxtimes\pi_0$ of the standard parabolic subgroup of 
$\U_{2n+1}(F)$, whose Levi subgroup is isomorphic to 
\[
\G_{r_1}(F)\x\cdots\x\G_{r_m}(F)\x\U_{2n+1}(F)
\]
where $r=r_1+\cdots r_m$. When $E=F\,\x\, F$ so that $\U_{2n+1}(F)=\GL_{2n+1}(F)$ and 
$\G_{r_j}(F)=\GL_{r_j}(F)\,\x\,\GL_{r_j}(F)$ for $j=1,\dots, m$, we understand that  
\[
I^{\U_{2(n+r)+1}(F)}(\tau_1,\dots,\tau_m;\pi)
=
I^{\GL_{2(n+r)+1}}(\tau'_1,\dots,\tau'_m,\pi,\tau''_m,\dots,\tau''_1).
\]
We hope this will not cause a serious confusion.

\subsection{Preliminaries}
As usual, \thmref{T:main} and \thmref{T:main'} are proved by standard global arguments. For this, we record the following 
globalization result:

\begin{lm}\label{L:globalization}
Suppose that $F$ is non-archimedean and  $\pi$ is a given irreducible square-integrable generic representation 
of $\U_{2n+1}(F)$. Then there exist 
\begin{itemize}
\item a totally complex number field $\mathbb{F}$;
\item a quadratic extension field $\mathbb{E}$ of $\mathbb{F}$;
\item a finite place $v_0$ of $\mathbb{F}$;
\item a finite set $S$ of split spaces of $\bbF$ disjoint from $s_0$;
\item a quasi-split unitary group $\mathbb{U}_{2n+1}$ of $2n+1$ variables over $\mathbb{F}$ and;
\item an irreducible generic cuspidal automorphic representation $\mathit{\Pi}=\ot_v'\mathit{\Pi}_v$ of 
$\mathbb{U}_{2n+1}(\mathbb{A})$,
\end{itemize}
where $\mathbb{A}$ is the ring of adeles of $\mathbb{F}$, such that
\begin{itemize}
\item $\mathbb{F}_{v_0}\cong F$ and $\mathbb{E}_{v_0}\cong E$;
\item $\mathbb{U}_{2n+1}(\mathbb{F}_{v_0})\cong\U_{2n+1}(F)$;
\item $\mathit{\Pi}_{v_0}\cong \pi$ and;
\item $\mathit{\Pi}_v$ is unramified for every finite place $v\nin S\cup\stt{v_0}$.
\end{itemize}
\end{lm}

\begin{proof}
When $\pi$ is supercuspidal, this follows from a well-known result of Shahidi (\cite[Proposition 5.1]{Shahidi1990}). 
In general, we apply the construction of Gan-Ichino in \cite[Section 6.4]{GanIchino2016} as well as the result of global 
descent of Ginzburg-Rallis-Soudry in \cite{GRS2011} to obtain the desired representation. First, we choose a totally 
complex number field $\bbF$, together with a place $v_0$ of $\bbF$ such that $\bbF_{v_0}\cong F$. Next, we choose a 
quadratic extension $\bbE$ of $\bbF$, such that $\bbE_{v_0}\cong E$ and such that $\bbE$ is unramified over $\bbF$
at all finite places of $\bbF$ outside $v_0$. Let $\bbU_{2n+1}$ be the unitary group over $\bbF$ defined by the matrix 
$S_{2n+1}$ given in \S\ref{SSS:embedding}, so that we have $\bbU_{2n+1}(\bbF_{v_0})\cong\U_{2n+1}(F)$\\

Fix a finite set $S$ of $\bbF$ consisting of split places such that $v_0\nin S$. Then as in \cite[Section 6.4]{GanIchino2016}, 
one can construct an isobaric sum
\[
\Sigma=\Sigma_1\boxplus\Sigma_2\boxplus\cdots\boxplus\Sigma_m
\]
of irreducible cuspidal automorphic representations $\Sigma_i$ of $\GL_{r_i}(\bbA_\bbE)$ ($1\le i\le m$), where 
$\bbA_\bbE$ is the ring of adeles of $\bbE$, such that 
\begin{itemize}
\item
$r_1+r_2+\cdots+r_m=2n+1$;
\item
$\Sigma_i\ncong\Sigma_j$ for $1\le i\neq j\le m$;
\item
the Asai $L$-function $L(s,\Sigma_i,{\rm As})$ has a pole at $s=1$ for all $i$;
\item
$\Sigma_{v_0}\cong {\rm BC}(\pi)$;
\item
$\Sigma_v$ is unramified for every finite place $v$ of $\bbF$ that is not in $S\cup\stt{v_0}$.
\end{itemize}
We indicate that their construction relies on the results of Shin (\cite{SWShin2012}) and Mok (\cite{Mok2015}). Also, in 
their construction, they require that the $L$-parameter of $\pi$, when restricting to $WD_E$, contains at least two 
irreducible direct summands, as they also have some restrictions for $\Sigma_v$ for $v\in S$. Here, since we don't have 
conditions for $\Sigma_v$ when $v\in S$, the assumption on $\pi$ can be removed. Nevertheless, we do need $S$ to be 
non-empty if $\pi$ is not supercuspidal in order to apply the result of Shin.\\

Now, the first three properties of $\Sigma$ meets the require conditions for global descent theorem of \cite{GRS2011}, 
and from $\Sigma$, the global descent gives an irreducible generic cuspidal automorphic representation $\mathit{\Pi}$ of 
$\bbU_{2n+1}(\bbA)$, such that ${\rm BC}(\mathit{\Pi}_v)\cong\Sigma_v$ for every place $v$ of $\bbF$. 
This is exactly what we want.
\end{proof}

\subsection{Proof of \thmref{T:main}}
Notice that \thmref{T:main} holds when (i) $F$ is archimedean; (ii) $E=F\,\x\, F$ and 
(iii) $F$ is non-archimedean and $\pi, \tau$ are unramfied. In fact, by the works of Shelstad (\cite{Shelstad1982}, 
\cite{Shelstad2012}), Mezo (\cite{Mezo2016}) and the recent work of Adrian-Henniart-Kaplan-Oi 
(\cite[cf. Remark B3]{AHKO}), the local Langlands correspondence established by Mok in \cite{Mok2015} coincide with the 
ones for unramified representations and for real algebraic groups (\cite{Langlands1989}). Combining these with the works of 
Shahidi in \cite{Shahidi1985} and \cite{Shahidi1990}, the claim follows. In view of above, we may assume that $F$ is 
non-archimedean and $E$ is a field. Let ${\rm BC}(\pi)$ stands for the (standard) base change lift of $\pi$, which is an 
irreducible representation of $\G_{2n+1}(F)=\GL_{2n+1}(E)$. By \lmref{L:L-isom}, we have to show that 
\begin{equation}\label{E:main id}
\gamma^{{\rm LS}}(s,\pi\x\tau,\psi)
=
\lambda_{E/F}(\psi)^{(2n+1)r}\gamma^{{\rm WD}}(s,{\rm BC}(\pi)\x\tau,\psi_E).
\end{equation}
By the multiplicativity of both gamma factors with respect to $\tau$, it suffices to prove \eqref{E:main id} when $\tau$
supercuspidal. We shall indicate that the proof is similar to the ones in the literature 
 (cf. \cite{JiangSoudry2004}, \cite{CKPSS2004}, \cite{KK2005}), which proved similar identities for other classical groups.\\

Let's begin with the case where $\pi$ is $square$-$integrable$. Let $(\bbF,\bbE,\bbU_{2n+1},\mathit{\Pi})$ be as in 
\lmref{L:globalization} and ${\rm BC}(\mathit{\Pi})=\ot'_v{\rm BC}(\mathit{\Pi})_v$ be the global base change lift of 
$\mathit{\Pi}$. Then ${\rm BC}(\mathit{\Pi})$ is an isobaric irreducible generic automorphic representation of 
$\G_{2n+1}(\bbA)$ (\cite{CPSS2011}, \cite{Mok2015}) such that 
\[
{\rm BC}(\mathit{\Pi})_v\cong{\rm BC}(\mathit{\Pi}_v)
\]
for every place $v$ of $\bbF$. Since $\tau$ is supercuspidal, a result of Shahidi (\cite[Proposition 5.1]{Shahidi1990}) says 
that there exists an irreducible cuspidal automorphic representation $\mathit{\Sigma}=\ot'_v\mathit{\Sigma}_v$ of 
$\G_r(\bbA)$ so that $\mathit{\Sigma}_v$ is unramified for each finite $v\ne v_0$, and at $v=v_0$, one has 
$\mathit{\Sigma}_{v_0}\cong\tau$. Let $\Psi=\prod_v\Psi_v:\bbF\backslash\bbA\to\mathbb{C}^\x$ be a non-trivial additive 
character such that $\Psi_{v_0}=\psi$. Let $T$ be a finite set of places of $\bbF$ containing $v_0\cup S$, all archimedean 
places and all finite places $v$ such that $\Psi_v$ is not unramified. Notices that $\mathit{\Pi}_v$ is unramified for every 
finite place $v\in T\setminus S\cup\stt{v_0}$.\\

By results of Shahidi (\cite{Shahidi1990}) and Jacquet--Piatetski-Shapiro--Shalika 
(\cite{JPSS1983}), the products
\[
L^{{\rm LS}}_{T}(s,\mathit{\Pi}\x\mathit{\Sigma})
:=
\prod_{v\nin T}
L^{{\rm LS}}(s,\mathit{\Pi}_v\x\mathit{\Sigma}_v)
\quad
\text{and}
\quad
L^{{\rm WD}}_T(s,{\rm BC}(\mathit{\Pi})\x\mathit{\Sigma})
:=
\prod_{v\nin T}
L^{{\rm WD}}(s,{\rm BC}(\mathit{\Pi})_v\x\mathit{\Sigma}_v)
\]
which converge for $\Re(s)\gg 0$, admit meromorphic continuations to $\bbC$, and satisfy the following functional equations
\begin{equation}\label{E:Shahidi FE}
L^{{\rm LS}}_T(s,\mathit{\Pi}\x\mathit{\Sigma})
=
\prod_{v\in T}
\gamma^{{\rm LS}}(s,\mathit{\Pi}_v\x\mathit{\Sigma}_v,\Psi_v)\cdot
L^{{\rm LS}}_T(1-s,\tilde{\mathit{\Pi}}\x\tilde{\mathit{\Sigma}})
\end{equation}
and
\[
L^{{\rm WD}}_T(s,{\rm BC}(\mathit{\Pi})\x\mathit{\Sigma})
=
\prod_{v\in T}
\gamma^{{\rm WD}}(s,{\rm BC}(\mathit{\Pi})_v\x\mathit{\Sigma}_v,\Psi_v)\cdot
L^{{\rm WD}}_S(1-s,\widetilde{{\rm BC}(\mathit{\Pi})}\x\tilde{\mathit{\Sigma}})
\]
where $\tilde{\mathit{\Pi}}$ and $\tilde{\mathit{\Sigma}}$ stand for the contragredients of $\mathit{\Pi}$ and $\mathit{\Sigma}$, 
respectively. Since 
\[
L^{{\rm LS}}(s,\mathit{\Pi}_v\x\mathit{\Sigma}_v)
=
L^{{\rm WD}}(s,{\rm BC}(\mathit{\Pi}_v)\x\mathit{\Sigma}_v)
=
L^{{\rm WD}}(s,{\rm BC}(\mathit{\Pi})_v\x\mathit{\Sigma}_v)
\]
for all $v\nin T$, and similar identities hold for the contragredient representations, we get that 
\[
\prod_{v\in T}\gamma^{{\rm LS}}(s,\mathit{\Pi}_v\x\mathit{\Sigma}_v,\Psi_v)
=
\prod_{v\in T}\gamma^{{\rm WD}}(s,{\rm BC}(\mathit{\Pi})_v\x\mathit{\Sigma}_v,\Psi_v).
\]
As \eqref{E:main id} holds for every $v\in T$ with $v\ne v_0$, and $\prod_{v\in T}\lambda_{\bbE_v/\bbF_v}(\Psi_v)=1$, we 
immediately obtain
\begin{align*}
\gamma^{{\rm LS}}(s,\pi\x\tau,\psi)
&=
\gamma^{{\rm LS}}(s,\mathit{\Pi}_{v_0}\x\mathit{\Sigma}_{v_0},\Psi_{v_0})\\
&=
\lambda_{\bbE_{v_0}/\bbF_{v_0}}(\Psi_{v_0})^{(2n+1)r}
\gamma^{{\rm WD}}(s,{\rm BC}(\mathit{\Pi}_{v_0})\x\mathit{\Sigma}_{v_0},(\Psi_{v_0})_{\bbE_{v_0}})\\
&=
\lambda_{E/F}(\psi)^{(2n+1)r}
\gamma^{{\rm WD}}(s,{\rm BC}(\pi)\x\tau,\psi_E).
\end{align*}
This proves \eqref{E:main id} when $\pi$ is square-integrable.\\ 

Next, we consider the case when $\pi$ is $tempered$, but not square-integrable. Let $\phi$ be its corresponding $L$-parameter.
Then the $L$-parameter of ${\rm BC}(\pi)$ is simply $\phi_E:=\phi|_{WD_E}$, which has the following decomposition
\[
\phi_E
=
\phi_1\oplus\dots\oplus\phi_k\oplus\phi_{0,E}\oplus\phi_k^*\oplus\dots\oplus\phi_1^*
\]
as representations of $WD_E$, where for $1\le j\le k$
\begin{itemize}
\item
$\phi_j$ is a $r_j$-dimensional irreducible representation of $WD_E$;
\item
$\phi^*_j$ is the conjugate-dual of $\phi_j$;
\item
$\phi_0$ is a square-integrable $L$-parameter of $\U_{2n_0+1}(F)$ with $n_0:=n-r_1-\dots-r_k\ge 0$ and;
\item
$\phi_{0,E}=\phi_0|_{WD_E}$.
\end{itemize}
Let $\tau_j$ be the square-integrable irreducible representation of $\GL_{r_j}(E)$ attached to $\phi_j$ for $j=1,2,\ldots, k$, 
$\Pi_{\phi_0}$ be the $L$-packet attached to $\phi_0$ and $\pi_0\in\Pi_{\phi_0}$ be the unique generic member. Then $\pi$
is the unique generic summand of 
\[
I^{\U_{2n+1}(F)}(\tau_1,\cdots,\tau_k;\pi_0).
\]
On the other hand, since $\phi_E$ is the $L$-parameter of ${\rm BC}(\pi)$, we see that 
\[
{\rm BC}(\pi)
=
I^{\GL_{2n+1}(E)}(\tau_1,\cdots,\tau_k,{\rm BC}(\pi_0),\tau^*_k,\cdots,\tau^*_1)
\]
where $\tau_j^*$ is the conjugate-dual of $\tau_j$ for $j=1,2,\dots, k$.
By the multiplicativity of the gamma factors and the result for square-integrable representations, we find that 
\begin{align}\label{E:mult-id}
\begin{split}
\gamma^{{\rm LS}}(s,\pi\x\tau,\psi)
=&
\gamma^{{\rm LS}}(s,\pi_0\x\tau,\psi)
\prod_{j=1}^k\gamma^{{\rm LS}}(s,\tau_j\x\tau,\psi)\gamma^{{\rm LS}}(s,\tau^*_j\x\tau,\psi)\\
=&
\lambda_{E/F}(\psi)^{(2n_0+1)r}\gamma^{{\rm WD}}(s,{\rm BC}(\pi_0)\x\tau,\psi_E)\\
&\x\prod_{j=1}^k
\lambda_{E/F}(\psi)^{rr_j}\gamma^{{\rm WD}}(s,\tau_j\x\tau,\psi_E)
\lambda_{E/F}(\psi)^{rr_j}\gamma^{{\rm WD}}(s,\tau^*_j\x\tau,\psi_E)\\
=&
\lambda_{E/F}(\psi)^{(2n+1)r}\gamma^{{\rm WD}}(s,{\rm BC}(\pi)\x\tau,\psi_E).
\end{split}
\end{align}
This proves \eqref{E:main id} when $\pi$ is tempered.\\

Finally, assume that $\pi$ is $non$-$tempered$. The proof of which is similar to the tempered case. Again, let $\phi$ be the 
$L$-parameter associated to $\pi$, and $\phi_E$ denotes its restriction to $WD_E$, which gives rise to the $L$-parameter 
of ${\rm BC}(\pi)$, and admits the following decomposition
\[
\phi_E
=
\phi_1\oplus\dots\oplus\phi_k\oplus\phi_{0,E}\oplus\phi_k^*\oplus\dots\oplus\phi_1^*
\]
as representations of $WD_E$, where for $1\le j\le k$
\begin{itemize}
\item
$\phi_j=\phi'_j\ot |\cdot|_E^{e_j}$ is a $r_j$-dimensional representation of $WD_E$ for some tempered $\phi'_j$ and real 
number $e_j$ such that 
$
e_1>\cdots>e_k>0;
$
\item
$\phi^*_j$ is the conjugate-dual of $\phi_j$;
\item
$\phi_0$ is a tempered $L$-parameter of $\U_{2n_0+1}(F)$ with $n_0:=n-r_1-\dots-r_k\ge 0$ and;
\item
$\phi_{0,E}=\phi_0|_{WD_E}$.
\end{itemize}
Let $\tau_j$ be the essentially tempered irreducible representation of $\GL_{r_j}(E)$ attached to $\phi_j$ for $j=1,2,\ldots, k$, 
$\Pi_{\phi_0}$ be the $L$-packet attached to $\phi_0$ and $\pi_0\in\Pi_{\phi_0}$ be the unique generic member. Then we 
have
\[
\pi=I^{\U_{2n+1}(F)}(\tau_1,\cdots,\tau_k;\pi_0)
\quad\text{and}\quad
{\rm BC}(\pi)
=
I^{\GL_{2n+1}(E)}(\tau_1,\cdots,\tau_k,{\rm BC}(\pi_0),\tau^*_k,\cdots,\tau^*_1).
\]
Now, exactly the same argument as \eqref{E:mult-id} gives \eqref{E:main id} when $\pi$ is non-tempered.
This completes the proof of \thmref{T:main}.\qed

\section{Proof of \thmref{T:main'}}\label{S:proof of main'}

In this section, we prove \thmref{T:main'} under the following assumptions, whose proofs will occur in the subsequent 
sections. We shall follow the notation in \S\ref{SS:notation}.

\begin{assumption}\label{SS:assumption}\noindent
\begin{itemize}
\item
\thmref{T:main''} and \thmref{T:main'''} hold;
\item
\thmref{T:main'} holds when $n=0$ and $r=1$;
\item
\thmref{T:main'} holds when $F$ is archimedean, $E=F\,\x\, F$ and $n\ge r-1$.
\end{itemize}
\end{assumption}

\subsection{Preliminaries}
As an immediately corollary of our assumptions, we see that \thmref{T:main'} holds for unramified representations, which
we record in the following lemma. 

\begin{lm}\label{L:unramified case}
Suppose that $F$ is non-archimedean and both $\pi$ and $\tau$ are unramified. Then under the assumptions 
\S\ref{SS:assumption}, \thmref{T:main'} holds.
\end{lm} 

The next lemma is crucial to proving \thmref{T:main'} when $n<r$. It allows us to overcome the challenges posed by
incomplete results in the archimedean case.

\begin{lm}\label{L:sqr quot}
Let $\pi$ be an irreducible generic supercuspidal representation of $\U_{2n+1}(F)$ (in particular, $F$ is non-archimedean),
and  $r>n+1$ be an integer. Then there exist an integer $k>0$ and an irreducible generic representation $\sigma$ of 
$\G_k(F)$ such that $k+n\ge r-1$ and the generic constituent of the normalized induced representation 
$I^{\U_{2(n+k)+1}(F)}(\sigma;\pi)$ of $\U_{2(n+k)+1}(F)$ is a quotient and square-integrable.
\end{lm}

\begin{proof}
By duality, it suffices to prove the lemma with "quotient" being replaced by "subrepresentation". In this case, 
the assertion follows from a proposition of M\oe glin-Tadi\'c in \cite{MoeglinTadic2002},  the generalized injectivity 
conjecture proposed in \cite{CasselmanShahidi1998}, and proved in \cite{CasselmanShahidi1998}, \cite{Hanzer2010},  
\cite{Hanzer2020}, as well as the injectivity property of general linear groups established in \cite{Zelevinsky1980}. 
More concretely, assume first that $E$ is a field. Then \cite[Proposition 9.1]{MoeglinTadic2002} implies
\footnote{To see this, we may take $\rho$ in the beginning of \cite[Section 9]{MoeglinTadic2002} to be 
any character of $E^\x$ that is trivial on $F^\x$, so that the condition (J-1) in the cited paper is 
satisfied. Then since $Jord_\rho(\pi)$ is a finite subset of $\bbZ$, for a given integer 
$k>0$, we can always find positive integers $a, a_{-}$ satisfying the conditions indicated in the 
beginning of \cite[Section 9]{MoeglinTadic2002} and such that $(a+a_{-})/2=k$.}
that for each integer $k>0$, there exists an essentially square-integrable (and hence generic) 
representation $\sigma$ of $\G_k(F)$ such that the normalized induced representation 
$
I^{\U_{2(n+k)+1}}(\sigma;\pi)
$ 
which is a standard module, contains exactly two irreducible subrepresentations, both of which are square-integrable. Since 
$\sigma$ and $\pi$ are generic, the induced representation $I^{\U_{2(n+k)+1}}(\sigma;\pi)$ has a unique generic 
constituent, which, by the generalized injectivity conjecture proved in \cite{Hanzer2020}, must be a 
subrepresentation. This verifies the lemma when $E$ is a field.\\

The case $E=F\,\x\,F$ can be proved in a similar way. Recall that $\U_{2n+1}(F)=\GL_{2n+1}(F)$, 
$\sigma=\sigma_1\boxtimes\sigma_2$ and 
$
I^{\U_{2n+1}(F)}(\sigma;\pi)=I^{\GL_{2n+1}(F)}(\sigma_1,\pi,\sigma_2)
$ 
in this case. Pick an integer $m>0$ so that $(2n+1)m\ge r-1$. Let $\sigma_1$
(resp. $\sigma_2$) be the unique essentially square-integrable subrepresentation of the normalized induced 
representation 
$
I^{\GL_{(2n+1)m}(F)}(\pi_m,\pi_{m-1},\cdots,\pi_1)\quad \left(\text{resp.}\,\, 
I^{\GL_{(2n+1)m}(F)}(\pi_{-1},\pi_{-2},\cdots,\pi_{-m})\right) 
$
of $\GL_{(2n+1)m}(F)$,
where we denote $\pi_c=\pi\ot|\det|_F^c$ with $c$ an integer. By duality and \cite[Theorem 6.1]{Zelevinsky1980}, we find that 
$I^{\GL_{(2n+1)(2m+1)}(F)}(\sigma_1,\pi,\sigma_2)$ has a unique irreducible subrepresentation. Since the induced 
representation 
\[
I^{\GL_{(2n+1)(2m+1)}(F)}(\pi_m,\cdots,\pi_1,\pi,\pi_{-1},\cdots,\pi_{-m})
\] 
also has the unique irreducible square-integrable subrepresentation, and
$I^{\GL_{(2n+1)(2m+1)}(F)}(\sigma_1,\pi,\sigma_2)$ appears as a subrepresentation, the lemma follows.
\end{proof}

\subsection{Proof of \thmref{T:main'}}
Suppose that $F$ is non-archimedean. Since the cases $E$ is a field and $E=F\,\x\, F$ can be handled similarly, we only 
treat the case where $E$ is a field. As usual, by the multiplicativity of the gamma factors and the fact that 
\[
\gamma^{{\rm WD}}(s,\sigma\x\tau,\psi)=\gamma^{{\rm LS}}(s,\sigma\x\tau,\psi)
\]
for every irreducible generic representations $\sigma$ and $\tau$ of $\G_k(F)$ and $\G_r(F)$, respectively, it is enough to 
prove \thmref{T:main'} when both $\pi$ and $\sigma$ are supercuspidal. Assume first that $n\ge r-1$. The proof of this case
is similar to that of \thmref{T:main}. Indeed, as in \thmref{T:main}, let $(\bbF, \bbE, \bbU_{2n+1},\mathit{\Pi})$ be a global 
object corresponding to the local object $(F,E,\U_{2n+1},\pi)$ by \lmref{L:globalization} (or rather 
\cite[Proposition 5.1]{Shahidi1990}). Since $\pi$ is supercuspidal, we may actually assume that $\mathit{\Pi}_v$ is unramified 
for every finite place $v\neq v_0$.  Similarly, let $\mathit{\Sigma}=\ot'_v\mathit{\Sigma}_v$  be an irreducible cuspudal 
automorphic representation of $\G_r(\bbA)$, which is unramified for every finite $v\ne v_0$, and at $v=v_0$, one has 
$\mathit{\Sigma}_{v_0}\cong\tau$.  Let $\Psi=\prod_v\Psi_v:\bbF\backslash\bbA\to\mathbb{C}^\x$ be a non-trivial additive 
character such that $\Psi_{v_0}=\psi$, and $T$ be a finite set of places of $\bbF$ containing $v_0$, all archimedean places 
and all finite places $v$ such that $\Psi_v$ is not unramified. By results of Ben-Artzi--Soudry (\cite{Ben-ArtziSoudry2009}) 
and Morimoto-Soudry (\cite{MorimotoSoudry2020}), the product 
\[
L^{{\rm RS}}_T(s,\mathit{\Pi}\x\mathit{\Sigma})
:=
\prod_{v\nin T}
L^{{\rm RS}}(s,\mathit{\Pi}_v\x\mathit{\Sigma}_v)
\] 
which converges for $\Re(s)\gg 0$, admits a meromorphic continuation to $\bbC$, and satisfies the following functional 
equation
\[
L^{{\rm RS}}_T(s,\mathit{\Pi}\x\mathit{\Sigma})
=
\prod_{v\in T}
\gamma^{{\rm RS}}(s,\mathit{\Pi}_v\x\mathit{\Sigma}_v,\Psi_v)\cdot
L^{{\rm RS}}_T(1-s,\tilde{\mathit{\Pi}}\x\tilde{\mathit{\Sigma}})
\]
where $\tilde{\mathit{\Pi}}$ and $\tilde{\mathit{\Sigma}}$ stand for the contragredients of $\mathit{\Pi}$ and $\mathit{\Sigma}$, 
respectively. Since 
$
L^{{\rm RS}}(s,\mathit{\Pi}_v\x\mathit{\Sigma}_v)=L^{{\rm LS}}(s,\mathit{\Pi}_v\x\mathit{\Sigma}_v)
$
for all $v\nin T$, and since the similar identity holds for the contragredient representations, we get from \eqref{E:Shahidi FE}
that 
\[
\prod_{v\in T}\gamma^{{\rm RS}}(s,\mathit{\Pi}_v\x\mathit{\Sigma}_v,\Psi_v)
=
\prod_{
v\in T}\gamma^{{\rm LS}}(s,\mathit{\Pi}_v\x\mathit{\Sigma}_v,\Psi_v).
\]
By \lmref{L:unramified case} and the third assumption, we have
$\gamma^{{\rm RS}}(s,\mathit{\Pi}_v\x\mathit{\Sigma}_v,\Psi_v)
=\gamma^{{\rm LS}}(s,\mathit{\Pi}_v\x\mathit{\Sigma}_v,\Psi_v)$ for every $v\in T$, $v\ne v_0$, and hence 
\[
\gamma^{{\rm RS}}(s,\pi\x\tau,\psi)
=
\gamma^{{\rm RS}}(s,\mathit{\Pi}_{v_0}\x\mathit{\Sigma}_{v_0},\Psi_{v_0})
=
\gamma^{{\rm LS}}(s,\mathit{\Pi}_{v_0}\x\mathit{\Sigma}_{v_0},\Psi_{v_0})
=
\gamma^{{\rm LS}}(s,\pi\x\tau,\psi)
\]
as wanted.\\

Now, suppose that $n<r-1$. In this case, we can not apply the above argument directly, as we are not able to define the 
gamma factor $\gamma^{{\rm RS}}(s,\pi\x\tau,\psi)$ when $F$ is archimedean and $n<r-1$. To bypass this difficulty, we 
apply \lmref{L:sqr quot}. Pick an integer $k>0$ and an irreducible generic representation $\sigma$ of $\G_k(F)$ 
such that $n+k\ge r-1$ and the normalized induced representation $I^{\U_{2(n+k)+1}(F)}(\sigma;\pi)$ of $\U_{2(n+k)+1}(F)$ 
has an irreducible square-integrable generic quotient $\eta$. Then exactly the same argument shows
$
\gamma^{{\rm RS}}(s,\eta\x\tau,\psi)
=
\gamma^{{\rm LS}}(s,\eta\x\tau,\psi).
$
It then follows from the multiplicativity of the gamma factors that
\[
\gamma^{{\rm RS}}(s,\pi\x\tau,\psi)
\gamma^{{\rm WD}}(s,\sigma\x\tau,\psi)
\gamma^{{\rm WD}}(s,\tilde{\sigma}\x\tau,\psi)
=
\gamma^{{\rm LS}}(s,\pi\x\tau,\psi)
\gamma^{{\rm LS}}(s,\sigma\x\tau,\psi)
\gamma^{{\rm LS}}(s,\tilde{\sigma}\x\tau,\psi).
\]
Since 
\[
\gamma^{{\rm WD}}(s,\sigma\x\tau,\psi)=\gamma^{{\rm LS}}(s,\sigma\x\tau,\psi)
\quad\text{and}\quad
\gamma^{{\rm WD}}(s,\tilde{\sigma}\x\tau,\psi)=\gamma^{{\rm LS}}(s,\tilde{\sigma}\x\tau,\psi)
\]
which are both non-zero, we conclude that $\gamma^{{\rm RS}}(s,\pi\x\tau,\psi)=\gamma^{{\rm LS}}(s,\pi\x\tau,\psi)$. This 
finishes the proof of \thmref{T:main'}.\qed

\section{Multiplicativity of the Rankin-Selberg gamma factors: 1st variable}\label{S:1st}

\subsection{Prelimilnaries}\label{SS:pre for mult 1}
The aim of this section is to prove \thmref{T:main''}. We begin with recalling the functional equations of the Rankin-Selberg 
integrals developed in \cite{JPSS1983} and \cite{Jacquet2009}.

\subsubsection{Functional equations for $\GL_r\x\GL_k$}
Let $\tau$ and $\sigma$ be irreducible generic representations of $\GL_r(F)$ and $\GL_k(F)$, respectively, with $r>k$. 
Let $n_0,\ell\ge 0$ be integers satisfying $n_0+\ell=r-k-1$. Let $W\in\cW(\tau,\psi_{\GL_r})$,
$W'\in\cW(\sigma,\psi^{-1}_{\GL_k})$ and $s$ be a complex number. Then the Rankin-Selberg integral attached to 
$W$, $W'$ and $n_0$ is given by 
\[
\int_{V_{\GL_k}(F)\backslash\GL_k(F)}\int_{{\rm Mat}_{n_0, k}(F)}
W\left(
\begin{pmatrix}
a&&\\
x&I_{n_0}&\\
&&I_{\ell+1}
\end{pmatrix}
\right)
W'(a)|\det(a)|_F^{s-\frac{r-k}{2}}dxda.
\]
This integral converges absolutely when $\Re(s)\gg 0$ and satisfies the following functional equation
\begin{align}\label{E:FE GL}
\begin{split}
\omega_\sigma(-1)^{r-1}&\gamma^{{\rm WD}}(s,\tau\x\sigma,\psi)
\int_{V_{\GL_k}(F)\backslash\GL_k(F)}\int_{{\rm Mat}_{n_0, k}(F)}
W\left(
\begin{pmatrix}
a&&\\
x&I_{n_0}&\\
&&I_{\ell+1}
\end{pmatrix}
\right)
W'(a)|\det(a)|_F^{s-\frac{r-k}{2}}dxda\\
&=\int_{V_{\GL_k}(F)\backslash\GL_k(F)}\int_{{\rm Mat}_{k,\ell}(F)}
W\left(
\begin{pmatrix}
&I_{n_0+1}&\\
&&I_{\ell}\\
a&&x
\end{pmatrix}
\right)
W'(a)|\det(a)|_F^{(s-1)+\frac{r-k}{2}-\ell}dxda.
\end{split}
\end{align}
We also record the following identities
\begin{equation}\label{E:gamma GL 1}
\gamma^{{\rm WD}}(s,\tau\x\sigma,\psi)\gamma^{{\rm WD}}(1-s,\t{\tau}\x\t{\sigma},\psi^{-1})=1;
\end{equation}
\begin{equation}\label{E:gamma GL 2}
\gamma(s,\tau\x\sigma,\psi^{-1})
=\omega_\tau(-1)^k\omega_\sigma(-1)^r\gamma(s,\tau\x\sigma,\psi).
\end{equation}

\subsubsection{A reduction}
Following a trick of Soudry in \cite{Soudry2000}, we can reduce the proof of \thmref{T:main''} to the case $n<r$, assuming 
the validity of \thmref{T:main'''}. This is documented in the next lemma.

\begin{lm}\label{L:reduction}
Suppose that $F$ is non-archimedean and \thmref{T:main'''} holds. Then if \thmref{T:main''} holds for $n<r$, it holds for 
every $n,r$.
\end{lm}

\begin{proof}
Suppose that $E$ is a field and $\pi$ is a quotient of the induced representation $I^{\U_{2n+1}(F)}(\sigma;\pi_0)$ of 
$\U_{2n+1}(F)$, where $\sigma, \pi_0$ are irreducible generic representations of $\G_k(F), \U_{2n_0+1}(F)$, respectively, 
with $k\ge 1$, $n\ge 0$ integers such that $k+n_0=n$. Let $\tau$ be an irreducible generic representation of $\G_r(F)$ 
with $n\ge r$. Since $\tau$ is generic, we can always find a character $\chi$ of $E^\x$, such that the induced representation
\[
\eta=I^{\GL_{r+m}(E)}(\tau,\chi_m)
\quad\text{with}\quad
\chi_m:=I^{\GL_m(F)}(\chi,\cdots,\chi)
\quad\text{($m$ copies of $\chi$)}
\]
of $\G_{r+m}(F)=\GL_{r+m}(E)$ is irreducible and generic for every $m>0$. This follows from the classification result of 
Zelevinsky (\cite{Zelevinsky1980}).
Now pick $m$ so that $n<r+m$. Then our assumptions imply
\begin{align*}
\gamma^{{\rm RS}}(s,\pi\x\eta,\psi)
=&
\gamma^{{\rm RS}}(s,\pi_0\x\eta,\psi)
\gamma^{{\rm WD}}(s,\sigma\x\eta,\psi)
\gamma^{{\rm WD}}(s,\t{\sigma}\x\eta,\psi)\\
=&
\gamma^{{\rm RS}}(s,\pi_0\x\tau,\psi)
\gamma^{{\rm WD}}(s,\sigma\x\tau,\psi)
\gamma^{{\rm WD}}(s,\t{\sigma}\x\tau,\psi)\\
&\x
\left[
\gamma^{{\rm RS}}(s,\pi_0\x\chi,\psi)
\gamma^{{\rm WD}}(s,\sigma\x\chi,\psi)
\gamma^{{\rm WD}}(s,\t{\sigma}\x\chi,\psi)
\right]^m
\end{align*}
and
\[
\gamma^{{\rm RS}}(s,\pi\x\eta,\psi)
=
\gamma^{{\rm RS}}(s,\pi\x\tau,\psi)\gamma^{{\rm RS}}(s,\pi\x\chi,\psi)^m.
\]
Notice that in the above derivation, we also use the multiplicativity of the Weil-Degline gamma factors. Comparing these two
equations, it remains to verify that 
\begin{equation}\label{E:mult char}
\gamma^{{\rm RS}}(s,\pi\x\chi,\psi)
=
\gamma^{{\rm RS}}(s,\pi_0\x\chi,\psi)
\gamma^{{\rm WD}}(s,\sigma\x\chi,\psi)
\gamma^{{\rm WD}}(s,\t{\sigma}\x\chi,\psi).
\end{equation}
We use a similar track again, namely, for each $m>n$, the representation $\chi_m$ of $\G_m(F)$ is irreducible and generic,
so the assumptions give 
\begin{align*}
\gamma^{{\rm RS}}(s,\pi\x\chi_m,\psi)
&=
\gamma^{{\rm RS}}(s,\pi_0\x\chi_m,\psi)
\gamma^{{\rm WD}}(s,\sigma\x\chi_m,\psi)
\gamma^{{\rm WD}}(s,\t{\sigma}\x\chi_m,\psi)\\
&=
\left[
\gamma^{{\rm RS}}(s,\pi_0\x\chi,\psi)
\gamma^{{\rm WD}}(s,\sigma\x\chi,\psi)
\gamma^{{\rm WD}}(s,\t{\sigma}\x\chi,\psi)
\right]^m
\end{align*}
and 
\[
\gamma^{{\rm RS}}(s,\pi\x\chi_m,\psi)
=
\gamma^{{\rm RS}}(s,\pi\x\chi,\psi)^m.
\]
Since $m>n$ can be arbitrary, the identity \eqref{E:mult char} follows. This shows the lemma when $E$ is a field.\\

The argument for the case $E=F\,\x\,F$ is similar, we just need to pick two characters $\chi_1, \chi_2$ of $F^\x$ such that 
the induced representations: 
\[
I^{\GL_{r+m}(F)}(\tau_1,\chi_1,\cdots,\chi_1)
\quad\text{($m$ copies of $\chi_1$)}
\quad\text{and}\quad
I^{\GL_{r+m}(F)}(\tau_2,\chi_2,\cdots,\chi_2)
\quad\text{($m$ copies of $\chi_2$)}
\]
of $\GL_{r+m}(F)$ are both irreducible and generic. Here $\tau_1, \tau_2$ are irreducible generic representations of 
$\GL_r(F)$ such that $\tau=\tau_1\boxtimes\tau_2$. This completes the proof.
\end{proof}


\subsubsection{Induced representations of $\U_{2n+1}$}
Let $1\le k\le n$ be an integer and put $n_0=n-k$. Let $\sigma$ be an irreducible generic representation of $\G_k(F)$ 
and $\pi_0$ be a similar representation of $\U_{2n_0+1}(F)$. We realize $\pi_0$ (resp. $\sigma$) in the Whittaker model
$\cW(\pi_0,\psi_{\U_{2n_0+1}})$ (resp. $\cW(\sigma,\psi_{\G_k})$). Let $P\subset \U_{2n+1}$ be the standard parabolic 
subgroup whose Levi subgroup $M_P\cong\G_k\,\x\,\U_{2n_0+1}$. When $E=F\,\x\,F$, we understand that 
$M_P\cong\GL_k\,\x\,\GL_{2n_0+1}\,\x\,\GL_k$. Let $\b{P}$ be the opposite of $P$, namely, $\b{P}$ and $P$ have the same 
Levi subgroup, and the unipotent radical $N_{\b{P}}$ of $\b{P}$ consists of low triangular matrices.
Let $\pi$ be an irreducible generic representation of $\U_{2n+1}(F)$, which appears as a quotient of the following normalized induced representation ${\rm Ind}_{\b{P}(F)}^{\U_{2n+1}(F)}(\sigma\boxtimes\pi_0)$.
When $E=F\,\x\,F$, we have $\sigma=\sigma_1\boxtimes\sigma_2$, so that $\sigma\boxtimes\pi$ stands for
$\sigma_1\boxtimes\pi_0\boxtimes\sigma_2$. Notice that the Whittaker models of $\pi$ and the induced 
representation mentioned above coincide. Since the Rankin-Selberg integrals are defined in terms of the Whittaker 
functions of $\pi$, we may assume that 
\[
\pi
=
{\rm Ind}_{\b{P}(F)}^{\U_{2n+1}(F)}(\sigma\boxtimes\pi_0).
\]
The underly space $\cV(\pi, \psi)$ of $\pi$ consists of smooth functions 
\[
\varphi:\U_{2n+1}(F)\,\x\,\G_k(F)\,\x\,\U_{2n_0+1}(F)\to\bbC
\]
satisfying (i) for a fixed $(g_0,a)\in\U_{2n+1}(F)\,\x\,\G_k(F)$, the function $g_0\mapsto\varphi(g,a, g_0)$ belongs to 
$\cW(\pi_0,\psi_{\U_{2n_0+1}})$; for a fixed $(g,g_0)\in\U_{2n+1}(F)\,\x\,\U_{2n_0+1}(F)$, the function 
$a\mapsto\varphi(g,a,g_0)$ belongs to $\cW(\sigma,\psi_{\G_k})$ and; (iii)
\[
\varphi(hg,a,g_0)=\delta^{\frac{1}{2}}_{\b{P}}(h)\varphi(g,ab, g_0h_0)
\]
for $g\in\U_{2n+1}(F)$, $h\in\b{P}(F)$, $g_0, h_0\in\U_{2n_0+1}(F)$ and $a,b\in\G_k(F)$, where $(b,h_0)$ is the "Levi part"
of $h$ under the Levi decomposition of $\b{P}(F)$.

\subsubsection{Whittaker functions of $\pi$}
We need the Whittaker model of $\pi$. If $\varphi\in\cV_\pi$, then the Whittaker function $W_\varphi$ attached to $\varphi$
can be realized as the following integral
\begin{equation}\label{E:Whittaker function for pi}
W_\varphi(g)
=
\int_{N_P(F)}
\varphi(yg,I_k,I_{2n_0+1})\psi^{-1}_{N_P}(y)dy
\end{equation}
where $\psi_{N_P}$ is the character of $N_P(F)$ from the restriction of $\psi_{\U_{2n+1}}$ to $N_P(F)$. Observe
that formally, 
\[
W_\varphi(ug)=\psi_{\U_{2n+1}}(u)W_\varphi(g)
\]
for $u\in V_{\U_{2n+1}}(F)$, and $g\in\U_{2n+1}(F)$. The problem is that the integral may not converge absolutely. To rectify
this, let $\zeta$ be a complex number and replace $\sigma$ by
\[
\sigma_\zeta=\sigma\ot|\det|_K^{-\zeta}
\]
where, as usual, $K=E$ or $F$ depending on whether $E$ is a field or not.
The resulting induced representation becomes $\pi_\zeta$ and we take a holomorphic section $\varphi_\zeta$ 
instead of $\varphi$ in \eqref{E:Whittaker function for pi}. Now the integral defining $W_{\varphi_\zeta}$ converges absolutely 
for $\Re(\zeta)\gg 0$, and has a continuation to a holomorphic function on the whole complex plane. Moreover, for a given
$W\in\cW(\pi,\psi_{\U_{2n+1}})$, there is a standard section $\varphi_\zeta$ of $\pi_\zeta$ such that $W_{\varphi_0}=W$ 
(cf. \cite{BP2021}).

\subsubsection{Conventions}\label{SSS:convention}
Similar to \cite{Kaplan2015} and \cite{Morimoto}, we prove \thmref{T:main''} using formal arguments, neglecting the 
convergence issues. The convergence can be confirmed through standard methods, as demonstrated in 
\cite{Ben-ArtziSoudry2009} and \cite{MorimotoSoudry2020}. For instance, we will omit $\zeta$ in the proof of \thmref{T:main''} 
when considering the Whittaker function for $\varphi\in\cV(\pi,\psi)$, pretending that \eqref{E:Whittaker function for pi} 
converges absolutely. For the rigorous arguments regarding the establishment of the multiplicativity of the gamma factors 
involving $\zeta$, as well as how the domains of convergence of the Rankin-Selberg integrals depend on $\zeta$, we refer 
to the proofs in \cite[Section 11]{Soudry1993}.
During the proofs, we have to show that one integral is equal to another at various places. It may happen that the domains of 
absolute convergence of these two integrations are disjoint. However, since the integrals always admit the meromorphic 
continuation to the whole complex plane, the equality is understood in the sense of the meromorphic continuation; to 
prove that they are the same, we use the fact that these integrals define elements in certain Hom-space, the dimension of 
which is one. For more in-depth understanding of this type of argument, see \cite[11.8]{Soudry1993}.
Another convention is that we only prove \thmref{T:main''} when $E=F\,\x\,F$. 
This choice is motivated by two main reasons. First, the calculations for these two cases, namely when $E$ is a field and 
when $E=F\,\x\,F$, are quite similar. Second, when $E$ is a field, the computations closely follow, almost word for word, 
the calculations in op. cit.. Actually, our guideline for the proof of \thmref{T:main''} when $E$ is a field is that of Soudry in 
\cite[Section 11]{Soudry1993}. For the case $E=F\,\x\,F$, we expect that the computations are similar to the case $E$ is a 
field. This is the idea behind our proofs. Finally, to simplify the notation, references to the field $F$ are omitted from the 
notation; thus for instance, $\GL_r$ means $\GL_r(F)$.

\subsection{Proof of \thmref{T:main''}}
We follow the notation and conventions in \S\ref{SS:pre for mult 1}. Our local field $F$ now is non-archimedean, and by 
\lmref{L:reduction}, we may assume that $n<r$. Let us denote 
\[
\ell=r-n-1.
\]
Let $\varphi\in\cV(\pi,\psi)$ and $\xi_s\in I_r(\tau,s)$. Our goal is to establish the following identity
\begin{align}\label{E:main id for main''}
\begin{split}
\omega_{\sigma}&(-1)^{r-1}
\gamma^{{\rm WD}}(s,\sigma_1\x\tau_1,\psi^{-1})\gamma^{{\rm WD}}(s,\t{\sigma}_2\x\tau_2,\psi^{-1})
\Psi_{n,r}(\varphi\ot\xi_s)\\
&\quad=
\int_{N_{\b{P}}\hat{V}_{\G_k}\GL_{2n_0+1}^\Diamond\backslash\GL_{2n+1}}\int_{{\rm Mat}_{\ell, k}}
\int_{{\rm Mat}_{k,\ell}}\\
&\quad\quad\x\int_{V_{\GL_{2n_0+1}}\backslash\GL_{2n_0+1}}
\varphi(g, I_k, g_0)
\int_{\b{X}^{n_0,r}}
f_{\xi_s}\left(
\b{u}\jmath^{n_0,r}(g_0)\pMX{\dot{x}}{}{}{\ddot{y}}\jmath^{n,r}(g), I_r\right)\psi^{-1}_{\b{X}^{n_0,r}}(\b{u})d\b{u}dg_0dxdydg
\end{split}
\end{align}
where $\hat{V}_{\G_k}$ and  $\GL_{2n_0+1}^\Diamond$ are subgroups 
of $\GL_{2n+1}$ given by
\[
\hat{V}_{\G_k}
=
\stt{\begin{pmatrix}z_1&&\\ &I_{2n_0+1}&\\ &&z_2\end{pmatrix}
\mid z_1,z_2\in V_{\GL_k}}
\quad\text{and}\quad
\GL_{2n_0+1}^\Diamond
=
\stt{\begin{pmatrix}1_k&&\\ &g_0&\\ &&i_k\end{pmatrix}
\mid g_0\in\GL_{2n_0+1}}
\]
while for $x\in{\rm Mat}_{k\x\ell}$ and $y\in{\rm Mat}_{\ell\x k}$, we put
\[
\dot{x}
=
\begin{pmatrix}
&I_{n_0+1}&\\
&&I_\ell\\
I_k&&x
\end{pmatrix}
\quad\text{and}\quad
\ddot{y}
=
\begin{pmatrix}
&&I_k\\
I_\ell&&y\\
&I_{n_0+1}&
\end{pmatrix}.
\]
Notice that for fixed $g$ and $x,y$, the inner integral
\[
\int_{V_{\GL_{2n_0+1}}\backslash\GL_{2n_0+1}}
\varphi(g, I_k, g_0)
\int_{\b{X}^{n_0,r}}
f_{\xi_s}\left(
\b{u}\jmath^{n_0,r}(g_0)\pMX{\dot{x}}{}{}{\ddot{y}}\jmath^{n,r}(g), I_r\right)\psi^{-1}_{\b{X}^{n_0,r}}(\b{u})d\b{u}dg_0
\]
is noting but the Rankin-Selberg integral attached to $\pi_0$ and $\tau$.
We should point out that when $\pi$ and $\tau$ are unramifed, Morimoto-Soudry have obtained a similar identity in 
\cite[Theorem 7.5]{MorimotoSoudry2020}. Furthermore, once \eqref{E:main id for main''} is 
established, one can apply the arguments of Soudry in \cite[Sections 11.1-11.4]{Soudry1993} to obtain 
\thmref{T:main''}. More precisely, we can deduce from \eqref{E:main id for main''} the following identity
\[
\Gamma_\delta(s,\pi\x\tau,\psi)
=
\omega_{\sigma}(-1)^r\omega_\tau(-1)^k
\Gamma(s,\pi_0\x\tau,\psi)
\gamma^{{\rm WD}}(s,\sigma\x\tau_1,\psi)
\gamma^{{\rm WD}}(s,\tilde{\sigma}\x\tau_2,\psi)
\]
from which \thmref{T:main''} follows.\\

The derivation of \eqref{E:main id for main''} is a multistep process, which we will begin now. For simplicity, let us denote in
this section that $\nu(a)=|\det(a)|_F$ for $a\in\GL_N$.

\subsubsection{}
Substitute \eqref{E:Whittaker function for pi} in $\Psi_{n,r}(\varphi\ot\xi_s)$ (cf. \eqref{E:RS int n<r} with $v$ instead of 
$\varphi$), we get that
\begin{align}\label{E:main'' 1}
\begin{split}
\Psi_{n,r}(\varphi\ot\xi_s)
&=
\int_{V_{\GL_{2n+1}}\backslash\GL_{2n+1}}\int_{N_P}
\varphi(yg,I_k, I_{2n_0+1})\psi^{-1}_{N_P}(y)
\int_{\b{X}^{n,r}}f_{\xi_s}(\b{u}\jmath^{n,r}(g),I_r)\psi^{-1}_{\b{X}^{n,r}}(\b{u})d\b{u}dydg\\
&=
\int_{V_{\GL_{2n+1}}\backslash\GL_{2n+1}}\int_{N_P}
\varphi(yg,I_k, I_{2n_0+1})
\int_{\b{X}^{n,r}}f_{\xi_s}(\b{u}\jmath^{n,r}(yg),I_r)\psi^{-1}_{\b{X}^{n,r}}(\b{u})d\b{u}dydg\\
&=
\int_{\hat{V}_{\G_k}V_{\GL_{2n_0+1}}^\Diamond\backslash\GL_{2n+1}}
\varphi(g,I_k, I_{2n_0+1})
\int_{\b{X}^{n,r}}f_{\xi_s}(\b{u}\jmath^{n,r}(g),I_r)\psi^{-1}_{\b{X}^{n,r}}(\b{u})d\b{u}dg.
\end{split}
\end{align}
In the above computation, we have used the fact that 
\begin{align*}
\int_{\b{X}^{n,r}}f_{\xi_s}(\b{u}\jmath^{n,r}(yg),I_r)\psi^{-1}_{\b{X}^{n,r}}(\b{u})d\b{u}
&=
\psi^{-1}_{V_{\GL_{2n+1}}}(y)
\int_{\b{X}^{n,r}}f_{\xi_s}(\b{u}\jmath^{n,r}(g),I_r)\psi^{-1}_{\b{X}^{n,r}}(\b{u})d\b{u}\\
&=
\psi^{-1}_{N_P}(y)
\int_{\b{X}^{n,r}}f_{\xi_s}(\b{u}\jmath^{n,r}(g),I_r)\psi^{-1}_{\b{X}^{n,r}}(\b{u})d\b{u}
\end{align*}
for $y\in N_P$, and we denote 
\[
V_{\GL_{2n_0+1}}^\Diamond
=
\stt{\begin{pmatrix}I_k&&\\ &z_0&\\ &&I_k\end{pmatrix}
\mid z_0\in V_{\GL_{2n_0+1}}}.
\]

\subsubsection{}
To proceed, we need the following integration formula:
\begin{align}\label{E:int formula}
\begin{split}
\int_{V_{\G_k} V_{\GL_{2n_0+1}}\backslash\GL_{2n+1}} F(g)dg
=
&\int_{\b{P}'V_{\GL_{2n_0+1}}^\Diamond\backslash\GL_{2n+1}}
\int_{V_{\GL_k}\backslash\GL_k}\int_{V_{\GL_k}\backslash\GL_k}\int_{{\rm Mat}_{k,n_0}}\int_{\M_{n_0,k}}\\
&\quad\quad 
F\left(
\begin{pmatrix}
a&&&&\\
&I_{n_0}&&&\\
&&1&&\\
&&&I_{n_0}&\\
&&&&b
\end{pmatrix}
\begin{pmatrix}
I_k&&&&\\
x&I_{n_0}&&&\\
&&1&&\\
&&&I_{n_0}&\\
&&&y&I_k
\end{pmatrix}g
\right)
dxdydadbdg
\end{split}
\end{align}
where $\b{P}'\subset\b{P}$ is the subgroup given by 
\[
\b{P'}
=
\stt{
\begin{pmatrix}
a&&&&\\
x&I_{n_0}&&&\\
&&1&&\\
&&&I_{n_0}&\\
&&&y&b
\end{pmatrix}\mid
a,b\in\GL_k, x\in\M_{n_0,k}, y\in\M_{k,n_0}
}.
\]
The integral $"\int_{\b{P}'V_{\GL_{2n_0+1}}^\Diamond\backslash\GL_{2n+1}}"$ involves a slight abuse of notation, and it 
should be interpreted within the context of the Iwasawa decomposition of $\GL_{2n+1}$ (as detailed in 
\cite[Page 69]{Soudry1993}).

\subsubsection{}
Applying \eqref{E:int formula} to \eqref{E:main'' 1}, and noticing that for 
\[
\b{p}
=
\begin{pmatrix}
a&&&&\\
x&I_{n_0}&&&\\
&&1&&\\
&&&I_{n_0}&\\
&&&by&b
\end{pmatrix}
\in\b{P}'
\]
we have
\begin{itemize}
\item
$\b{u}':=\jmath^{n,r}(\b{p})^{-1}\b{u}\jmath^{n,r}(\b{p})\in\b{X}^{n,r}$, $\psi_{\b{X}^{n,r}}(\b{u}')=\psi_{\b{X}^{n,r}}(\b{u})$
and $d\b{u}=\nu(ab^{-1})^{-\ell} d\b{u}'$;\\
\item 
$\varphi(\b{p}g,I_k,I_{2n_0+1})=\nu(ab^{-1})^{\frac{2n+1-k}{2}}\varphi(g,a,b,I_{2n_0+1})$ and;
\item
$f_{\xi_s}(\b{u}\jmath^{n,r}(\b{p}g), I_r)=\nu(ab^{-1})^{s-\frac{r-1}{2}}f_{\xi_s}\left(\b{u}'\jmath^{n,r}(g), 
\begin{pmatrix}a&&\\x&I_{n_0}\\&&I_{\ell+1}\end{pmatrix}, \begin{pmatrix}I_{\ell+1}&&\\&I_{n_0}&\\&by&b\end{pmatrix}\right)$;
\end{itemize}
it follows that $\Psi_{n,r}(\varphi\ot\xi_s)$ is equal to
\begin{align*}
\int_{\b{P}'V_{\GL_{2n_0+1}}^\Diamond\backslash\GL_{2n+1}}
&\int_{\b{X}^{n,r}}\psi^{-1}_{\b{X}^{n,r}}(\b{u})
\int_{V_{\GL_k}\backslash\GL_k}\int_{V_{\GL_k}\backslash\GL_k}\int_{{\rm Mat}_{k,n_0}}\int_{\M_{n_0,k}}
dxdydadbd\b{u}dg\\
&\varphi(g,a,b,I_{2n_0+1})f_{\xi_s}\left(\b{u}\jmath^{n,r}(g), 
\begin{pmatrix}a&&\\x&I_{n_0}\\&&I_{\ell+1}\end{pmatrix}, \begin{pmatrix}I_{\ell+1}&&\\&I_{n_0}&\\&by&b\end{pmatrix}\right)
\nu(ab^{-1})^{s-\frac{r-k}{2}}
\end{align*}
which, after changing the variable $by\mapsto y$, becomes
\begin{align}\label{E:main'' 2}
\begin{split}
&\int_{\b{P}'V_{\GL_{2n_0+1}}^\Diamond\backslash\GL_{2n+1}}
\int_{\b{X}^{n,r}}\psi^{-1}_{\b{X}^{n,r}}(\b{u})
\int_{V_{\GL_k}\backslash\GL_k}\int_{V_{\GL_k}\backslash\GL_k}
\int_{{\rm Mat}_{k,n_0}}\int_{\M_{n_0,k}}
dxdydadbd\b{u}dg\\
&\quad\quad\quad
\varphi(g,a,b,I_{2n_0+1})
f_{\xi_s}\left(\b{u}\jmath^{n,r}(g), 
\begin{pmatrix}a&&\\x&I_{n_0}\\&&I_{\ell+1}\end{pmatrix}, \begin{pmatrix}I_{\ell+1}&&\\&I_{n_0}&\\&y&b\end{pmatrix}\right)
\nu(a)^{s-\frac{r-k}{2}}\nu(b)^{-s+\frac{r-k}{2}-n_0}.
\end{split}
\end{align}

\subsubsection{}
Observe that 
\[
\begin{pmatrix}
I_{\ell+1}&&\\&I_{n_0}&\\&y&b
\end{pmatrix}
=
\begin{pmatrix}
&I_{\ell+1}&\\
&&I_{n_0}\\
b&&y
\end{pmatrix}
\begin{pmatrix}
&&I_k\\
I_{\ell+1}&&\\
&I_{n_0}&
\end{pmatrix}.
\]
Therefore, we can apply the functional equations \eqref{E:FE GL} for $\GL_r\x\GL_k$ to the integrations
\[
\int_{V_{\GL_k}\backslash\GL_k}\int_{V_{\GL_k}\backslash\GL_k}\int_{{\rm Mat}_{n_0, k}}\int_{\M_{k,n_0}}
\]
in \eqref{E:main'' 2}. More precisely, by multiplying the gamma factors 
\[
\omega_{\sigma}(-1)^{r-1}\gamma^{{\rm WD}}(s,\sigma_1\x\tau_1,\psi^{-1})\gamma(1-s,\sigma_2\x\t{\tau}_2,\psi)^{-1}
\]
to $\Psi_{n,r}(\varphi\ot\xi_s)$, and applying the functional equation \eqref{E:FE GL}, we see that \eqref{E:main'' 2} changes 
to 
\begin{align*}
&\int_{\b{P}'V_{\GL_{2n_0+1}}^\Diamond\backslash\GL_{2n+1}}
\int_{\b{X}^{n,r}}\psi^{-1}_{\b{X}^{n,r}}(\b{u})
\int_{V_{\GL_k}\backslash\GL_k}\int_{V_{\GL_k}\backslash\GL_k}
\int_{{\rm Mat}_{\ell,k}}\int_{\M_{k,\ell}}
dxdydadbd\b{u}dg\\
&\quad\varphi(g,a,b,I_{2n_0+1})f_{\xi_s}\left(\b{u}\jmath^{n,r}(g), 
\begin{pmatrix}
&I_{n_0+1}&\\&&I_\ell\\a&&x
\end{pmatrix}, 
\begin{pmatrix}
b&&\\y&I_\ell&\\&&I_{n_0+1}
\end{pmatrix}
\begin{pmatrix}
&&I_k\\
I_{r-n}&&\\
&I_{n_0}&
\end{pmatrix}
\right)\nu(a)^{s-\frac{r+k}{2}+n}\nu(b)^{1-s+\frac{r-k}{2}}.
\end{align*}
Since 
\[
\begin{pmatrix}
b&&\\y&I_\ell&\\&&I_{n_0+1}
\end{pmatrix}
\begin{pmatrix}
&&I_k\\
I_{r-n}&&\\
&I_{n_0}&
\end{pmatrix}
=
\begin{pmatrix}
b&&\\y&I_\ell&\\&&I_{n_0+1}
\end{pmatrix}
\begin{pmatrix}
&&I_k\\
I_\ell&&\\
&I_{n_0+1}&
\end{pmatrix}
=
\begin{pmatrix}
&&b\\
I_\ell&&y\\
&I_{n_0+1}&
\end{pmatrix}
\]
and 
\[
\gamma(1-s,\sigma_2\x\t{\tau}_2,\psi)^{-1}
=
\gamma(s,\sigma_2\x\t{\tau}_2,\psi^{-1})
\]
by \eqref{E:gamma GL 1}, we find that
\[
\omega_{\sigma}(-1)^{r-1}\gamma^{{\rm WD}}(s,\sigma_1\x\tau_1,\psi^{-1})\gamma(s,\sigma_2\x\t{\tau}_2,\psi^{-1})
\Psi_{n,r}(\varphi\ot\xi_s)
\]
equals to
\begin{align*}
&\int_{\b{P}'V_{\GL_{2n_0+1}}^\Diamond\backslash\GL_{2n+1}}
\int_{\b{X}^{n,r}}\psi^{-1}_{\b{X}^{n,r}}(\b{u})
\int_{V_{\GL_k}\backslash\GL_k}\int_{V_{\GL_k}\backslash\GL_k}\int_{{\rm Mat}_{\ell,k}}\int_{\M_{k,\ell}}
dxdydadbd\b{u}dg\\
&\quad\quad\quad\varphi(g,a,b,I_{2n_0+1})f_{\xi_s}\left(\b{u}\jmath^{n,r}(g), 
\begin{pmatrix}
&I_{n_0+1}&\\&&I_\ell\\a&&x
\end{pmatrix}, 
\begin{pmatrix}
&&b\\
I_\ell&&y\\
&I_{n_0+1}&
\end{pmatrix}
\right)
\nu(a)^{s-\frac{r+k}{2}+n}\nu(b)^{1-s+\frac{r-k}{2}}
\end{align*}
which, after changing the variable $y\mapsto by$, can be written as
\begin{align}\label{E:main'' 3}
\begin{split}
&\int_{\b{P}'V_{\GL_{2n_0+1}}^\Diamond\backslash\GL_{2n+1}}
\int_{\b{X}^{n,r}}\psi^{-1}_{\b{X}^{n,r}}(\b{u})
\int_{V_{\GL_k}\backslash\GL_k}\int_{V_{\GL_k}\backslash\GL_k}\int_{{\rm Mat}_{\ell,k}}\int_{\M_{k,\ell}}
dxdydadbd\b{u}dg\\
&\quad\quad\varphi(g,a,b,I_{2n_0+1})f_{\xi_s}\left(\b{u}\jmath^{n,r}(g), 
\begin{pmatrix}
&I_{n_0+1}&\\&&I_\ell\\a&&x
\end{pmatrix}, 
\begin{pmatrix}
&&b\\
I_\ell&&by\\
&I_{n_0+1}&
\end{pmatrix}
\right)\nu(a)^{s-\frac{r+k}{2}+n}\nu(b)^{1-s+\frac{r-k}{2}+\ell}.
\end{split}
\end{align}

\subsubsection{}
To continue, for given $x\in\M_{k,\ell}$ and $y\in\M_{\ell,k}$, we put
\[
\dot{x}
=
\begin{pmatrix}
&I_{n_0+1}&\\
&&I_\ell\\
I_k&&x
\end{pmatrix}
\quad\text{and}\quad
\ddot{y}
=
\begin{pmatrix}
&&I_k\\
I_\ell&&y\\
&I_{n_0+1}&
\end{pmatrix}.
\]
We also write
\[
\hat{a}
=
\pMX{a}{}{}{I_{r-k}}
\quad\text{and}\quad
\Check{b}
=
\pMX{I_{r-k}}{}{}{b}
\]
for $a,b\in\GL_k$, so that
\[
\begin{pmatrix}&I_{n_0+1}&\\&&I_\ell\\a&&x\end{pmatrix}=\dot{x}\hat{a}
\quad\text{and}\quad
\begin{pmatrix}&&b\\I_\ell&&by\\&I_{n_0+1}&\end{pmatrix}=\ddot{y}\Check{b}.
\]
Noticing that for 
\[
c
:=
\pMX{\hat{a}}{}{}{\Check{b}}\in\GL_{2r}
\]
we have
\begin{itemize}
\item if $\b{u}\in\b{X}^{n,r}$ and $\b{u}':=c\b{u}c^{-1}$, then $\b{u}'\in\b{X}^{n,r}$, 
$\psi_{\b{X}^{n,r}}(\b{u}')=\psi_{\b{X}^{n,r}}(\b{u})$ and $d\b{u}=\nu(ab^{-1})^\ell d\b{u}'$;\\
\item
$\varphi(g,a,b,I_{2n_0+1})=\nu(ab^{-1})^{\frac{2n+1-k}{2}}\varphi(cg,I_k,I_{2n_0+1})$;\\
\item
$f_{\xi_s}\left(\b{u}\jmath^{n,r}(g), \dot{x}\hat{a},\ddot{y}\Check{b}\right)
=\nu(ab^{-1})^{-s-\frac{r-1}{2}}f_{\xi_s}\left(\pMX{\dot{x}}{}{}{\ddot{y}}\b{u}'\jmath(cg);I_r\right)$.
\end{itemize}
Moreover, if $\b{P}''\subset\b{P}'$ is the subgroup defined by 
\[
\b{P}''
=
\stt{
\begin{pmatrix}
z_1&&&&\\
x&I_{n_0}&&&\\
&&1&&\\
&&&I_{n_0}&\\
&&&y&z_2
\end{pmatrix}\mid
z_1,z_2\in\GL_k, x\in\M_{n_0,k}, y\in\M_{k,n_0}
}
\]
then we have the following integration formula
\[
\int_{\b{P}''\backslash\b{P}'}
F(\b{p})d\b{p}
=
\int_{V_{\GL_k}\backslash\GL_k}\int_{V_{\GL_k}\backslash\GL_k}
F\left(
\begin{pmatrix}
a&&&&\\
&I_{n_0}&&&\\
&&1&&\\
&&&I_{n_0}&\\
&&&&b
\end{pmatrix}\right)
\nu(ab^{-1})^{n_0}dadb.
\]
All together, the integral \eqref{E:main'' 3} becomes
\begin{align}\label{E:main'' 4}
\begin{split}
\int_{\b{P}''V_{\GL_{2n_0+1}}^\Diamond\backslash\GL_{2n+1}}
\varphi(g, I_k,I_{2n_0+1})
\int_{\b{X}^{n,r}}\int_{{\rm Mat}_{\ell, k}}\int_{\M_{k,\ell}}
f_{\xi_s}\left(
\pMX{\dot{x}}{}{}{\ddot{y}}\b{u}\jmath^{n,r}(g),I_r
\right)\psi^{-1}_{\b{X}^{n,r}}(\b{u})dxdyd\b{u}dg.
\end{split}
\end{align}

\subsubsection{}
Our next step is to factor the $dg$-integral through the rest of the unipotent subgroup $N_{\b{P}}$, namely, through the 
subgroup $N'_{\b{P}}\subset N_{\b{P}}$ of elements 
\[
u'
=
\begin{pmatrix}
I_k&&&&\\
0&I_{n_0}&&&\\
d_1&0&1&&\\
d_2&0&0&I_{n_0}&\\
e&f_2&f_1&0&I_k
\end{pmatrix}.
\]
To do so, let us write (cf. \S\ref{SSS:embedding})
\[
\jmath^{n,r}(u')
=
\begin{pmatrix}
L&&&\\
&I_\ell&&\\
&&I_\ell&\\
K&&&W
\end{pmatrix}
=
\begin{pmatrix}
L&&&\\
&I_\ell&&\\
&&I_\ell&\\
&&&W
\end{pmatrix}
\begin{pmatrix}
I_{n+1}&&&\\
&I_\ell&&\\
&&I_\ell&\\
W^{-1}K&&&I_{n+1}
\end{pmatrix}
\]
where
\[
L
=
\begin{pmatrix}
I_k&&\\
0&I_{n_0}&\\
d_1&0&1
\end{pmatrix},
\quad
W
=
\begin{pmatrix}
1&&\\
0&I_{n_0}&\\
-\frac{1}{2}f_1&0&I_k
\end{pmatrix}
\quad\text{and}\quad
K
=
\begin{pmatrix}
-d_1&0&0\\
d_2&0&0\\
e&f_2&\frac{1}{2}f_1
\end{pmatrix}.
\]
Notice that 
\[
Z:=W^{-1}K
=
\begin{pmatrix}
-d_1&0&0\\
d_2&0&0\\
e'&f_2&\frac{1}{2}f_1
\end{pmatrix}
\quad\text{with}\quad
e'=e-\frac{1}{2}f_1d_1.
\]
For convenient, we denote
\[
L^\vartriangle
=
\pMX{L}{}{}{I_\ell},
\quad
W^\triangledown
=
\pMX{I_\ell}{}{}{W}
\quad\text{and}\quad
Z^\flat
=
\pMX{0}{0}{Z}{0}
\]
all of which are contained in $\M_{r,r}$. In particular, we have
\[
\jmath^{n,r}(u')
=
\pMX{L^\vartriangle}{}{}{W^\triangledown}\pMX{I_r}{}{Z^\flat}{I_r}.
\]
On the other hand, every $\b{u}\in\b{X}^{n,r}$ can be written as
\[
\b{u}
=
\pMX{I_r}{}{D}{I_r}
\quad\text{with}\quad
D
=
\begin{pmatrix}
A&B\\
0&C\\
\end{pmatrix}
\]
where $A\in\M_{\ell,n+1}$, $B\in\M_{\ell,\ell}$ and $C\in\M_{n+1,\ell}$.\\

Now a simple calculation gives
\[
\b{u}\jmath^{n,r}(u')
=
\pMX{L^\vartriangle}{}{}{W^\triangledown}\pMX{I_r}{}{D'}{I_r}\pMX{I_r}{}{Z^\flat}{I_r}
\quad\text{with}\quad
D'=\pMX{AL}{B}{0}{W^{-1}C}.
\]
Since $(AL)_{n+1,\ell}=A_{n+1,\ell}$ and $(W^{-1}C)_{1,1}=C_{1,1}$, we find, after changing the variables 
$AL\mapsto A$ and $W^{-1}C\mapsto C$, that \eqref{E:main'' 4} equals to
\begin{align}\label{E:main'' 5}
\begin{split}
&\int_{N_{\b{P}}\hat{V}_{\G_k}V_{\GL_{2n_0+1}}^\Diamond\backslash\GL_{2n+1}}
\varphi(g, I_k,I_{2n_0+1})\int_{N'_{\b{P}}}
\int_{\b{X}^{n,r}}\int_{{\rm Mat}_{\ell, k}}\int_{\M_{k,\ell}}\\
&\quad\quad\quad f_{\xi_s}\left(
\pMX{\dot{x}}{}{}{\ddot{y}}\pMX{L^\vartriangle}{}{}{W^\triangledown}\b{u}\pMX{I_r}{}{Z^\flat}{I_r}\jmath^{n,r}(g),I_r
\right)\psi^{-1}_{\b{X}^{n,r}}(\b{u})dxdyd\b{u}du'dg.
\end{split}
\end{align}
Here we the fact that $N'_{\b{P}}\b{P}''=N_{\b{P}}\hat{V}_{\G_k}$ and we understand that $du'=dLdWdZ$.

\subsubsection{}
To go further, we compute
\[
\pMX{\dot{x}}{}{}{\ddot{y}}\pMX{L^\vartriangle}{}{}{W^\triangledown}\pMX{\dot{x}}{}{}{\ddot{y}}^{-1}
=
\pMX{\dot{x}L^\vartriangle\dot{x}^{-1}}{}{}{\ddot{y}W^\triangle\ddot{y}^{-1}}.
\]
By direct computations, we find that
\[
\dot{x}L^\vartriangle\dot{x}^{-1}
=
\begin{pmatrix}
I_{n_0+1}&-dx&d\\
&I_\ell&0\\
&&I_k
\end{pmatrix}
\quad\text{and}\quad
\ddot{y}W^\triangledown\ddot{y}^{-1}
=
\begin{pmatrix}
I_k&0&f\\
&I_\ell&yf\\
&&I_{n_0+1}
\end{pmatrix}
\]
where
\[
d
=
\begin{pmatrix}
0\\d_1
\end{pmatrix}\in\M_{n_0+1,k}
\quad\text{and}\quad
f
=
\begin{pmatrix}
-\frac{1}{2}f_1&0
\end{pmatrix}
\in\M_{k,n_0+1}.
\]
It follows that 
\begin{align*}
 &f_{\xi_s}\left(
\pMX{\dot{x}}{}{}{\ddot{y}}\pMX{L^\vartriangle}{}{}{W^\triangledown}\b{u}\pMX{I_r}{}{Z^\flat}{I_r}\jmath^{n,r}(g),I_r
\right)\\
&\quad\quad\quad=
 f_{\xi_s}\left(
\pMX{\dot{x}}{}{}{\ddot{y}}\b{u}\pMX{I_r}{}{Z^\flat}{I_r}\jmath^{n,r}(g),\dot{x}L^\vartriangle\dot{x}^{-1},
\ddot{y}W^\triangledown\ddot{y}^{-1}
\right)\\
&\quad\quad\quad\quad\quad\quad=
\psi\left(d_1x_1-\frac{1}{2}y_\ell f_1\right)
 f_{\xi_s}\left(
\pMX{\dot{x}}{}{}{\ddot{y}}\b{u}\pMX{I_r}{}{Z^\flat}{I_r}\jmath^{n,r}(g),I_r
\right)
\end{align*}
where $x_1$ is the first column of $x$, while $y_\ell$ is the last row of $y$. From this, we see that \eqref{E:main'' 5}
can be written as
\begin{align}\label{E:main'' 6}
\begin{split}
&\int_{N_{\b{P}}\hat{V}_{\G_k}V_{\GL_{2n_0+1}}^\Diamond\backslash\GL_{2n+1}}
\varphi(g, I_k,I_{2n_0+1})
\int_{N'_{\b{P}}}\int_{\b{X}^{n,r}}\int_{{\rm Mat}_{\ell, k}}\int_{\M_{k,\ell}}\\
&\quad\quad\quad 
\psi\left(d_1x_1-\frac{1}{2}y_\ell f_1\right)
f_{\xi_s}\left(
\pMX{\dot{x}}{}{}{\ddot{y}}\b{u}\pMX{I_r}{}{Z^\flat}{I_r}\jmath^{n,r}(g),I_r
\right)\psi^{-1}_{\b{X}^{n,r}}(\b{u})dxdyd\b{u}du'dg.
\end{split}
\end{align}

\subsubsection{}
We proceed to deal with the product
\[
\b{u}’
=
\pMX{\dot{x}}{}{}{\ddot{y}}\b{u}\pMX{I_r}{}{Z^\flat}{I_r}\pMX{\dot{x}}{}{}{\ddot{y}}^{-1}
\]
Again, we can write 
\[
\b{u}
=
\pMX{I_r}{}{D}{I_r}\in\b{X}^{n,r}
\]
so that 
\[
\b{u}'
=
\pMX{I_r}{}{H}{I_r}
\quad\text{with}\quad
H=\ddot{y}\left(D+Z^\flat\right)\dot{x}^{-1}=\ddot{y}D\dot{x}^{-1}+\ddot{y}Z^\flat\dot{x}^{-1}.
\]
We compute $H$ term by term. For the first one, let
\[
D
=
\begin{pmatrix}
A'&A&B\\
0&0&C\\
0&0&C'
\end{pmatrix}
\]
where $A'\in\M_{\ell,k}$, $A\in\M_{\ell,n_0+1}$, $B\in\M_{\ell,\ell}$, $C\in\M_{n_0+1,\ell}$ and $C'\in\M_{k,\ell}$. Then we
have 
\[
\ddot{y}D\dot{x}^{-1}
=
\begin{pmatrix}
0&C'&0\\
A&B'&A'\\
0&C&0
\end{pmatrix}
\quad\text{with}\quad
B'=B-A'x+yC'.
\]
Next, recall that 
\[
Z^\flat=\pMX{0}{0}{Z}{0}\in\M_{r,r}
\]
where
\[
Z
=
\begin{pmatrix}
-d_1&0&0\\
d_2&0&0\\
e'&f_2&\frac{1}{2}f_1
\end{pmatrix}\in\M_{n+1,n+1}
\quad\text{with}\quad
e'=e-\frac{1}{2}f_1d_1.
\]
We have
\[
\ddot{y}Z^\flat\dot{x}^{-1}
=
\begin{pmatrix}
f'&-e'x&e'\\
yf'&-ye'x&ye'\\
0&-d'x&d'
\end{pmatrix}
\]
where
\[
d'
=
\begin{pmatrix}
-d_1\\d_2
\end{pmatrix}\in\M_{n_0+1,k}
\quad\text{and}\quad
f'
=
\begin{pmatrix}
f_2&\frac{1}{2}f_1
\end{pmatrix}\in\M_{k,n_0+1}.
\]
Together, we find that 
\[
\b{u}'
=
\pMX{I_r}{}{H}{I_r}
\quad\text{with}\quad
H
=
\begin{pmatrix}
f'&C'-e'x&e'\\
A+yf'&B'-ye'x&A'+ye'\\
0&C-d'x&d'
\end{pmatrix}.
\]
Noticing that as $\b{u}$ and $u'$ vary, the subgroup of $\GL_{2r}$ consisting of the elements $\b{u}'$, with
$x$ and $y$ fixed, is precisely $\b{X}^{n_0,r}$. 

\subsubsection{}
Recall that the subgroup $\b{X}^{n_0,r}\subset\GL_{2r}$ consists of the elements
\[
\b{u}^0
=
\pMX{I_r}{}{H^0}{I_r}
\quad\text{with}\quad
H^0
=
\begin{pmatrix}
f&C'&e\\
A&B&A'\\
0&C&d
\end{pmatrix}
\]
where $A'\in\M_{\ell,k}$, $A\in\M_{\ell,n_0+1}$, $B\in\M_{\ell,\ell}$, $C\in\M_{n_0+1,\ell}$, $C'\in\M_{k,\ell}$,
$f\in\M_{k,n_0+1}$, $e\in\M_{k,k}$ and $d\in\M_{n_0+1,k}$. Moreover, the character $\psi_{\b{X}^{n_0,r}}$ of 
$\b{X}^{n_0,r}$ is given by 
\[
\psi_{\b{X}^{n_0,r}}\left(\b{u}^0\right)=\psi(A_{\ell,n_0+1}-C_{1,1}).
\]
To continue, we change the variable $\b{u}'\mapsto\b{u}^0$, which is achieved by changing $B'\mapsto B$, $e'\mapsto e$, 
and then 
\[
A+yf'\mapsto A,\quad A'+ye\mapsto A',\quad B-yex\mapsto B,\quad C-dx'\mapsto C,\quad C'-ex\mapsto C'.
\]
Noting that 
\[
\left(A+yf'\right)_{\ell,n_0+1}=A_{\ell,n_0+1}+\frac{1}{2}y_\ell f_1
\quad\text{and}\quad
\left(C-d'x\right)_{1,1}=C_{1,1}-d_1x_1
\]
so that $\psi^{-1}_{\b{X}^{n,r}}(\b{u})$ changes to
\[
\psi\left(-d_1x_1+\frac{1}{2}y_\ell f_1\right)
\psi^{-1}_{\b{X}^{n_0,r}}\left(\b{u}^0\right).
\]
Together, the integral \eqref{E:main'' 6} becomes
\begin{align}\label{E:main'' 7}
\begin{split}
\int_{N_{\b{P}}\hat{V}_{\G_k}V_{\GL_{2n_0+1}}^\Diamond\backslash\GL_{2n+1}}
\varphi(g, I_k,I_{2n_0+1})
&\int_{\b{X}^{n_0,r}}\int_{{\rm Mat}_{\ell, k}}\int_{\M_{k,\ell}}\\
&f_{\xi_s}\left(
\b{u}^0\pMX{\dot{x}}{}{}{\ddot{y}}\jmath^{n,r}(g),I_r
\right)\psi^{-1}_{\b{X}^{n_0,r}}\left(\b{u}^0\right)dxdyd\b{u}^0dg.
\end{split}
\end{align}
To complete the proof, it remains to factor the $dg$-integration through the subgroup $\GL_{2n_0+1}^\Diamond$ of 
$\GL_{2n+1}$. Let $g_0\in\GL_{2n_0+1}$, and hence
\[
\begin{pmatrix}
I_k&&\\
&g_0&\\
&&I_k
\end{pmatrix}\in\GL_{2n_0+1}^\Diamond.
\]
A simple calculation shows 
\[
\pMX{\dot{x}}{}{}{\ddot{y}}
\jmath^{n,r}\left(\begin{pmatrix}
I_k&&\\
&g_0&\\
&&I_k
\end{pmatrix}
\right)
=
\jmath^{n_0,r}(g_0)
\pMX{\dot{x}}{}{}{\ddot{y}}.
\]
Form this, we see that \eqref{E:main'' 7} can be made into the RHS of \eqref{E:main id for main''}, and the (formal) proof
follows.\qed 

\section{Multiplicativity of the Rankin-Selberg gamma factors: 2nd variable}\label{S:2nd}

\subsection{Preliminaries}
The objective of this section is to establish \thmref{T:main'''}. We will adhere to the notation and conventions introduced in 
\S\ref{SSS:rho_tau},  \S\ref{SS:notation} and \S\ref{SSS:convention}. Therefore, we are working with the \'etale algebra 
$E = F\,\x\,F$, and references to the field $F$ are omitted from the notation. While \thmref{T:main'''} is originally formulated 
for non-archimedean local fields, we should note that its proof is applicable to archimedean local fields as well. Hence, there
are no restrictions on the local field $F$. Nevertheless, it's essential to emphasize that when $F$ is archimedean, we need 
to assume $n \ge r-1$, as we can't define the gamma factors when $n < r-1$.
\subsubsection{Various realizations of $\rho_{\tau,s}$}
Let $\tau=\tau_1\boxtimes\tau_2$ be an irreducible generic representation of $\G_r=\GL_r\,\x\,\GL_r$. 
We assume that $\tau_j$ is the irreducible generic quotient of $I^{\GL_r}(\tau'_j, \tau''_j)$ for $j=1,2$, where 
$\tau'_1$ and $\tau'_2$ (resp. $\tau''_1$ and $\tau''_2$) are irreducible generic representations of $\GL_{r'}$ 
(resp. $\GL_{r''}$) for some integers $r'>0$ and $r''>0$ such that $r'+r''=r$. Since $\tau_j$ and $\tau'_j\x\tau''_j$ have the 
same Whittaker model for $j=1,2$, we may actually assume that
\[
\tau_1=I^{\GL_r}(\tau'_1,\tau''_1)
\quad\text{and}\quad
\tau_2=I^{\GL_r}(\tau'_2,\tau''_2).
\]
By induction in stages, we have
\[
\rho_{\tau,s}
=
I^{\GL_{2r}}(\tau_{1,s},\tau^*_{2,1-s})
\cong
I^{\GL_{2r}}(\tau'_{1,s},\rho_{\tau'',s},\tau'^*_{2,1-s})
\cong
I^{\GL_{2r}}\left(I^{\GL_r}(\tau'_{1,s},\tau''_{1,s}),I^{\GL_r}(\tau''^*_{2,1-s},\tau'^*_{2,1-s})\right)
\]
where $\tau'':=\tau''_1\boxtimes\tau''_2$ is an irreducible generic representation of $\G_{r''}$.\\

We proceed to describe elements in the underlying spaces of these induced representations. For this, 
let $P=M_P\ltimes N_P$ be the standard parabolic subgroup of $\GL_{2r}$ with the Levi subgroup
\[
M_P\cong\GL_{r'}\,\x\,\GL_{r''}\,\x\,\GL_{r''}\,\x\,\GL_{r'}.
\]
Let 
\[
\Upsilon
=
\tau'_1\bt\tau''_1\bt\tau''^*_2\bt\tau'^*_2
\]
be an irreducible generic representation of $M_P$, and $\nu_s$ be a character of $M_P$ given by 
\[
\nu_s(a,b,c,d)=|\det(a)|_F^{s-\frac{1}{2}}|\det(b)|_F^{s-\frac{1}{2}}|\det(c)|_F^{\frac{1}{2}-s}|\det(d)|_F^{\frac{1}{2}-s}.
\] 
Define a non-degenerate character $\psi'_{\GL_{2r}}$ of the maximal unipotent subgroup $V_{\GL_{2r}}$ of $\GL_{2r}$ by
\[
\psi'_{\GL_{2r}}(z)
=
\psi(z_{12}+\cdots+z_{r-1,r}-2z_{r,r+1}-z_{r+1,r+2}-\cdots-z_{2r-1, 2r}).
\]
Let $\psi'_{M_P}$ denote the restriction of $\psi'_{\GL_{2r}}$ to $M_P\cap V_{\GL_{2r}}$.\\

The underlying space of $I^{\GL_{2r}}(\tau'_{1,s},\rho_{\tau'',s},\tau'^*_{2,1-s})$ consists of the smooth functions
\[
\phi_s:\GL_{2r}\,\x\,\GL_{2r''}\,\x\,M_P\longto\mathbb{C}
\]
such that for $h\in\GL_{2r}$, $h',h''\in\GL_{2r''}$, $a,b\in\GL_{r'}$, $c,d\in\GL_{r''}$ and $m\in M_P$, 
\begin{itemize}
\item
$\phi_s\left(\begin{pmatrix}a&*&*\\&h'&*\\&&b\end{pmatrix}h,h'',m\right)
=|\det(ab^{-1})|_F^{\frac{r-r'}{2}}\phi_s\left(h,h'h'',m\begin{pmatrix}a&&\\&I_{2r''}&\\&&b\end{pmatrix}\right)$;\\
\item
$\phi_s\left(h,\pMX{c}{*}{}{d}h',m\right)
=|\det(cd^{-1})|_F^{-\frac{r''}{2}}
\phi_s\left(h,h',m\begin{pmatrix}I_{r'}&&&\\&c&&\\&&d&\\&&&I_{r'}\end{pmatrix}\right)$ and;
\item
for each fixed $(h,h'')\in\GL_{2r}\,\x\,\GL_{2r''}$, the function $m\mapsto \phi_s(h,h'',m)$ belongs to 
$\cW(\Upsilon, \psi'^{-1}_{M_P})$.
\end{itemize}
For each such $\phi_s$, define the smooth function
\[
\xi_s=\xi_{\phi_s}:\GL_{2r}\,\x\,\GL_r\,\x\,\GL_r\,\x\,M_P\to\mathbb{C}
\]
by 
\begin{equation}\label{E:phi to xi}
\xi_s(h,a_1,a_2,m)
=
\delta_{Q_{2r}}^{-\frac{1}{2}}\left(\pMX{a_1}{}{}{a_2}\right)\phi_s\left(\pMX{a_1}{}{}{a_2}h,I_{2r''},m\right)
\end{equation}
for $h\in\GL_{2r}$, $a_1,a_2\in\GL_r$ and $m\in M_P$. Then the underlying space of
$I^{\GL_{2r}}\left(I^{\GL_r}(\tau'_{1,s},\tau''_{1,s}),I^{\GL_r}(\tau''^*_{2,1-s},\tau'^*_{2,1-s})\right)$
consists of the functions $\xi_s$ as $\phi_s$ varies. Note that the $\xi_s$ here is the same as the one defined by 
\eqref{E:f_xi}.

\subsubsection{Intertwining maps and their factorizations}
Let $i,j$ be two positive integers, we denote 
\[
w_{i, j}
=
\pMX{}{I_j}{I_i}{}\in\GL_{i+j}
\quad\text{and}\quad 
N_{i,j}
=
\stt{
\pMX{I_i}{x}{}{I_j}\mid x\in\M_{i,j}
}.
\]
Note that $w_{i,j}^{-1}=w_{j,i}$. On the other hand, if $H$ is a subgroup of $\GL_r$, we write
\[
H^\vartriangle
=
\stt{
\pMX{h}{}{}{I_r}\mid h\in H
}
\quad\text{and}\quad
H^\triangledown
=
\stt{
\pMX{I_r}{}{}{h}\mid h\in\GL_r
}
\]
while if $H$ is a subgroup of $\GL_{2r''}$, we put
\[
H^\Diamond
=
\stt{
\begin{pmatrix}
I_{r'}&&\\
&h&\\
&&I_{r'}
\end{pmatrix}
\mid h\in H
}.
\]
These are subgroups of $\GL_{2r}$. In the following, we choose Haar measures on various unipotent subgroups of $\GL_N$ 
by defining them as product measures, where each root group is naturally isomorphic to $F$. Then on $F$, we use the 
measure that is self-dual with respect to $\psi_2$.\\

Recall that we have defined the intertwining map (cf. \S\ref{SSS:intertwining map}) $A(w_{r,r},\tau_1\bt\tau_2,s)$ in
\[
{\rm Hom}_{\GL_{2r}}
\left(
I^{\GL_{2r}}\left(I^{\GL_r}(\tau'_{1,s},\tau''_{1,s}),I^{\GL_r}(\tau''^*_{2,1-s},\tau'^*_{2,1-s})\right),
I^{\GL_{2r}}\left(I^{\GL_r}(\tau''^*_{2,1-s},\tau'^*_{2,1-s}),I^{\GL_r}(\tau'_{1,s},\tau''_{1,s})\right)
\right)
\]
and its normalized version $A_\psi(w_{r,r},\tau_1\bt\tau_2,s)$. Since $\tau_1$ and $\tau_2$ are induced representations, 
the normalized intertwining map $A_\psi(w_{r,r},\tau_1\bt\tau_2,s)$ can be written as a composition of certain (normalized) 
intertwining maps. This fact will be used in the proof of \thmref{T:main'''}. Let $\phi_s$ be an element in the underlying 
space of $I^{\GL_{2r}}(\tau'_{1,s},\rho_{\tau''_s},\tau'^*_{2,1-s})$ as above. Define an intertwining map
\[
A(w_{r'',r''},\tau''_1\bt\tau''_2,s)\in{\rm Hom}_{\GL_{2r''}^\Diamond}
\left(
I^{\GL_{2r}}(\tau'_{1,s},\rho_{\tau'',s},\tau'^*_{2,1-s}),
I^{\GL_{2r}}(\tau'_{1,s},\rho_{\tau''^*,1-s},\tau'^*_{2,1-s})
\right)
\]
by the integral 
\[
A\left(w_{r'',r''},\tau''_1\bt\tau''_2,s\right)\phi_s(h, h', m)
=
\int_{N_{r'',r''}}
\phi_s\left(h,w_{r'',r''}^{-1}uh', \hat{I}'_{2r''}\hat{w}_{r'',r''}^{-1}m\hat{w}_{r'',r''}\right)du
\]
for $\Re(s)\gg 0$, and by the meromorphic continuation in general, where $(h,h',m)\in\GL_{2r}\,\x\,\GL_{2r''}\,\x\,M_P$,
\[
\hat{w}_{r'',r''}
=
\begin{pmatrix}
I_{r'}&&\\
&w_{r'',r''}&\\
&&I_{r'}
\end{pmatrix}
\quad\text{and}\quad
\hat{I}'_{2r''}
=
\begin{pmatrix}
I_{r'}&&\\
&I'_{2r''}&\\
&&I_{r'}
\end{pmatrix}
\]
with $I'_{2r''}\in\GL_{2r''}$ given by \eqref{E:d_r}. The normalization $A_{\psi,\delta}(w_{r'',r''},\tau''_1\bt\tau''_2,s)$ of 
$A(w_{r'',r''},\tau''_1\bt\tau''_2,s)$ is defined to satisfy the following identity
\[
\int_{N_{r'',r''}}
\phi_s(h,uh',I_{M_P})\psi_2^{-1}(u_{r'',r''+1})du
=
\int_{N_{r'',r''}}
A_{\psi,\delta}(w_{r'',r''}, \tau''_1\bt\tau''_2,s)\phi_s(h,uh',I_{M_P})\psi^{-1}_2(u_{r'',r''+1})du.
\]\\

Next, let 
$\xi'_s$ be an element in the underlying space of 
$I^{\GL_{2r}}\left(I^{\GL_r}(\tau'_{1,s},\tau''^*_{2,1-s}),I^{\GL_r}(\tau''_{1,s},\tau'^*_{2,1-s})\right)$. We define
$A\left(w_{r',r''},\tau'_1\bt\tau''_2,s\right)$ to be the intertwining map in 
\[
{\rm Hom}_{\GL_r^\vartriangle}
\left(
I^{\GL_{2r}}\left(I^{\GL_r}(\tau'_{1,s},\tau''^*_{2,1-s}),I^{\GL_r}(\tau''_{1,s},\tau'^*_{2,1-s})\right),
I^{\GL_{2r}}\left(I^{\GL_r}(\tau''^*_{2,1-s},\tau'_{1,s}),I^{\GL_r}(\tau''_{1,s},\tau'^*_{2,1-s})\right)
\right)
\]
by the integral
\[
A\left(w_{r',r''},\tau'_1\bt\tau''_2,s\right)
\xi'_s(h,a_1,a_2,m)
=
\int_{N_{r'',r'}}
\xi'_s
\left(
h, w^{-1}_{r',r''}ua_1,a_2, \dot{w}_{r',r''}^{-1}m\dot{w}_{r',r''}
\right)du
\]
for $\Re(s)\gg 0$ and by meromorphic continuation in general, where $h\in\GL_{2r}$, $a_1,a_2\in\GL_r$, $m\in M_P$, and
\[
\dot{w}_{r',r''}
=
\pMX{w_{r',r''}}{}{}{I_r}.
\]
The normalized version $A_{\psi,\delta}\left(w_{r',r''},\tau'_1\bt\tau''_2,s\right)$ is designed to fit into the identity
\[
\int_{N_{r'',r'}}
\xi'_s\left(h,w_{r',r''}^{-1}u,I_r,I_{M_P})\psi(u_{r',r'+1}\right)dn
=
\int_{N_{r'',r'}}
A_{\psi,\delta}(w_{r',r''},\tau'_1\bt\tau''_2,s)
\xi'_s\left(h,w_{r',r''}^{-1}u,I_r,I_{M_P})\psi(u_{r',r'+1}\right)dn.
\]\\

Finally, let $A\left(w_{r'',r'},\tau''_1\bt\tau'_2,s\right)$ be the intertwining map in 
\[
{\rm Hom}_{\GL_r^\triangledown}
\left(
I^{\GL_{2r}}\left(I^{\GL_r}(\tau''^*_{2,1-s},\tau'_{1,s}),I^{\GL_r}(\tau''_{1,s},\tau'^*_{2,1-s})\right),
I^{\GL_{2r}}\left(I^{\GL_r}(\tau''^*_{2,1-s},\tau'_{1,s}),I^{\GL_r}(\tau'^*_{2,1-s},\tau''_{1,s})\right)
\right)
\]
given by the integral
\[
A\left(w_{r'',r'},\tau''_1\bt\tau'_2,s\right)
\xi_{\zeta'_s}(h,a_1,a_2,m)
=
\int_{N_{r',r''}}
\xi'_s
\left(
h, w^{-1}_{r'',r'}na_1,a_2, \ddot{w}_{r'',r'}^{-1}m\ddot{w}_{r'',r'}
\right)dn
\]
for $\Re(s)\gg 0$ and by meromorphic continuation in general, where $h\in\GL_{2r}$, $a_1,a_2\in\GL_r$, $m\in M_P$, and
\[
\ddot{w}_{r'',r'}
=
\pMX{I_r}{}{}{w_{r'',r'}}.
\]
Similarly,  $A_{\psi,\delta}\left(w_{r'',r'},\tau''_1\bt\tau'_2,s\right)$ is defined through the following identity
\begin{align*}
\int_{N_{r',r''}}
&\xi'_s\left(h,I_r,w_{r'',r'}^{-1}u,I_{M_P})\psi^{-1}(u_{r'',r''+1}\right)du\\
&\quad\quad\quad\quad\quad=
\int_{N_{r',r''}}
A_{\psi,\delta}(w_{r'',r'},\tau''_1\bt\tau'_2,s)
\xi'_s\left(h,I_r, w_{r'',r'}^{-1}u,I_{M_P})\psi^{-1}(u_{r'',r''+1}\right)du.
\end{align*}\\

Now we have (cf. \cite{Shahidi1981})
\[
A_{\psi,\delta}(w_{r,r},\tau_1\bt\tau_2,s)
=
A_{\psi,\delta}(w_{r',r'},\tau'_1\bt\tau'_2,s)
A_{\psi,\delta}(w_{r'',r'},\tau''_1\bt\tau'_2,s)
A_{\psi,\delta}(w_{r,r},\tau'_1\bt\tau''_2,s)
A_{\psi,\delta}(w_{r,r},\tau''_1\bt\tau''_2,s)
\]
where $\tau'=\tau'_1\bt\tau'_2$ and
\[
A_{\psi,\delta}(w_{r',r'},\tau'_1\bt\tau'_2,s)\in{\rm Hom}_{\GL_{2r'}^\Diamond}
\left(
I^{\GL_{2r}}(\tau''^*_{2,1-s},\rho_{\tau',s},\tau''_{1,s}),
I^{\GL_{2r}}(\tau''^*_{2,1-s},\rho_{\tau'^*,1-s},\tau''_{1,s})
\right)
\]
is defined analogously to $A_{\psi,\delta}(w_{r'',r''},\tau''_1\bt\tau''_2,s)$. In concrete terms, let $\phi_s$ be an element in the 
underlying space of $I^{\GL_{2r}}(\tau'_{1,s},\rho_{\tau'',s},\tau'^*_{2,1-s})$, and $\phi'_s$ be the element in 
the underlying space of $I^{\GL_{2r}}(\tau'_{1,s},\rho_{\tau'',1-s},\tau'^*_{2,1-s})$ such that 
\[
A_{\psi,\delta}(w_{r',r'},\tau'_1\bt\tau'_2,s)\phi_s=\phi'_s.
\]
Next, let $\phi''_s$ be the element in the underlying space of 
$I^{\GL_{2r}}\left(\tau''^*_{2,1-s},\rho_{\tau',s},\tau''_{1,s}\right)$ so that 
\[
A_{\psi,\delta}(w_{r'',r'},\tau''_1\bt\tau'_2,s)A_{\psi,\delta}(w_{r,r},\tau'_1\bt\tau''_2,s)\xi_{\phi'_s}
=
\xi_{\phi''_s}.
\]
Now if $\phi'''_s$ is the element in the underlying space of 
$I^{\GL_{2r}}\left(\tau''^*_{2,1-s},\rho_{\tau'^*,1-s},\tau''_{1,s}\right)$ with
\[
A_{\psi,\delta}(w_{r',r'},\tau'_1\bt\tau'_2,s)\phi''_s
=
\phi'''_s
\]
then we have
\[
A_{\psi,\delta}(w_{r,r},\tau_1\bt\tau_2,s)\xi_{\phi_s}
=
\xi_{\phi'''_s}.
\]

\subsection{Proof of \thmref{T:main'''}}
Let $\phi_s$ be an element in the underlying space of $I^{\GL_{2r}}(\tau'_{1,s},\rho_{\tau'',s},\tau'^*_{2,1-s})$, and 
$\xi_s=\xi_{\phi_s}$ 
given by the equation \eqref{E:phi to xi}. The aim is to establish the following identity
\begin{equation}\label{E:main id for main'''}
\Gamma_\delta(s,\pi\x\tau'',\psi)\Psi_{n,r}(v\ot\xi_s)=\Psi_{n,r}(v\ot\xi'_s)
\end{equation}
where $\xi'_s=\xi_{\phi'_s}$ 
is the one associated to $\phi'_s$ by the formula similar to \eqref{E:phi to xi} with $\phi'_s$ given by
\[
A_{\psi,\delta}(w_{r',r'},\tau'_1\bt\tau'_2,s)\phi_s=\phi'_s.
\]
Then we can apply the argument in the beginning of \cite[Section 8]{Kaplan2015} to conclude the proof. 
A careful reader might notice that we have implicitly use the fact that the Rankin-Selberg integrals and their associated 
analytic properties can be defined and hold for the induced representation 
$I^{\GL_{2r}}\left(I^{\GL_r}(\tau'_{1,s},\tau''_{1,s}),I^{\GL_r}(\tau''^*_{2,1-s},\tau'^*_{2,1-s})\right)$, despite the representation
$I^{\GL_r}(\tau'_{1,s},\tau''_{1,s})$ (resp. $I^{\GL_r}(\tau''^*_{2,1-s},\tau'^*_{2,1-s})$) is not of the form $\eta_s$ 
(resp. $\eta_{1-s}$), for some representation $\eta$ of $\GL_r$.\\

We briefly explain how to deduce \thmref{T:main'''} from \eqref{E:main id for main'''}. For this, let $\phi''_s$ be the element 
in the underlying space of $I^{\GL_{2r}}\left(\tau''^*_{2,1-s},\rho_{\tau',s},\tau''_{1,s}\right)$ such that 
\[
A_{\psi,\delta}(w_{r'',r'},\tau''_1\bt\tau'_2,s)A_{\psi,\delta}(w_{r,r},\tau'_1\bt\tau''_2,s)\xi_{\phi'_s}
=
\xi_{\phi''_s}
\]
and $\phi'''_s$ be the element in the underlying space of 
$I^{\GL_{2r}}\left(\tau''^*_{2,1-s},\rho_{\tau'^*,1-s},\tau''_{1,s}\right)$ so that
\[
A_{\psi,\delta}(w_{r',r'},\tau'_1\bt\tau'_2,s)\phi''_s
=
\phi'''_s.
\]
Then \eqref{E:main id for main'''} (with $r''$ being replaced by $r'$) gives
\[
\Gamma_\delta(s,\pi\x\tau',\psi)\Psi_{n,r}(v\ot\xi_{\phi''_s})=\Psi_{n,r}(v\ot\xi_{\phi'''_s}).
\]
Now since $A_\psi(w_{r,r},\tau_1\bt\tau_2,s)\xi_{\phi_s}=\xi_{\phi'''_s}$, we get that 
\begin{align*}
\Gamma_\delta(s,\pi\x\tau',\psi)
\Gamma_\delta(s,\pi\x\tau'',\psi)
\Psi_{n,r}(v\ot\xi_{\phi_s})
&=
\Gamma_\delta(s,\pi\x\tau',\psi)
\Psi_{n,r}(v\ot\xi_{{\phi'_s}})\\
&=
\Psi_{n,r}(v\ot\xi_{\phi'''_s})
=
\Psi_{n,r}(v\ot A_{\psi,\delta}(w_{r,r},\tau_1\bt\tau_2,s)\xi_{\phi_s})
\end{align*}
which implies \thmref{T:main'''}.\\ 

The proof of \eqref{E:main id for main'''} is divided into three cases depending on the sizes of $n,r$ and $r''$. Let
\begin{align}\label{E:f_xi 2}
\begin{split}
f_{\xi_s}(h,a_1,a_2)
&=
\int_{N_{r'',r'}}\int_{N_{r',r''}}
\xi_s\left(
h, w_{r',r''}^{-1}u_1a_1, w_{r'',r'}^{-1}u_2a_2, I_{M_P}
\right)
\psi\left((u_1)_{r',r'+1}-(u_2)_{r'',r''+1}\right)du_1du_2\\
&=
\int_{N_{P'}}
\delta^{-\frac{1}{2}}_P\left(\pMX{a_1}{}{}{a_2}\right)\phi_s\left(w^{-1}\pMX{a_1}{}{}{a_2}uh,I_{2r''}, I_{M_P}\right)
\psi^{-1}_{\hat{N}_{r'',r'}}(u)du
\end{split}
\end{align}
for $h\in\GL_{2r}$ and $a_1,a_2\in\GL_r$, where 
\[
w
=
\dot{w}_{r',r''}\ddot{w}_{r'',r'}
=
\pMX{w_{r',r''}}{}{}{w_{r'',r'}}
\quad\text{and}\quad
\h{N}_{r'',r'}
=
\stt{
\pMX{u_1}{}{}{u_2}\mid
u_1\in N_{r'',r'},\,u_2\in N_{r',r''}
}.
\]
Moreover, $\psi_{\h{N}_{r'',r'}}$ is the character of $\h{N}_{r'',r'}$ defined by 
\[
\psi_{\h{N}_{r'',r'}}(u)
=
\psi\left(-u_{r',r'+1}+u_{r+r'',r+r''+1}\right)
\]
for $u\in\h{N}_{r'',r'}$. Note that $f_{\xi_s}$ here is the same as the one given by \eqref{E:f_xi} when $E=F\,\x\,F$.

\subsection{Case $n<r''$}
We put $\ell''=r''-n-1\ge 0$. Notice that $n<r''$ implies $\ell=r-n-1=r'+\ell''\ge r'$. By \eqref{E:RS int n<r} and \eqref{E:f_xi 2},
the Rankin-Selberg integral $\Psi_{n,r}(v\ot\xi_s)$ is equal to
\begin{align*}
\int_{V_{\GL_{2n+1}}\backslash\GL_{2n+1}}
W_v(g)
\int_{\b{X}^{n,r}}
\psi^{-1}_{\b{X}^{n,r}}(\b{u})
\int_{\h{N}_{r'',r'}}
\phi_s
\left(
w^{-1}u\b{u}\jmath^{n,r}(g), I_{2r''}, I_{M_P}
\right)
\psi^{-1}_{\h{N}_{r'',r'}}(u)dud\b{u}dg.
\end{align*}
One checks directly that 
\[
\b{u}':=u\b{u}u^{-1}\in\b{X}^{n,r}
\quad\text{and}\quad
\psi_{\b{X}^{n,r}}(\b{u}')
=
\psi_{\b{X}^{n,r}}(\b{u})
\]
for $u\in N_P^w$ and $\b{u}\in\b{X}^{n,r}$. Since $d\b{u}'=d\b{u}$, we can switch the order of $u$ and $\b{u}$ in the integral 
above, and get
\begin{align*}
\int_{V_{\GL_{2n+1}}\backslash\GL_{2n+1}}
W_v(g)
\int_{\b{X}^{n,r}}
\psi^{-1}_{\b{X}^{n,r}}(\b{u})
\int_{\h{N}_{r'',r'}}
\phi_s
\left(
w^{-1}\b{u}u\jmath^{n,r}(g), I_{2r''}, I_{M_P}
\right)
\psi^{-1}_{\h{N}_{r'',r'}}(u)d\b{u}dudg.
\end{align*}
To proceed, let's compute $w^{-1}\b{u}w$ for $\b{u}\in\b{X}^{n,r}$. We can write
\[
\b{u}
=
\pMX{I_r}{0}{D}{I_r}
\quad\text{with}\quad
D=\pMX{A}{B}{D'}{C}
\]
for some $D'\in\M_{r'',r''}$, $A\in\M_{r',r''}$, $B\in\M_{r',r'}$ and $C\in\M_{r'',r'}$. Then simple calculations give
\[
w^{-1}\b{u}w
=
\begin{pmatrix}
I_{r'}&&&\\
0&I_{r''}&&\\
C&D'&I_{r''}&\\
B&A&0&I_{r'}
\end{pmatrix}
=
\begin{pmatrix}
I_{r'}&&\\
&\b{u}'&\\
&&I_{r'}
\end{pmatrix}
\begin{pmatrix}
I_{r'}&&&\\
0&I_{r''}&&\\
C&0&I_{r''}&\\
B&A&0&I_{r'}
\end{pmatrix}
\quad\text{where}\quad
\b{u}'
=
\pMX{I_{r''}}{0}{D'}{I_{r''}}\in\b{X}^{n,r''}
\]
and
\[
w^{-1}uw
=
\begin{pmatrix}
I_{r''}\\
x&I_{r'}\\
0&0&I_{r'}\\
0&0&y&I_{r''}
\end{pmatrix}
\quad\text{where}\quad
u
=
\begin{pmatrix}
I_{r''}&x&0&0\\
&I_{r'}&0&0\\
&&I_{r'}&y\\
&&&I_{r''}
\end{pmatrix}.
\]
It follows that 
\[
w^{-1}\b{u}uw
=
\begin{pmatrix}
I_{r'}&&\\
&\b{u}'&\\
&&I_{r'}
\end{pmatrix}
\begin{pmatrix}
I_{r'}&&&\\
x&I_{r''}&&\\
C&0&I_{r''}&\\
B'&A&y&I_{r'}
\end{pmatrix}
\quad\text{where}\quad
B'=B+Ax.
\]
At this point, let us put
\[
\b{n}
=
\b{n}(A,B,C,x,y)
=
\begin{pmatrix}
I_{r'}&&&\\
x&I_{r''}&&\\
C&0&I_{r''}&\\
B&A&y&I_{r'}
\end{pmatrix}
\]
which is an element in the unipotent radical $N_{\b{P}}$ of the opposite $\b{P}$ of $P$.

\subsubsection{}
Now if $\psi_{N_{\b{P}}}$ is the character of $N_{\b{P}}$ given by
\[
\psi_{N_{\b{P}}}(\b{n})
=
\begin{cases}
\psi\left(-x_{r'',1}+y_{r',1}\right)\quad&\text{if $\ell''>0$},\\
\psi\left(-x_{r'',1}+y_{r',1}+A_{r',r''}-C_{1,1}\right)\quad&\text{if $\ell''=0$},
\end{cases}
\]
then we have
\[
\psi_{\b{X}^{n,r}}(\b{u})\psi_{\h{N}_{r'',r'}}(u)
=
\psi_{\b{X}^{n,r''}}(\b{u}')\psi_{N_{\b{P}}}(\b{n}(A,B',C,X,Y))
\]
and hence the integral can be written as
\[
\int_{V_{\GL_{2n+1}}\backslash\GL_{2n+1}}W_v(g)
\int_{\b{X}^{n,r''}}\psi^{-1}_{\b{X}^{n,r''}}(\b{u}')
\int_{N_{\b{P}}}\phi_s\left(\b{n}w^{-1}\jmath^{n,r}(g), \b{u}', I_{M_P}\right)\psi^{-1}_{N_{\b{P}}}(\b{n})d\b{n}d\b{u}'dg
\]
after changing the variable $B'\mapsto B$. Since 
\[
w^{-1}\jmath^{n,r}(g)w=\jmath^{n,r''}(g)
\]
for $g\in\GL_{2n+1}$, and 
\begin{itemize}
\item
$N_{\b{P}}$ is normalized by $\jmath^{n,r''}(g)$;
\item
$d\b{n}'=d\b{n}$ where $\b{n}'=\jmath^{n,r''}(g)^{-1}\b{n}\jmath^{n,r''}(g)$ and;
\item
$\psi_{N_{\b{P}}}(\b{n}')=\psi_{N_{\b{P}}}(\b{n})$, 
\end{itemize}
the integral further becomes
\begin{align*}
&\int_{V_{\GL_{2n+1}}\backslash\GL_{2n+1}}W_v(g)
\int_{\b{X}^{n,r''}}\psi^{-1}_{\b{X}^{n,r''}}(\b{u}')
\int_{N_{\b{P}}}\phi_s\left(\b{n}\jmath^{n,r''}(g)w^{-1}, \b{u}', I_{M_P}\right)\psi^{-1}_{N_{\b{P}}}(\b{n})d\b{n}d\b{u}'dg\\
&=
\int_{V_{\GL_{2n+1}}\backslash\GL_{2n+1}}W_v(g)
\int_{\b{X}^{n,r''}}\psi^{-1}_{\b{X}^{n,r''}}(\b{u}')
\int_{N_{\b{P}}}\phi_s\left(\b{n}w^{-1},\b{u}'\jmath^{n,r}(g), I_{M_P}\right)\psi^{-1}_{N_{\b{P}}}(\b{n})d\b{n}d\b{u}'dg\\
&=
\int_{N_{\b{P}}}\psi^{-1}_{N_{\b{P}}}(\b{n})
\int_{V_{\GL_{2n+1}}\backslash\GL_{2n+1}}W_v(g)
\int_{\b{X}^{n,r''}}
\phi_s\left(\b{n}w^{-1},\b{u}'\jmath^{n,r}(g), I_{M_P}\right)\psi^{-1}_{\b{X}^{n,r''}}(\b{u}')d\b{u}'dgd\b{n}
\end{align*}
after changing the variable $\b{n}'\mapsto\b{n}$. As the inner integral
\[
\int_{V_{\GL_{2n+1}}\backslash\GL_{2n+1}}W_v(g)
\int_{\b{X}^{n,r''}}
\phi_s\left(\b{n}w^{-1},\b{u}'\jmath^{n,r}(g), I_{M_P}\right)\psi^{-1}_{\b{X}^{n,r''}}(\b{u}')d\b{u}'dg
\]
is the Rankin-Selberg integral attached to $\pi$ and $\tau''$ for each fixed $\b{n}\in N_{\b{P}}$, the functional equation 
implies that 
\[
\Gamma(s,\pi\x\tau'',\psi)\Psi_{n,r}(v\ot\xi_s)
\]     
can be transformed into the integral
\begin{align*}
\int_{N_{\b{P}}}\psi^{-1}_{N_{\b{P}}}(\b{n})
\int_{V_{\GL_{2n+1}}\backslash\GL_{2n+1}}W_v(g)
\int_{\b{X}^{n,r''}}
\phi'_s\left(\b{n}w^{-1},\b{u}'\jmath^{n,r}(g), I_{M_P}\right)\psi^{-1}_{\b{X}^{n,r''}}(\b{u}')d\b{u}'dgd\b{n},
\end{align*}
which is equal to $\Psi_{n,r}(v\ot\xi'_s)$ by the above derivations (with $\phi_s$ being replaced by $\phi'_s$). This verifies
\eqref{E:main id for main'''} when $n<r''$. 

\subsection{Case: $r''\le n<r$}
Note that in this case, we have $\ell=r-n-1=r'+(r''-n-1)<r'$. To simplify the notation, we write $n''=n-r''$.
As in the previous case, the Rankin-Selberg integral 
$\Psi_{n,r}(v\ot\xi_s)$ is equal to
\begin{align}\label{E:id for 0 step}
\int_{V_{\GL_{2n+1}}\backslash\GL_{2n+1}}
W_v(g)
\int_{\b{X}^{n,r}}
\psi^{-1}_{\b{X}^{n,r}}(\b{u})
\int_{\h{N}_{r'',r'}}
\phi_s
\left(
w^{-1}u\b{u}\jmath^{n,r}(g), I_{2r''}, I_{M_P}
\right)
\psi^{-1}_{\h{N}_{r'',r'}}(u)dud\b{u}dg.
\end{align}
The proof of \eqref{E:main id for main'''} when $r''\le n<r$ is actually quite lengthy, so we divide it into three steps.
Let $V'_{\GL_{2n+1}}$ be the subgroup of $V_{\GL_{2n+1}}$ consisting of the matrices of the form
\[
\begin{pmatrix}
z_1&0&0&*&*\\
&I_{n''}&0&0&*\\
&&1&0&0\\
&&&I_{n''}&0\\
&&&&z_2
\end{pmatrix}
\]
for $z_1,z_2\in V_{\GL_{r''}}$. 
On the other hand, let $\h{N}'_{r'',r'}$ and $\h{N}''_{r'',r'}$ be the subgroups of $\h{N}_{r'',r'}$ consisting of the matrices of 
the form
\[
\begin{pmatrix}
I_{r''}&0&*&0&0&0\\
&I_{n''+1}&0&0&0&0\\
&&I_{\ell}&0&0&0\\
&&&I_{\ell}&0&*\\
&&&&I_{n''+1}&0\\
&&&&&I_{r''}
\end{pmatrix}
\quad\text{and}\quad
\begin{pmatrix}
I_{r''}&*&0&0&0&0\\
&I_{n''+1}&0&0&0&0\\
&&I_{\ell}&0&0&0\\
&&&I_{\ell}&0&0\\
&&&&I_{n''+1}&*\\
&&&&&I_{r''}
\end{pmatrix}
\]
respectively. 

\subsubsection*{Step 1}
The first step is to show that $\Psi_{n,r}(v\ot\xi_s)$ can be written as
\begin{align}\label{E:id for 1 step}
\begin{split}
\int_{V'_{\GL_{2n+1}}V^\Diamond_{\GL_{2n''+1}}\backslash\GL_{2n+1}}W_v(g)
\int_{\b{X}^{n,r}}\psi^{-1}_{\b{X}^{n,r}}(\b{u})
\int_{\h{N}'_{r'',r'}}
\phi_s\left(w^{-1}u'\b{u}\jmath^{n,r}(g),I_{2r''}, I_{M_P}\right)du'd\b{u}dg.
\end{split}
\end{align}
The idea is to "get rid of $\h{N}''_{r'',r'}$". To do so, let $u\in\h{N}_{r'',r'}$ be given by
\[
u
=
\begin{pmatrix}
I_{r''}&x''&x'&0&0&0\\
&I_{n''+1}&0&0&0&0\\
&&I_{\ell}&0&0&0\\
&&&I_{\ell}&0&y'\\
&&&&I_{n''+1}&y''\\
&&&&&I_{r''}
\end{pmatrix}.
\]
We can write
\begin{equation}\label{E:u'u''}
u
=
\begin{pmatrix}
I_{r''}&0&x'&0&0&0\\
&I_{n''+1}&0&0&0&0\\
&&I_{\ell}&0&0&0\\
&&&I_{\ell}&0&y'\\
&&&&I_{n''+1}&0\\
&&&&&I_{r''}
\end{pmatrix}
\begin{pmatrix}
I_{r''}&x''&0&0&0&0\\
&I_{n''+1}&0&0&0&0\\
&&I_{\ell}&0&0&0\\
&&&I_{\ell}&0&0\\
&&&&I_{n''+1}&y''\\
&&&&&I_{r''}
\end{pmatrix}
=u'u''
\end{equation}
so that $u'\in\h{N}'_{r'',r'}$ and $u''\in\h{N}''_{r'',r'}$. We note that  $\b{u}':=u''\b{u}u''^{-1}$ is contained in $\b{X}^{n,r}$. 
Indeed, if
\begin{equation}\label{E:b{u}}
\b{u}
=
\begin{pmatrix}
I_{r''}&&&&&\\
0&I_{n''+1}&&&&\\
0&0&I_{\ell}&&&\\
A'&A''&B&I_{\ell}&&\\
0&0&C''&0&I_{n''+1}&\\
0&0&C'&0&0&I_{r''}
\end{pmatrix}
\end{equation}
then a direct computation gives 
\[
\b{u}'
=
\begin{pmatrix}
I_{r''}&&&&&\\
0&I_{n''+1}&&&&\\
0&0&I_{\ell}&&&\\
A'&A'''&B&I_{\ell}&&\\
0&0&C'''&0&I_{n''+1}&\\
0&0&C'&0&0&I_{r''}
\end{pmatrix}
\]
where $A'''=A''-A'x''$ and $C'''=C''+y''C'$. This shows that $\b{u}'\in\b{X}^{n,r}$. Now if we write 
\begin{equation}\label{E:x'' and y''}
x''=
(x''_1\,\,x''_2),
\quad
y''
=
\begin{pmatrix}
y''_1\\y''_2
\end{pmatrix}
\quad\text{and}\quad
A'
=
\begin{pmatrix}
A'_1\\A'_2
\end{pmatrix},
\quad
C'
=
(C'_1\,\,C'_2)
\end{equation}
where $A'_2$ (resp. $x''_2$) is the last row (resp. last column) of $A'$ (resp. $x''$) and $C'_1$ (resp. $y''_1$) is the 
first column (resp. first row) of $C'$ (resp. $y''$), and change the variable $\b{u}'\mapsto\b{u}$, then since $d\b{u}'=d\b{u}$ 
and $\psi_{\b{X}^{n,r}}(\b{u})$ changes to 
\[
\psi_{\b{X}^{n,r}}(\b{u})\psi\left(A'_2x''_2+y''_1C'_1\right)
\]
the integral \eqref{E:id for 0 step} becomes
\begin{align}\label{E:id for 1 step 1}
\begin{split}
\int_{V_{\GL_{2n+1}}\backslash\GL_{2n+1}}&W_v(g)
\int_{\b{X}^{n,r}}\psi^{-1}_{\b{X}^{n,r}}(\b{u})
\int_{\h{N}'_{r'',r'}}\int_{\h{N}''_{r'',r'}}\\
&\phi_s\left(
w^{-1}u'\b{u}u''\jmath^{n,r}(g), I_{2r''}, I_{M_P}\right)
\psi^{-1}_{\h{N}_{r'',r'}}(u'')\psi^{-1}\left(A'_2x''_2+y''_1C'_1\right)du'du''d\b{u}dg.
\end{split}
\end{align}
Here we use the fact that $\psi_{\h{N}_{r'',r'}}(u)=\psi_{\h{N}_{r'',r'}}(u'')$.

\subsubsection{}
To proceed, let 
\[
z
=
z(x''_1,x''_2,y''_1,y''_2)
=
\begin{pmatrix}
I_{r''}&x''_1&2x''_2&0&0\\
&I_{n''}&0&0&0\\
&&1&0&-y''_1\\
&&&I_{n''}&y''_2\\
&&&&I_{r''}
\end{pmatrix}
\in V_{\GL_{2n+1}}
\]
where $x''_i, y''_j$ for $i,j=1,2$ is as in \eqref{E:x'' and y''}. Then we have 
\[
\jmath^{n,r}(z)
=
\begin{pmatrix}
I_{r''}&x''_1&x''_2&0&0&-x''_2&0&0\\
&I_{n''}&0&0&0&0&0&0\\
&&1&0&0&0&0&-y''_1\\
&&&I_{\ell}&0&0&0&0\\
&&&&I_{\ell}&0&0&0\\
&&&&&1&0&y''_1\\
&&&&&&I_{n''}&y''_2\\
&&&&&&&I_{r''}
\end{pmatrix}
=
\pMX{I_r}{Z}{}{I_r}u''
\]
where $u''\in\h{N}''_{r'',r'}$ is as before (cf. \eqref{E:u'u''}), and 
\[
Z
=
Z(x''_2, y''_2)
=
\begin{pmatrix}
0&-x''_2&0&x''_2y''_1\\
0&0&0&0\\
0&0&0&-y''_1\\
0&0&0&0
\end{pmatrix}\in\M_{r,r}
\]
with the lower left corner $0\in\M_{\ell,\ell}$. Since 
\[
\psi_{\U_{2n+1}}(z)=\psi^{-1}_{N_{P'}}(u'')
\quad\text{and}\quad
u''\jmath^{n,r}(g)=\pMX{I_r}{-Z}{}{I_r}\jmath^{n,r}(zg)
\]
the integral \eqref{E:id for 1 step 1} can be written as
\begin{align}\label{E:id for 1 step 2}
\begin{split}
\int_{V_{\GL_{2n+1}}\backslash\GL_{2n+1}}&W_v(zg)
\int_{\b{X}^{n,r}}\psi^{-1}_{\b{X}^{n,r}}(\b{u})
\int_{\h{N}'_{r'',r'}}\int_{\M_{r'',n''}}\int_{\M_{n'',r''}}\int_{\M_{r'',1}}\int_{\M_{1,r''}}\\
&\phi_s\left(
w^{-1}u'\b{u}\pMX{I_r}{-Z}{}{I_r}\jmath^{n,r}(zg), I_{2r''}, I_{M_P}\right)
\psi^{-1}\left(A'_2x''_2+y''_1C'_1\right)dy''_1dx''_2dy''_2dx''_1du'd\b{u}dg.
\end{split}
\end{align}

\subsubsection{}
We continue to compute (with $\b{u}\in\b{X}^{n,r}$ being written as \eqref{E:b{u}})
\[
\pMX{I_r}{Z}{}{I_r}\b{u}\pMX{I_r}{-Z}{}{I_r}
=
\begin{pmatrix}
I_{n+1}&c&&\\
&I_{\ell}&&\\
&&I_{\ell}&a\\
&&&I_{n+1}
\end{pmatrix}
\begin{pmatrix}
I_{n+1}&&&\\
0&I_{\ell}&&\\
A&B'&I_{\ell}&\\
0&C&0&I_{n+1}
\end{pmatrix}
=
b\b{u}''
\]
where 
\[
A=(A'\,\,A''),\quad 
C=\begin{pmatrix}C''\\C'\end{pmatrix}
\quad\text{and}\quad
B'
=
B
+
A
\begin{pmatrix}
-x''_2&0&x''_2y''_1\\
0&0&0\\
0&0&-y''_1
\end{pmatrix}
C
\]
with the center $0\in\M_{n'',n''}$. Note that $\b{u}''\in\b{X}^{n,r}$. To explain $c$ and $a$, we further denote 
\[
A''=(A''_1\,\,A''_2)
\quad\text{and}\quad
C''
=
\begin{pmatrix}
C''_1\\C''_2
\end{pmatrix}
\]
where $A''_2$ is the last column of $A''$, and $C''_1$ is the first row of $C''$. Then
\[
c
=
\begin{pmatrix}
-x''_2C'+x''_2y''_1C''_1\\
0\\
-y''_1C'
\end{pmatrix}
\quad\text{and}\quad
a
=
\left(
A'x''_2\quad 0\quad A''_2x''_1-A'x''_2y''_1
\right)
\]
where the center zeros of $z$ and $t$ belong to $\M_{n'',\ell}$ and $\M_{\ell, n''}$, respectively. Now we decompose
\[
b
=
\begin{pmatrix}
I_{r''}&0&0&c'&0&0&0&0\\
&I_{n''}&0&0&0&0&0&0\\
&&1&0&0&0&0&0\\
&&&I_{\ell}&0&0&0&0\\
&&&&I_{\ell}&0&0&a'\\
&&&&&1&0&0\\
&&&&&&I_{n''}&0\\
&&&&&&&I_{r''}
\end{pmatrix}
\begin{pmatrix}
I_{r''}&0&0&0&0&0&0&0\\
&I_{n''}&0&0&0&0&0&0\\
&&1&c''&0&0&0&0\\
&&&I_{\ell}&0&0&0&0\\
&&&&I_{\ell}&a''&0&0\\
&&&&&1&0&0\\
&&&&&&I_{n''}&0\\
&&&&&&&I_{r''}
\end{pmatrix}
=
b'b''
\]
where
\[
c'=-x''_2C'+x''_2y''_1C''_1,
\quad
c''=-y''_1C'
\quad\text{and}\quad
a'=A''_2x''_1-A'x''_2y''_1,
\quad
a''=A'x''_2
\]
so that 
\[
\b{u}\pMX{I_r}{-Z}{}{I_r}
=
\pMX{I_r}{-Z}{}{I_r}b'b''\b{u}''.
\]
Then since 
\[
u'\pMX{I_r}{-Z}{}{I_r}=\pMX{I_r}{-Z}{}{I_r}u'
\quad
\text{and} 
\quad
w^{-1}\pMX{I_r}{-Z}{}{I_r}w
=
\begin{pmatrix}
I_{n''}&0&0&0&0&0&0&0\\
&1&0&0&y''_1&0&0&0\\
&&I_\ell&0&0&0&0&0\\
&&&I_{r''}&-x''_2y''_2&0&-x''_2&0\\
&&&&I_{r''}&0&0&0\\
&&&&&I_\ell&0&0\\
&&&&&&1&0\\
&&&&&&&I_{n''}
\end{pmatrix}
\]
we find that
\[
\phi_s\left(w^{-1}u'\b{u}\pMX{I_r}{-Z}{}{I_r}\jmath^{n,r}(zg), I_{2r''}, I_{M_P}\right)
=
\phi_s\left(w^{-1}u'b'b''\b{u}''\jmath^{n,r}(zg), I_{2r''}, I_{M_P}\right).
\]
As $b'\in N'_{P'}$, we can change the variables $\b{u}''\mapsto \b{u}$ and $u'b'\mapsto u'$ in the integral 
\eqref{E:id for 1 step 2} to obtain
\begin{align}\label{E:id for 1 step 3}
\begin{split}
\int_{V_{\GL_{2n+1}}\backslash\GL_{2n+1}}&W_v(zg)
\int_{\b{X}^{n,r}}\psi^{-1}_{\b{X}^{n,r}}(\b{u})
\int_{\h{N}'_{r'',r'}}\int_{\M_{r'',n''}}\int_{\M_{n'',r''}}\int_{\M_{r'',1}}\int_{\M_{1,r''}}\\
&\phi_s\left(
w^{-1}u'b''\b{u}\jmath^{n,r}(zg), I_{2r''}, I_{M_P}\right)
\psi^{-1}\left(A'_2x''_2+y''_1C'_1\right)dy''_1dx''_2dy''_2dx''_1du'd\b{u}dg.
\end{split}
\end{align}

\subsubsection{}
To accomplish the first step, observe that  
\[
u'b''=b''u'
\quad\text{and}\quad
w^{-1}b''w
=
\begin{pmatrix}
I_{n''}&0&0&0&0&0&0&0\\
&1&-y''_1C'&0&0&0&0&0\\
&&I_\ell&0&0&0&0&0\\
&&&I_{r''}&0&0&0&0\\
&&&&I_{r''}&0&0&0\\
&&&&&I_\ell&A'x''_2&0\\
&&&&&&1&0\\
&&&&&&&I_{n''}
\end{pmatrix}.
\]
These imply
\[
\phi_s\left(w^{-1}u'b''\b{u}\jmath^{n,r}(zg), I_{2r''}, I_{M_P}\right)\psi^{-1}\left(A'_2x''_2+y''_1C'_1\right)
=
\phi_s\left(w^{-1}u'\b{u}\jmath^{n,r}(zg), I_{2r''}, I_{M_P}\right)
\]
and hence the integral \eqref{E:id for 1 step 3} can be further changed into 
\[
\int_{V'_{\GL_{2n+1}}V^\Diamond_{\GL_{2n''+1}}\backslash\GL_{2n+1}}W_v(g)
\int_{\b{X}^{n,r}}\psi^{-1}_{\b{X}^{n,r}}(\b{u})
\int_{\h{N}'_{r'',r'}}
\phi_s\left(w^{-1}u'\b{u}\jmath^{n,r}(g),I_{2r''}, I_{M_P}\right)du'd\b{u}dg
\]
which is what we want. This finishes the first step.

\subsubsection*{Step 2}
The second step is show that the integral \eqref{E:id for 1 step} obtained in the first step can be written as 
\begin{align}\label{E:id for 2 step}
\begin{split}
\int_{V'_{\GL_{2n+1}}V^\Diamond_{\GL_{2n''+1}}\b{X}_{n,r''}\backslash\GL_{2n+1}}\int_{\b{X}_{n,r''}}&W_v(\b{x}g)
\int_{\b{X}^{n,r}}\psi^{-1}_{\b{X}^{n,r}}(\b{u})\\
&\int_{\h{N}'_{r'',r'}}
\phi_s\left(w^{-1}u'\b{u}\jmath^{n,r}(g),I_{2r''}, I_{M_P}\right)du'd\b{u}d\b{x}dg.
\end{split}
\end{align}
Here the $dg$ integration should be understood in the sense of the Iwasawa decomposition. We first factor the 
$dg$ integration in \eqref{E:id for 1 step} through $\b{X}_{n,r''}$ to obtain
\begin{align}\label{E:id for 2 step 1}
\begin{split}
\int_{V'_{\GL_{2n+1}}V^\Diamond_{\GL_{2n''+1}}\b{X}_{n,r''}\backslash\GL_{2n+1}}\int_{\b{X}_{n,r''}}&W_v(\b{x}g)
\int_{\b{X}^{n,r}}\psi^{-1}_{\b{X}^{n,r}}(\b{u})\\
&\int_{\h{N}'_{r'',r'}}
\phi_s\left(w^{-1}u'\b{u}\jmath^{n,r}(\b{x})\jmath^{n,r}(g),I_{2r''}, I_{M_P}\right)du'd\b{u}d\b{x}dg.
\end{split}
\end{align}
To proceed, we claim that 
\[
\b{u}':=\jmath^{n,r}(\b{x})^{-1}\b{u}\jmath^{n,r}(\b{x})\in\b{X}^{n,r}
\quad\text{and}\quad
\psi_{\b{X}^{n,r}}(\b{u}')=\psi_{\b{X}^{n,r}}(\b{u}).
\]
For this, let us denote
\[
\b{x}
=
\begin{pmatrix}
I_{r''}\\
x_1&I_{n''}\\
0&0&1\\
0&0&0&I_{n''}\\
0&0&0&x_2&I_{r''}
\end{pmatrix}
\in\b{X}_{n,r''}.
\]
Then we have 
\[
\jmath^{n,r}(\b{x})
=
\begin{pmatrix}
X_1\\
&I_{\ell}\\
&&I_{\ell}\\
&&&X_2
\end{pmatrix}
\quad\text{with}\quad
X_1
=
\begin{pmatrix}
I_{r''}\\
x_1&I_{n''}\\
0&0&1
\end{pmatrix}
\quad\text{and}\quad
X_2
=
\begin{pmatrix}
1\\
0&I_{n''}\\
0&x_2&1_{r''}
\end{pmatrix}.
\]
On the other hand, recall that 
\[
\b{u}
=
\begin{pmatrix}
I_{n+1}\\
0&I_\ell\\
A&B&I_\ell\\
0&C&0&I_{n+1}
\end{pmatrix}
\quad\text{and}\quad
\psi_{\b{X}^{n,r}}(\b{u})
=
\psi\left(A_{\ell,n+1}-C_{1,1}\right).
\]
One can verify that 
\[
\b{u}'
=
\begin{pmatrix}
I_{n+1}\\
0&I_\ell\\
AX_1&B&I_\ell\\
0&X^{-1}_2C&0&I_{n+1}
\end{pmatrix}
\quad\text{and}\quad
(AX_1)_{\ell,n+1}=A_{\ell,n+1},\quad (X_2^{-1}C)_{1,1}=C_{1,1}.
\]
These prove the claim, and hence the integral \eqref{E:id for 2 step 1} can be altered to
\begin{align}\label{E:id for 2 step 2}
\begin{split}
\int_{V'_{\GL_{2n+1}}V^\Diamond_{\GL_{2n''+1}}\b{X}_{n,r''}\backslash\GL_{2n+1}}\int_{\b{X}_{n,r''}}&W_v(\b{x}g)
\int_{\b{X}^{n,r}}\psi^{-1}_{\b{X}^{n,r}}(\b{u})\\
&\int_{\h{N}'_{r'',r'}}
\phi_s\left(w^{-1}u'\jmath^{n,r}(\b{x})\b{u}\jmath^{n,r}(g),I_{2r''}, I_{M_P}\right)du'd\b{u}d\b{x}dg
\end{split}
\end{align}
after changing the variable $\b{u}'\mapsto\b{u}$ and employing the fact that $d\b{u}'=d\b{u}$.

\subsubsection{}
We continue to compute (with $u'\in N'_{P'}$ as \eqref{E:u'u''})
\[
\jmath^{n,r}(\b{x})^{-1}u'\jmath^{n,r}(\b{x})
=
\begin{pmatrix}
I_{r''}&0&0&0&0&0&0&0\\
&I_{n''}&0&-x_1x'&0&0&0&0\\
&&1&0&0&0&0&0\\
&&&I_\ell&0&0&0&0\\
&&&&I_\ell&0&y'x_2&0\\
&&&&&1&0&0\\
&&&&&&I_{n''}&0\\
&&&&&&&I_{r''}
\end{pmatrix}
u'.
\]
At this point, let us denote 
\[
z_1
=
\begin{pmatrix}
I_{r''}&0&0&0\\
&I_{n''}&0&-x_1x'\\
&&1&0\\
&&&I_\ell
\end{pmatrix}
\quad\text{and}\quad
z_2
=
\begin{pmatrix}
I_\ell&0&y'x_2&0\\
&1&0&0\\
&&I_{n''}&0\\
&&&I_{r''}
\end{pmatrix}
\]
so that 
\[
u'\jmath^{n,r}(\b{x})
=
\jmath^{n,r}(\b{x})\pMX{z_1}{}{}{z_2}u'.
\]
On the other hand, a simple calculation shows
\[
w^{-1}\jmath^{n,r}(\b{x})w
=
\begin{pmatrix}
I_{n''}&0&0&x_1&0&0&0&0\\
&1&0&0&0&0&0&0\\
&&I_\ell&0&0&0&0&0\\
&&&I_{r''}&0&0&0&0\\
&&&&I_{r''}&0&0&x_2\\
&&&&&I_\ell&0&0\\
&&&&&&1&0\\
&&&&&&&I_{n''}
\end{pmatrix}
\]
and this implies
\[
\phi_s\left(w^{-1}u'\jmath^{n,r}(\b{x})\b{u}\jmath^{n,r}(g),I_{2r''}, I_{M_P}\right)
=
\phi_s\left(w^{-1}\pMX{z_1}{}{}{z_2}u'\b{u}\jmath^{n,r}(g),I_{2r''}, I_{M_P}\right).
\]
Since 
\[
w^{-1}\pMX{z_1}{}{}{z_2}w
=
\begin{pmatrix}
I_{n''}&0&-x_1x'&0&0&0&0&0\\
&1&0&0&0&0&0&0\\
&&I_\ell&0&0&0&0&0\\
&&&I_{r''}&0&0&0&0\\
&&&&I_{r''}&0&0&0\\
&&&&&I_\ell&0&y'x_2\\
&&&&&&1&0\\
&&&&&&&I_{n''}
\end{pmatrix}
\]
we further get that 
\[
\phi_s\left(w^{-1}\pMX{z_1}{}{}{z_2}u'\b{u}\jmath^{n,r}(g),I_{2r''}, I_{M_P}\right)
=
\phi_s\left(w^{-1}u'\b{u}\jmath^{n,r}(g),I_{2r''}, I_{M_P}\right).
\]
From these, we see that the last integral in \eqref{E:id for 2 step 2} becomes
\begin{align*}
\int_{V'_{\GL_{2n+1}}V^\Diamond_{\GL_{2n''+1}}\b{X}_{n,r''}\backslash\GL_{2n+1}}\int_{\b{X}_{n,r''}}&W_v(\b{x}g)
\int_{\b{X}^{n,r}}\psi^{-1}_{\b{X}^{n,r}}(\b{u})\\
&\int_{\h{N}'_{r'',r'}}
\phi_s\left(w^{-1}u'\b{u}\jmath^{n,r}(g),I_{2r''}, I_{M_P}\right)du'd\b{u}d\b{x}dg
\end{align*}
which is exactly the integration \eqref{E:id for 2 step}. This completes the second step.\\

\subsubsection*{Step 3}
The last step is to establish \eqref{E:main id for main'''}. We first factor the $dg$ integration in \eqref{E:id for 2 step} through 
$\jmath_{n,r''}(\GL_{2r''})$ to obtain
\begin{align}\label{E:id for 3 step 1}
\begin{split}
\int_{V''_{\GL_{2n+1}}V^\Diamond_{\GL_{2n''+1}}\b{X}_{n,r''}\jmath_{n,r''}(\GL_{2r''})\backslash\GL_{2n+1}}
&\int_{V_{\GL_{2r''}}\backslash\GL_{2r''}}
\int_{\b{X}_{n,r''}}
W_v(\b{x}\jmath_{n,r''}(h)g)\int_{\b{X}^{n,r}}\psi^{-1}_{\b{X}^{n,r}}(\b{u})\\
&\quad\int_{\h{N}'_{r'',r'}}
\phi_s\left(w^{-1}u'\b{u}\jmath^{n,r}(\jmath_{n,r''}(h)g),I_{2r''}, I_{M_P}\right)du'd\b{u}d\b{x}dhdg,
\end{split}
\end{align}
where $V''_{\GL_{2n+1}}$ is the subgroup of $V'_{\GL_{2n+1}}$ consisting of matrices of the form
\[
\begin{pmatrix}
I_{r''}&0&0&*&0\\
&I_{n''}&0&0&*\\
&&1&0&0\\
&&&I_{n''}&0\\
&&&&I_{r''}
\end{pmatrix}.
\]
Again, the $dg$ integration in \eqref{E:id for 3 step 1} should be understood in the sense of the Iwasawa decomposition. 
We then compute (with $u'$ and $\b{u}$ given by \eqref{E:u'u''} and \eqref{E:b{u}} respectively)
\[
w^{-1}u'w
=
\begin{pmatrix}
I_{n''+1}\\
0&I_\ell\\
0&x'&I_{r''}\\
0&0&0&I_{r''}\\
0&0&0&y'&I_\ell\\
0&0&0&0&0&I_{n''+1}
\end{pmatrix}
\quad\text{and}\quad
w^{-1}\b{u}w
=
\begin{pmatrix}
I_{n''+1}\\
0&I_\ell\\
0&0&I_{r''}\\
0&C'&0&I_{r''}\\
A''&B&A'&0&I_\ell\\
0&C''&0&0&0&I_{n''+1}
\end{pmatrix}.
\]
These imply
\begin{equation}\label{E:b{Y}}
w^{-1}u'\b{u}w
=
\begin{pmatrix}
I_{n''+1}\\
0&I_\ell\\
0&x'&I_{r''}\\
0&C'&0&I_{r''}\\
A''&B'&A'&y'&I_\ell\\
0&C''&0&0&0&I_{n''+1}
\end{pmatrix}
\quad\text{with}\quad
B'=B+y'C'.
\end{equation}
At this point, let us denote $\b{Y}$ as the subgroup of $\GL_{2r}$ consisting of the matrices similar to those in \eqref{E:b{Y}}, 
but with $B'$ replaced by arbitrary matrices $B\in\M_{\ell,\ell}$. We also adapt $\psi_{\b{X}^{n,r}}$ to be the character 
$\psi_{\b{Y}}$ of $\b{Y}$. More concretely, if $\b{y}\in\b{Y}$ is given by \eqref{E:b{Y}} (with $B'$ being replaced by $B$), then
\[
\psi_{\b{Y}}(\b{y})
=
\psi\left(
(A'')_{n''+1, \ell}-(C'')_{1,1}
\right).
\]
With these, we see that \eqref{E:id for 3 step 1} becomes
\begin{align}\label{E:id for 3 step 2}
\begin{split}
&\int_{V'_{\GL_{2n+1}}V^\Diamond_{\GL_{2n''+1}}\b{X}_{n,r''}\jmath_{n,r''}(\GL_{2r''})\backslash\GL_{2n+1}}
\int_{\b{Y}}\psi^{-1}_{\b{Y}}(\b{y})\\
&\quad\quad\quad\quad
\int_{V_{\GL_{2r''}}\backslash\GL_{2r''}}
\int_{\b{X}_{n,r''}}
W_v(\b{x}\jmath_{n,r''}(h)g)
\phi_s\left(\b{y}w^{-1}\jmath^{n,r}(\jmath_{n,r''}(h)g),I_{2r''}, I_{M_P}\right)d\b{x}dhd\b{y}dg.
\end{split}
\end{align}

\subsubsection{}
To proceed, let $h=\pMX{a}{b}{c}{d}\in\GL_{2r''}$ with $a,b,c,d\in\M_{r'',r''}$. We have
\[
\jmath^{n,r}\left(\jmath_{n,r''}(h)\right)
=
\begin{pmatrix}
a&&&b\\
&I_{r'}&&\\
&&I_{r'}&\\
c&&&d
\end{pmatrix}
\]
and a simple calculation gives
\[
w^{-1}\jmath^{n,r}\left(\jmath_{n,r''}(h)\right)w
=
\begin{pmatrix}
I_{r'}\\
&h\\
&&I_{r'}
\end{pmatrix}.
\]
Now the key observation is that the subgroup $\b{Y}$, the character $\psi_{\b{Y}}$ and the Haar measure $d\b{y}$ are
invariant under the conjugation of $w^{-1}\jmath^{n,r}\left(\jmath_{n,r''}(h)\right)w$. It follows that 
\[
\phi_s\left(\b{y}w^{-1}\jmath^{n,r}(\jmath_{n,r''}(h)g),I_{2r''}, I_{M_P}\right)
=
\phi_s\left(\b{y}w^{-1}\jmath^{n,r}(g),h, I_{M_P}\right)
\]
and hence the integral \eqref{E:id for 3 step 2} can be changed to
\begin{align}\label{E:id for 3 step 3}
\begin{split}
&\int_{V'_{\GL_{2n+1}}V^\Diamond_{\GL_{2n''+1}}\b{X}_{n,r''}\jmath_{n,r''}(\GL_{2r''})\backslash\GL_{2n+1}}
\int_{\b{Y}}\psi^{-1}_{\b{Y}}(\b{y})\\
&\quad\quad\quad\quad
\int_{V_{\GL_{2r''}}\backslash\GL_{2r''}}
\int_{\b{X}_{n,r''}}
W_v(\b{x}\jmath_{n,r''}(h)g)
\phi_s\left(\b{y}w^{-1}\jmath^{n,r}(g),h, I_{M_P}\right)d\b{x}dhd\b{y}dg.
\end{split}
\end{align}
Since the inner integral
\[
\int_{V_{\GL_{2r''}}\backslash\GL_{2r''}}
\int_{\b{X}_{n,r''}}
W_v(\b{x}\jmath_{n,r''}(h)g)
\phi_s\left(\b{y}w^{-1}\jmath^{n,r}(g),h, I_{M_P}\right)d\b{x}dh
\]
in \eqref{E:id for 3 step 3} represents the Rankin-Selberg integral attached to $\pi$ and $\tau''$ for every fixed 
$g\in\GL_{2n+1}$ and $\b{y}\in\b{Y}$, we conclude that \eqref{E:main id for main'''} holds when $r''\le n<r$ by the same 
argument used in the proof of the case $n<r''$.

\subsection{Case $r\le n$}
By \eqref{E:RS int n>r} and \eqref{E:f_xi 2}, the Rankin-Selberg integral $\Psi_{n,r}(v\ot\xi_s)$ is equal to
\begin{align}\label{E:id for n>r 1}
\int_{V_{\GL_{2r}}\backslash\GL_{2r}}
\int_{\b{X}_{n,r}}
\int_{\h{N}_{r'',r'}}
W_v\left(\b{x}\jmath_{n,r}(h)\right)
\phi_s\left(w^{-1}uh, I_{2r''}, I_{M_P}\right)\psi^{-1}_{\h{N}_{r'',r'}}(u)dud\b{x}dh.
\end{align}
Notice that $\jmath_{n,r}(u)\in V_{\GL_{2n+1}}$ for $u\in\h{N}_{r'',r'}$ and 
$\psi_{\U_{2n+1}}(\jmath_{n,r}(u))=\psi^{-1}_{\h{N}_{r'',r'}}(u)$; it follows that
\[
W_v(\jmath_{n,r}(u)h)
=
\psi_{\U_{2n+1}}(\jmath_{n,r}(u))W_v(h)
=
\psi^{-1}_{\h{N}_{r'',r'}}(u)W_v(u)
\]
and hence the integral \eqref{E:id for n>r 1} can be written as
\begin{align}\label{E:id for n>r 2}
\int_{V_{\GL_{2r}}\backslash\GL_{2r}}
\int_{\b{X}_{n,r}}
\int_{\h{N}_{r'',r'}}
W_v\left(\jmath_{n,r}(u)\b{x}\jmath_{n,r}(h)\right)
\phi_s\left(w^{-1}uh, I_{2r''}, I_{M_P}\right)dud\b{x}dh.
\end{align}
We then claim that 
\[
\b{x}':=\jmath_{n,r}(u)\b{x}\jmath_{n,r}(u)^{-1}\in\b{X}_{n,r}
\quad\text{and}\quad
d\b{x}'=d\b{x}.
\]
For this, let us denote 
\begin{equation}\label{E:b{x}}
\b{x}
=
\begin{pmatrix}
I_{r''}\\
0&I_{r'}\\
A_1&A_1&I_{n-r}\\
0&0&0&1\\
0&0&0&0&I_{n-r}\\
0&0&0&0&B_2&I_{r'}\\
0&0&0&0&B_1&0&I_{r''}
\end{pmatrix}
\quad\text{and}\quad
u
=
\begin{pmatrix}
I_{r''}&x&0&0&0&0&0\\
&I_{r'}&0&0&0&0&0\\
&&I_{n-r}&0&0&0&0\\
&&&1&0&0&0\\
&&&&I_{n-r}&0&0\\
&&&&&I_{r'}&y\\
&&&&&&I_{r''}
\end{pmatrix}.
\end{equation}
Then we have 
\[
\b{x}'
=
\begin{pmatrix}
I_{r''}\\
0&I_{r'}\\
A_1&A'_2&I_{n-r}\\
0&0&0&1\\
0&0&0&0&I_{n-r}\\
0&0&0&0&B'_2&I_{r'}\\
0&0&0&0&B_1&0&I_{r''}
\end{pmatrix}
\quad\text{with}\quad
A'_2=A_2-A_1x
\quad\text{and}\quad
B'_2=B_2+yB_1.
\]
This proves the claim. From this, we see that the integral \eqref{E:id for n>r 2} can be altered to
\begin{align}\label{E:id for n>r 3}
\begin{split}
&\int_{V_{\GL_{2r}}\backslash\GL_{2r}}
\int_{\b{X}_{n,r}}
\int_{\h{N}_{r'',r'}}
W_v\left(\b{x}\jmath_{n,r}(uh)\right)
\phi_s\left(w^{-1}uh, I_{2r''}, I_{M_P}\right)dud\b{x}dh\\
&\quad\quad\quad\quad=
\int_{V_{M_{P'}}N_{Q_{2r}}\backslash\GL_{2r}}
\int_{\b{X}_{n,r}}
W_v\left(\b{x}\jmath_{n,r}(h)\right)
\phi_s\left(w^{-1}h, I_{2r''}, I_{M_P}\right)d\b{x}dh
\end{split}
\end{align}
after changing the variable $\b{x}'\mapsto\b{x}$. Here $V_{M_{P'}}:=M_{P'}\cap V_{\GL_{2r}}$ and we recall that 
$Q_{2r}$ (resp. $P'$) is the standard parabolic subgroup of $\GL_{2r}$ whose Levi subgroup 
$M_{Q_{2r}}\cong\GL_r\,\x\,\GL_r$ (resp. $\GL_{r''}\,\x\,\GL_{r'}\,\x\,\GL_{r'}\,\x\,\GL_{r''}$).

\subsubsection{}
At this point, let us put $\check{N}_{r'',r'}$ to be the subgroup of $\GL_{2r}$ consisting of the matrices of the form
\[
\begin{pmatrix}
I_{r''}\\
*&I_{r'}\\
0&0&I_{r'}\\
0&0&*&I_{r''}
\end{pmatrix}.
\]
We then factor the $dh$ integration in \eqref{E:id for n>r 3} through $\check{N}_{r'',r'}$ to obtain
\begin{align}\label{E:id for n>r 4}
\int_{V_{M_{P'}}N_{Q_{2r}}\check{N}_{r'',r'}\backslash\GL_{2r}}
\int_{\b{X}_{n,r}}
\int_{\check{N}_{r'',r'}}
W_v\left(\b{x}\jmath_{n,r}(\b{y}h)\right)
\phi_s\left(w^{-1}\b{y}h, I_{2r''}, I_{M_P}\right)d\b{y}d\b{x}dh.
\end{align}
As before, the $dh$ integration should be understood in the sense of the Iwasawa decomposition. Note that if 
\[
\b{y}
=
\begin{pmatrix}
I_{r''}\\
y_1&I_{r'}\\
0&0&I_{r'}\\
0&0&y_2&I_{r''}
\end{pmatrix}\in\check{N}_{r'',r'}
\]
then 
\[
w^{-1}\b{y}w
=
\begin{pmatrix}
I_{r''}&y_1&0&0\\
0&I_{r'}&0&0\\
&&I_{r'}&y_2\\
&&&I_{r''}
\end{pmatrix}.
\]
This implies
\[
\phi_s\left(w^{-1}\b{y}h, I_{2r''}, I_{M_P}\right)
=
\phi_s\left(w^{-1}h, I_{2r''}, I_{M_P}\right).
\]
so that the integral \eqref{E:id for n>r 4} becomes
\begin{align}\label{E:id for n>r 5}
\int_{V_{M_{P'}}N_{Q_{2r}}\check{N}_{r'',r'}\backslash\GL_{2r}}
\int_{\b{X}_{n,r}}
\int_{\check{N}_{r'',r'}}
W_v\left(\b{x}\jmath_{n,r}(\b{y}h)\right)
\phi_s\left(w^{-1}h, I_{2r''}, I_{M_P}\right)d\b{y}d\b{x}dh.
\end{align}

\subsubsection{}
We proceed to compute $\b{x}\jmath_{n,r}(\b{y}y)$. Let $\b{x}\in\b{X}_{n,r}$ be as in \eqref{E:b{x}} and
 $\b{y}\in\check{N}_{r'',r'}$ be as above. We have
\[
\jmath_{n,r}(\b{y})
=
\begin{pmatrix}
I_{r''}\\
y_1&I_{r'}\\
0&0&I_{n-r}\\
0&0&0&1\\
0&0&0&0&I_{n-r}\\
0&0&0&0&0&I_{r'}\\
0&0&0&0&0&y_2&I_{r''}
\end{pmatrix}.
\]
It then follows from a direct calculation that 
\[
\b{x}\jmath_{n,r}(\b{y})
=
\begin{pmatrix}
I_{r''}\\
y_1&I_{r'}\\
A_1&A_2&I_{n-r}\\
0&0&0&1\\
0&0&0&0&I_{n-r}\\
0&0&0&0&B_2&I_{r'}\\
0&0&0&0&B'_1&y_2&I_{r''}
\end{pmatrix}
\quad\text{with}\quad
B'_1=B_1+y_2B_2.
\] 
We can decompose $\b{x}\jmath_{n,r}(\b{y})$ as
\[
\begin{pmatrix}
I_{r''}\\
y_1&I_{r'}\\
A_1&0&I_{n-r}\\
0&0&0&1\\
0&0&0&0&I_{n-r}\\
0&0&0&0&0&I_{r'}\\
0&0&0&0&B_1&y_2&I_{r''}
\end{pmatrix}
\begin{pmatrix}
I_{r''}\\
0&I_{r'}\\
0&A_2&I_{n-r}\\
0&0&0&1\\
0&0&0&0&I_{n-r}\\
0&0&0&0&B_2&I_{r'}\\
0&0&0&0&0&0&I_{r''}
\end{pmatrix}
=\b{x}''\b{y}'.
\]
Observe that $\b{x}''\in\b{X}_{n,r''}$. On the other hand, we let $\check{N}'_{r'',r'}$ be the subgroup of $\GL_{2n+1}$ 
consisting of the matrices $\b{y}'$ (as $A_2, B_2$ vary). Now the integral \eqref{E:id for n>r 5} can be written as
\begin{align}\label{E:id for n>r 6}
\int_{V_{M_{P'}}N_{Q_{2r}}\check{N}_{r'',r'}\backslash\GL_{2r}}
\int_{\check{N}'_{r'',r'}}
\int_{\b{X}_{n,r''}}
W_v\left(\b{x}''\b{y}'\jmath_{n,r}(h)\right)
\phi_s\left(w^{-1}h, I_{2r''}, I_{M_P}\right)d\b{x}''d\b{y}'dh.
\end{align}

\subsubsection{}
The last step is to factor the $dh$ integration in \eqref{E:id for n>r 6} through $\imath_{r,r''}(\GL_{2r''})$:
\begin{align}\label{E:id for n>r 7}
\begin{split}
&\int_{V'_{M_{P'}}N'_{Q_{2r}}\check{N}_{r'',r'}\jmath_{r,r''}(\GL_{2r''})\backslash\GL_{2r}}
\int_{\check{N}'_{r'',r'}}\\
&\quad\quad\quad\quad
\int_{V_{\GL_{2r''}}\backslash\GL_{2r''}}
\int_{\b{X}_{n,r''}}
W_v\left(\b{x}''\b{y}'\jmath_{n,r}\left(\imath_{r,r''}(h')h\right)\right)
\phi_s\left(w^{-1}\imath_{r,r''}(h')h, I_{2r''}, I_{M_P}\right)d\b{x}''dh'd\b{y}'dh.
\end{split}
\end{align}
Here $V'_{M_{P'}}$ is the subgroup of $V_{M_{P'}}$ consisting of the matrices of the form
\[
\begin{pmatrix}
I_{r''}\\
&z\\
&&z'\\
&&&I_{r''}
\end{pmatrix}
\]
for $z,z'\in V_{\GL_{2r'}}$, $N'_{Q_{2r}}$ is the subgroup of $N_{Q_{2r}}$ consisting of the matrices of the form
\[
\begin{pmatrix}
I_{r''}&0&*&0\\
&I_{r'}&*&*\\
&&I_{r'}&0\\
&&&I_{r''}
\end{pmatrix}
\]
and $\imath_{r,r''}:\GL_{2r''}\hookto\GL_{2r}$ is the embedding given by 
\[
h'=\pMX{a}{b}{c}{d}
\mapsto
\begin{pmatrix}
a&&b\\
&I_{2r'}&\\
c&&d
\end{pmatrix}
\quad\text{where}\quad
a,b,c,d\in\M_{r'',r''}.
\] 
Since 
\[
w^{-1}\imath_{r,r''}(h)w
=
\begin{pmatrix}
I_{r'}\\
&h'\\
&&I_{r'}
\end{pmatrix}
\]
we see that 
\[
\phi_s\left(w^{-1}\imath_{r,r''}(h')h, I_{2r''}, I_{M_P}\right)
=
\phi_s\left(w^{-1}h, h', I_{M_P}\right).
\]
Moreover, since $\jmath_{n,r}(\imath_{r,r''}(h'))=\jmath_{n,r''}(h')$ and $\b{y}'\jmath_{n,r''}(h')=\jmath_{n,r''}(h')\b{y}'$, we see
that the integral \eqref{E:id for n>r 7} becomes
\begin{align}\label{E:id for n>r 8}
\begin{split}
&\int_{V'_{M_{P'}}N'_{Q_{2r}}\check{N}_{r'',r'}\imath_{r,r''}(\GL_{2r''})\backslash\GL_{2r}}
\int_{\check{N}'_{r'',r'}}\\
&\quad\quad\quad\quad
\int_{V_{\GL_{2r''}}\backslash\GL_{2r''}}
\int_{\b{X}_{n,r''}}
W_v\left(\b{x}''\jmath_{n,r''}(h')\b{y}'\jmath_{n,r}\left(h\right)\right)
\phi_s\left(w^{-1}h, h', I_{M_P}\right)d\b{x}''dh'd\b{y}'dh.
\end{split}
\end{align}
Since the inner integral
\[
\int_{V_{\GL_{2r''}}\backslash\GL_{2r''}}
\int_{\b{X}_{n,r''}}
W_v\left(\b{x}''\jmath_{n,r''}(h')\b{y}'\jmath_{n,r}\left(h\right)\right)
\phi_s\left(w^{-1}h, h', I_{M_P}\right)d\b{x}''dh'
\]
in \eqref{E:id for n>r 8} represents the Rankin-Selberg integral attached to $\pi$ and $\tau''$ for every fixed 
$h\in\GL_{2r}$ and $\b{y}'\in\check{N}'_{r'',r'}$, we conclude that \eqref{E:main id for main'''} holds when $r\le n$ by the 
same argument used in the proof of the case $n<r''$. This conclude the proof of \eqref{E:main id for main'''}, and hence 
the proof of \thmref{T:main'''}.\qed

\section{Archimedean and Minimal cases}\label{S:mini}
In this section, we prove the following proposition.

\begin{prop}\label{P:main'}
\thmref{T:main'} hold when
\begin{itemize}
\item
$n=0$ and $r=1$;
\item
$F$ is archimedean, $E=F\,\x\, F$ and $n\ge r-1$.
\end{itemize}
\end{prop}

In particular, by combining this with the results obtained in the previous sections, we complete the verification of 
Assumption \ref{SS:assumption}. To prove \propref{P:main'} when $F$ is archimedean, $E=F\,\x\, F$ and $n\ge r-1$, we 
may assume $\tau$ is the generic quotient of an induced representation of $\GL_r(F)$ that is induced from the Borel 
subgroup of $\GL_r(F)$. In view of \thmref{T:main'''} (which holds when $F$ is archimedean as we remarked in the previous 
section), it suffices to prove the case when $r=1$. Consequently, the proof of \propref{P:main'} are separated into the 
following three cases: (i) $n=0$, $r=1$ and $E$ is a field; (ii) $n=0$, $r=1$ and $E=F\,\x\, F$; (iii) $n\ge r=1$ and 
$E=F\,\x\,F$.

\subsection{Preliminaries}
The local field $F$ in this section can be archimedean or non-archimedean. Recall that we have fixed an element 
$\delta\in E^\x$ satisfying $\theta(\delta)=-\delta$ when $E$ is a field, and have put $\Delta=\delta^2\in F^\x$. 
When $F=\mathbb{R}$, we simply take $\delta=i=\sqrt{-1}$. Recall also that $\psi_2$ is the character of $F$ defined by 
$\psi_2(x)=\psi(2x)$.

\subsubsection{An isomorphism}\label{SSS:gp isom}
Suppose that $E$ is a field. We have an isomorphism 
\[
{\rm SL}_2(F)\overset{\sim}{\longto}{\rm SU}_2(F);
\quad h_1\mapsto\pMX{\delta}{}{}{1}h_1\pMX{\delta}{}{}{1}^{-1}.
\]
It then follows from the Hilbert's Satz 90 that every $h\in\U_2(F)$ can be written as
\begin{equation}\label{E:h decomp}
h
=
\pMX{a}{}{}{\theta(a)^{-1}}\pMX{\delta}{}{}{1}h_1\pMX{\delta}{}{}{1}^{-1}
\end{equation}
for some $a\in E^\x$ and $h_1\in{\rm SL}_2(F)$. Notice that $\det(h)=a\theta(a)^{-1}$.

\subsubsection{Induced representations of $\U_2$}
Let $\tau$ be an irreducible representation of $E^\x$. Therefore, $\tau=\chi$ is a character of $E^\x$ if $E$ is a field; while
$\tau$ is of the form $\chi_1\bt\chi_2$ for some characters $\chi_1, \chi_2$ of $F^\x$. In any case, we denote by $\chi_0$
the restriction of $\tau$ to $F^\x$; we thus have $\chi_0=\chi_1\chi_2$ when $E=F\,\x\,F$. Since $Q_2=B_{\U_2}$, the 
induced representation $\rho_{\tau,s}$ defined in \S\ref{SSS:rho_tau} is noting but
\[
I_{B_{\U_2}}^{\U_2}(\chi|\cdot|_E^{s-\frac{1}{2}})
\quad\text{or}\quad
I_{B_{\GL_2}}^{\GL_2}\left(\chi_1|\cdot|_F^{s-\frac{1}{2}}\bt\chi^{-1}_2|\cdot|_F^{\frac{1}{2}-s}\right)
\]
according to $E$ is a field or not. Notice that $\tau^*(a)=\chi(\theta(a))^{-1}$ if $E$ is a field, and 
$\tau^*=\chi_2^{-1}\bt\chi_1^{-1}$ if $E=F\,\x\,F$. Then the intertwining map 
\[
A(w_{1,1}, \tau,s):\cV_{B_{\U_2}}^{\U_2}(\tau_s)\longto\cV_{B_{\U_2}}^{\U_2}(\tau^*_{1-s})
\]
is then given by 
\[
A(w_{1,1},\tau,s)\xi_s(h)
=
|\Delta|_F^{\frac{r}{2}}\int_{V_{\U_2}(F)}\xi_s(w^{-1}_{1,1}uh)du;
\]
whereas the normalized one $A_{\psi,\delta}(w_{1,1},\tau,s)$ is defined to satisfy the following identity
\[
\int_{V_{\U_2}(F)}\xi_s(w_{1,1}uh)\psi'^{-1}(u_{1,2})du
=
\int_{V_{\U_2}(F)}A_{\psi,\delta}(w_{r,r},\tau,s)\xi_s(w_{1,1}uh)\psi'^{-1}(u_{1,2})du.
\]
Here $\psi'$ is the character of $E$ given by $\psi'(x)=\psi_E(\delta x)$ when $E$ is a field, and $\psi'=\psi_2$ when 
$E=F\,\x\, F$. Furthermore, since $V_{\U_2}(F)\cong F$, the Haar measure $du$ is the one that is self-dual with respect to 
$\psi_2$
\subsubsection{Godement sections}\label{SSS:Godement}
Given a Bruhat-Schwartz function $\varphi\in\cS(F^2)$, a complex number $s$ and $h\in\U_2(F)$, we define the 
$Godement$ $section$ $f_s^\phi(h;\tau)$ attached to $\phi$ and $\tau$ by 
\begin{equation}\label{E:GJ section field}
f_s^\phi(h;\tau)
=
\chi(a)|a|_F^s\int_{F^\x}\varphi((0,t)h_1)\chi_0(t)|t|_F^{2s}d^\x t
\end{equation}
if $E$ is a field and $h$ is written as \eqref{E:h decomp}, where $\chi_0$ is the restriction of $\chi$ to $F^\x$.
On the other hand, if $E=F\,\x\, F$, then we define
\begin{equation}\label{E:GJ section split}
f^{\phi}_s(h;\tau)
=
\chi_1({\rm det}(h))|{\rm det}(h)|_F^s\int_{F^\x}\varphi((0,t)h)\chi_1\chi_2(t)|t|_F^{2s}d^\x t.
\end{equation}
Since the integrals \eqref{E:GJ section field} and \eqref{E:GJ section split} are essentially the Tate integrals 
(cf. \cite{Tatethesis}), they converge absolutely for $\Re(s)\gg 0$ and admit the meromorphic continuation to the whole 
complex plane. Moreover, it's not hard to check that 
\[
f^\varphi_s(-;\tau)\in\cV_{B_{\U_2}}^{\U_2}(\tau_s)
\]
whenever the integral defining $f_s^\varphi(-;\tau)$ is defined at $s$ (cf. \cite{JLbook}, \cite{Baruch1997}), and 
$f^\varphi_s(-;\tau)$ is well-defined when $E$ is a field, i.e. independent of the decomposition \eqref{E:h decomp} 
(cf. \cite[Lemma 2.5]{Baruch1997}).\\

To describe the action of the intertwining map $A(w_{1,1},\tau,s)$ on $f_s^{\varphi}(-;\tau)$, we define the (symplectic) 
Fourier transform $\h{\varphi}$ of $\varphi$ by 
\[
\h{\varphi}(x,y)=\int_F\int_F \varphi(z,w)\psi_2(zy-wx)dzdw
\]
where $dz$, $dw$ are the Haar measures on $F$ that are self-dual with respect to $\psi_2$. We have the following lemma.
\begin{lm}\label{L:int on sec}
We have $A_{\psi,\delta}(w_{1,1},\tau,s)f_s^{\varphi}(-;\tau)=\chi_0(-1)f_{1-s}^{\h{\varphi}}(-;\tau^*)$.
\end{lm}

\begin{proof}
Assume first that $E=F\,\x\,F$. Then the assertion follows from the results in \cite[Section 4.B]{GelbartJacquet1979} and 
\cite[Proposition 4.5.9]{Bump1998}. More concretely, the result in \cite[Section 4.B]{GelbartJacquet1979} implies
\begin{equation}\label{E:intertwining map on section}
\chi_1(-1)\gamma^{{\rm WD}}(2s-1,\chi_1\chi_2,\psi_2)A(w_{1,1},\tau,s)f^\varphi_s(h;\tau)
=
f^{\h{\varphi}}_{1-s}(h;\tau^*).
\end{equation}
On the other side, by \eqref{E:gamma GL 1}, \eqref{E:gamma GL 2} and \cite[Proposition 4.5.9]{Bump1998}, we find that 
\[
\int_{V_{\U_2}(F)}f_s(w_{1,1}uh)\psi_2^{-1}(u_{1,2})du
=
\chi_2(-1)\gamma^{{\rm WD}}(2s-1,\chi_1\chi_2,\psi_2)
\int_{V_{\U_2}(F)}A(w_{r,r},\tau,s)f_s(w_{1,1}uh)\psi^{-1}_2(u_{1,2})du.
\]
Together, we conclude that $A_{\psi,\delta}(w_{1,1},\tau,s)f_s^{\varphi}(-;\tau)=\chi_1\chi_2(-1)f^{\h{\varphi}}_{1-s}(-;\tau^*)$.\\

Next, assume that $E$ is a field. The observation is that, if we let $\tau_0=1\bt\chi_0$ be an irreducible representation of 
$F^\x\,\x\,F^\x$,  where $1$ stands for the trivial character of $F^\x$, then 
\[
f_s^\varphi(h;\tau)
=
f_s^\varphi(h_1;\tau_0)
\]
provided that $h$ is of the form \eqref{E:h decomp}. In other words, we would like to reduce the computation to the split case.
Since 
\[
\pMX{}{1}{1}{}\pMX{1}{\delta x}{}{1}
=
\pMX{\delta^{-1}}{}{}{-\delta}\pMX{\delta}{}{}{1}\pMX{}{1}{-1}{}\pMX{1}{x}{}{1}\pMX{\delta}{}{}{1}^{-1}
\]
where $x\in F$, we find that 
\begin{align*}
A(w_{1,1},\tau,s)f^\phi_s(I_2;\tau)
&=
|\delta|_E^{\frac{1}{2}}\int_{F}
f^\varphi_s\left(\pMX{}{1}{1}{}\pMX{1}{\delta x}{}{1};\tau\right)dx\\
&=
\chi(\delta^{-1})|\delta^{-1}|_E^{s-\frac{1}{2}}\int_{F}
f^\varphi_s\left(\pMX{}{1}{-1}{}\pMX{1}{x}{}{1};\tau_0\right)dx\\
&=
\chi^{-1}(-\delta)|\delta|_E^{-s+\frac{1}{2}}\gamma^{{\rm WD}}(2s-1,\chi_0,\psi_2)^{-1}f_{1-s}^{\h{\varphi}}(I_2;\tau^*_0)\\
&=
\chi_0(-1)|\delta|_E^{-\frac{1}{2}}\gamma^{{\rm WD}}(2s-1,\chi_0,\psi_2)^{-1}f_{1-s}^{\h{\varphi}}
\left(\pMX{\delta}{}{}{-\delta^{-1}};\tau^*\right)
\end{align*}
by \eqref{E:intertwining map on section}. It follows that 
\begin{align}\label{E:intertwining map on section field}
\begin{split}
A(w_{1,1},\tau,s)f^\phi_s(h;\tau)
&=
\chi_0(-1)|\delta|_E^{-\frac{1}{2}}\gamma^{{\rm WD}}(2s-1,\chi_0,\psi_2)^{-1}f_{1-s}^{\h{\varphi}}
\left(\pMX{\delta}{}{}{-\delta^{-1}}h;\tau^*\right)\\
&=
\chi^{-1}(-\delta)|\delta|_E^{-s+\frac{1}{2}}\gamma^{{\rm WD}}(2s-1,\chi_0,\psi_2)^{-1}f_{1-s}^{\h{\varphi}}(h;\tau^*)
\end{split}
\end{align}
for every $h\in\U_2(F)$.\\

We proceed to compute 
\begin{align*}
\int_{V_{\U_2}(F)}f^\varphi_s(w_{1,1}u;\tau)\psi'^{-1}(u_{1,2})du
&=
\int_{F}
f^\varphi_s\left(\pMX{}{1}{1}{}\pMX{1}{\delta x}{}{1};\tau\right)\psi_2^{-1}(\Delta x)dx\\
&=
\chi(\delta^{-1})|\delta^{-1}|_E^s\int_{F}
f^\varphi_s\left(\pMX{}{1}{-1}{}\pMX{1}{x}{}{1};\tau_0\right)\psi_2^{-1}(\Delta x)dx\\
&=
\chi^{-1}(-\delta)|\delta|_E^{-s}\int_{F}\int_{F^\x}
\varphi\left((0,t)\pMX{}{1}{1}{}\pMX{1}{x}{}{1}\right)\chi_0(t)|t|_F^{2s}\psi_2^{-1}(\Delta x)d^\x t\,dx\\
&=
\chi^{-1}(-\delta)|\delta|_E^{-s}\int_{F}\int_{F^\x}
\varphi\left(t,tx\right)\chi_0(t)|t|_F^{2s}\psi_2^{-1}(\Delta x)d^\x t\,dx\\
&=
\chi^{-1}(-\delta)|\delta|_E^{-s}\int_{F}\int_{F^\x}
\varphi\left(t,x\right)\chi_0(t)|t|_F^{2s-1}\psi_2^{-1}(\Delta t^{-1}x)d^\x t\,dx.
\end{align*}
It's clear that the above integral converges absolutely for every $s$. Moreover, by replacing $\phi$ by $\h{\phi}$, $\chi$ by 
$\chi^*$ and $s$ by $1-s$, we find that 
\begin{align*}
\int_{V_{\U_2}(F)}f^{\h{\phi}}_{1-s}(w_{1,1}u;\tau^*)\psi'^{-1}(u_{1,2})du
=
\chi(\delta)|\delta|_E^{s-1}\int_{F}\int_{F^\x}
\h{\phi}\left(t,x\right)\chi^{-1}_0(t)|t|_F^{1-2s}\psi_2^{-1}(\Delta t^{-1}x)d^\x t\,dx.
\end{align*}
By the Fourier inversion formula, we have
\[
\int_{F}
\h{\varphi}\left(t,x\right)\psi_2^{-1}(\Delta t^{-1}x)dx
=
\int_F
\varphi(\Delta t^{-1},x)\psi^{-1}_2(tx)dx.
\]
This implies
\begin{align*}
\int_{V_{\U_2}(F)}f^{\h{\varphi}}_{1-s}(w_{1,1}u;\tau^*)\psi'^{-1}(u_{1,2})du
&=
\chi(\delta)|\delta|_E^{s-1}\int_{F}\int_{F^\x}
\h{\varphi}\left(t,x\right)\chi^{-1}_0(t)|t|_F^{1-2s}\psi_2^{-1}(\Delta t^{-1}x)d^\x t\,dx\\
&=
\chi(\delta)|\delta|_E^{s-1}\int_{F}\int_{F^\x}
\varphi\left(\Delta t^{-1},x\right)\chi^{-1}_0(t)|t|_F^{1-2s}\psi_2^{-1}(tx)d^\x t\,dx\\
&=
\chi(\delta)|\delta|_E^{s-1}\chi_0(\Delta^{-1})|\Delta|_F^{1-2s}\int_{F}\int_{F^\x}
\varphi\left(t,x\right)\chi^{-1}_0(t)|t|_F^{1-2s}\psi_2^{-1}(\Delta t^{-1}x)d^\x t\,dx\\
&=
\chi_0(-1)\int_{V_{\U_2}(F)}f^\varphi_s(w_{1,1}u;\tau)\psi'^{-1}(u_{1,2})du.
\end{align*}
By \eqref{E:intertwining map on section field}, we therefore get that 
$A_{\psi,\delta}(w_{1,1},\tau,s)f^\varphi_s(-;\tau)=\chi_0(-1)f^{\h{\varphi}}_{1-s}(-;\tau^*)$. This completes the proof.
\end{proof}

\subsection{Proof of Case (i)}
In this case, we have $n=0$, $r=1$ and $E$ is a field. By definition, one has
\[
\U_1(F)=E^1:=\stt{a\in E^\x\mid |a|_E=1}.
\] 
On the other hand, the Hilbert's Satz 90 induces the following exact sequence,
\[
1\longto F^\x\overset{\id}{\longto} E^\x\overset{\iota}{\longto} E^1\longto 1
\] 
where $\id$ stands for the identity map; while $\iota(a)=a\theta(a)^{-1}$ for $a\in E^\x$. Therefore, we can identity 
$\U_1(F)$ with $E^\x/F^\x$ via this exact sequence. The representation $\pi$ of $\U_1(F)$ is now a character $\eta_1$ of 
$E^1$, and by pulling back to $E^\x$, we obtain a character $\eta$ of $E^\x$ that is trivial on $F^\x$. More concretely, 
we have 
\[
\eta(a)=\eta_1(a\theta(a)^{-1})
\quad
\text{for $a\in E^\x$}.
\]
Note that $\eta$ is the standard base change of $\eta_1$ to $\GL_1(E)=E^\x$. The representation $\tau$ is also a character 
of $E^\x$, and we denote it by $\chi$ as before.

\subsubsection{}
Let us now begin to prove the \propref{P:main'} for case (i). We should mention that Kaplan also did similar computations in
\cite[Section 6.1]{Kaplan2015}. The goal is establish the following identity: 
\begin{equation}\label{E:n=0, r=1, E is field}
\Gamma_\delta(s,\eta_1\x\chi,\psi)
=
\eta_1(-1)\chi(\delta)|\delta|_E^{s-\frac{1}{2}}\gamma^{{\rm WD}}(s,\eta\chi,\psi_E).
\end{equation}
The idea is to connect the Rankin-Selberg integrals in this case to the local Tate integrals. The key is to apply the Godement 
sections that we introduced in \S\ref{SSS:Godement}. Let $v\in\cV_{\eta_1}$ and $f_s\in\cV_{B_{{\U}_2}}^{\U_2}(\chi_s)$.
The Rankin-Selberg integral $\Psi_{0,1}(v\ot \xi_s)$ is of the form
\begin{equation}\label{E:int for n=0, r=1, E is field}
\int_{\U_1(F)} \eta_1(\alpha)\xi_s(\jmath^{0,1}(\alpha))d\alpha
=
\int_{E^\x/F^\x}\eta(a)\xi_s\left(\jmath^{0,1}(a\theta(a)^{-1})\right)d^\x a
\end{equation}
where the embedding $\jmath^{0,1}:\U_1(F)\hookto\U_2(F)$ is given by 
\[
\jmath^{0,1}(\alpha)=\frac{1}{2}\pMX{1+\alpha}{1-\alpha}{1-\alpha}{1+\alpha}
\]
for $\alpha\in\U_1(F)$. By \cite[Lemma 3.1]{YCheng}, we may assume that $\xi_s=f_s^\varphi(-;\chi)$ for some 
$\varphi\in\cS(F^2)$. Then by letting $\alpha=a\theta^{-1}$ for some $a\in E^\x$ and the decomposition 
\[
\jmath^{0,1}(\alpha)
=
\jmath^{0,1}(a\theta(a)^{-1})
=
\pMX{a}{}{}{\theta(a)^{-1}}
\pMX{\delta}{}{}{1}
\frac{1}{2}
\pMX{(a^{-1}+\theta(a)^{-1})}{\delta^{-1}(a^{-1}-\theta(a)^{-1})}{\delta(\theta(a)-a)}{(\theta(a)+a)}
\pMX{\delta}{}{}{1}^{-1}
\]
we find that 
\begin{equation}\label{E:xi_s to f_s E is field}
\xi_s(\jmath^{0,1}(\alpha))
=
f_s^{\varphi}(\jmath^{0,1}(a\theta(a)^{-1});\chi)
=
\chi(a)|a|_E^s
\int_{F^\x}
\Phi\left(ta\right)\chi_0(t)|t|_F^{2s}d^\x t
\end{equation}
by \eqref{E:GJ section field}, where we put
\[
\Phi(a)
=
\varphi\left(\frac{\delta(\theta(a)-a)}{2},\frac{\theta(a)+a}{2}\right)
\]
for $a\in E$. This is a Bruhat-Schwartz function on $E$. By \eqref{E:int for n=0, r=1, E is field} and 
\eqref{E:xi_s to f_s E is field}, we see that 
\[
\Psi_{0,1}(v\ot\xi_s)
=
\int_{E^\x}\Phi(a)\eta\chi(a)|a|^s d^\x a
\]
which is exactly the local Tate integral attached to $\Phi$ and $\eta\chi$.

\subsubsection{}
Next, we compute $\Psi_{0,1}(v\ot A_{\psi,\delta}(w_1,1,\chi,s)\xi_s)$ with $\xi_s=f_s^\phi$. By \lmref{L:int on sec}, we have 
\[
 A_{\psi,\delta}(w_1,1,\chi,s)\xi_s
 =
  A_{\psi,\delta}(w_1,1,\chi,s)f^{\varphi}_s(-;\chi)
  =
  \chi_0(-1)f^{\h{\varphi}}_{1-s}(-;\chi^*)
\]
and hence the above calculations (with $\chi$, $s$ and $\varphi$ replaced by $\chi^*$, $1-s$ and $\h{\varphi}$, respectively) 
imply
\[
\Psi_{0,1}(v\ot A_{\psi,\delta}(w_1,1,\chi,s)\xi_s)
=
\chi_0(-1)\int_{E^\x}\Phi'(a)\eta\chi^*(a)|a|^{1-s}_Ed^\x a
\]
where we define
\[
\Phi'(a)
=
\h{\varphi}\left(\frac{\delta(\theta(a)-a)}{2},\frac{\theta(a)+a}{2}\right)
\]
for $a\in E^\x$. To complete the proof, we claim that 
\begin{equation}\label{E:FT field}
\Phi'(a)
=
|\delta|_E^{\frac{1}{2}}\cdot\h{\Phi}(-\delta\theta(a))
\end{equation}
for $a\in E^\x$. Here $\h{\Phi}$ stands for the Fourier transform of $\Phi$ with respect to $\psi_E$, which is given by
\[
\h{\Phi}(a)
=
\int_E\Phi(b)\psi_E(ab)db
\]
with $db$ the Haar measure on $E$ that is self-dual with respect to $\psi_E$.

\subsubsection{}
To prove the claim, we note the formula
\[
\int_E f(b)db
=
|\delta|_E^{\frac{1}{2}}\int_F\int_F f(z+\delta w)dzdw
\]
for every $f\in L^1(E)$, where $db$ is the Haar measure on $E$ that is self-dual with to $\psi_E$, and $du, dv$ are the Haar 
measures on $F$ that are self-dual with to $\psi_2$. Let $a=x+\delta y$ for some $x,y\in F$. We have 
\begin{align*}
\h{\Phi}(-\delta\theta(a))
=
\int_E\Phi(b)\psi_E(-\delta\theta(a)b)db
&=
|\delta|_E^{\frac{1}{2}}
\int_F\int_F
\Phi(z+\delta w)\psi_2(-\Delta xw+\Delta yz)dzdw\\
&=
|\delta|_E^{\frac{1}{2}}
\int_F\int_F
\varphi(-\Delta w,z)\psi_2(-\Delta xw+\Delta yz)dzdw\\
&=
|\delta|_E^{-\frac{1}{2}}
\int_F\int_F
\varphi(v,u)\psi_2(xw+\Delta yz)dzdw\\
&=
|\delta|_E^{-\frac{1}{2}}\h{\varphi}(-\Delta y,x)\\
&=
|\delta|_E^{-\frac{1}{2}}\Phi'(a)
\end{align*}
which verifies the claim. It follows that 
\begin{align*}
\Psi_{0,1}(v\ot A_{\psi}(w_1,1,\chi,s)\xi_s)
&=
\chi_0(-1)\int_{E^\x}\Phi'(a)\eta\chi^*(a)|a|^{1-s}_Ed^\x a\\
&=
\chi_0(-1)|\delta|_E^{\frac{1}{2}}
\int_{E^\x}\h{\Phi}(-\delta\theta(a))\eta\chi^*(a)|a|^{1-s}_Ed^\x a\\
&=
\chi_0(-1)|\delta|_E^{\frac{1}{2}}
\int_{E^\x}\h{\Phi}(-\delta a)(\eta\chi)^{-1}(a)|a|^{1-s}_Ed^\x a\\
&=
\eta_1(-1)\chi(\delta)|\delta|_E^{s-\frac{1}{2}}
\int_{E^\x}\h{\Phi}(a)(\eta\chi)^{-1}(a)|a|^{1-s}_Ed^\x a\\
&=
\eta_1(-1)\chi(\delta)|\delta|_E^{s-\frac{1}{2}}\gamma^{{\rm WD}}(s,\eta\chi,\psi_E)
\int_{E^\x}\Phi(a)\eta\chi(a)|a|^s d^\x a\\
&=
\eta_1(-1)\chi(\delta)|\delta|_E^{s-\frac{1}{2}}\gamma^{{\rm WD}}(s,\eta\chi,\psi_E)
\Psi_{0,1}(v\ot\xi_s)
\end{align*}
where in the fifth equality, we have used the result of Tate (cf. \cite{Tatethesis}, \cite[Section 3.1]{Bump1998}). 
From this, the identity \eqref{E:n=0, r=1, E is field} follows.\qed

\subsection{Proof of Case (ii)}
In this case, we have $n=0$, $r=1$ and $E=F\,\x\, F$. The representation $\pi$ of $\GL_1(F)=F^\x$ is now a character $\eta$
of $F^\x$. On the other hand, the representation $\tau$ of $\G_1(F)=F^\x\,\x\, F^\x$ becomes $\chi_1\bt\chi_2$ for some 
characters $\chi_1,\chi_2$ of $F^\x$. The aim is to prove the following identity:
\begin{equation}\label{E:n=0,r=1,E=FxF}
\Gamma_\delta(s,\eta\x\tau,\psi)
=
\eta\chi_2(-1)\gamma^{{\rm WD}}(s,\eta\chi_1,\psi)\gamma^{{\rm WD}}(s,\eta^{-1}\chi_2,\psi).
\end{equation}
Notice that $\omega_\tau(\delta)|\delta|^{s-1/2}_E=\chi_2(-1)$ in this case.

\subsubsection{}
To establish \eqref{E:n=0,r=1,E=FxF}, the idea is again to connect the Rankin-Selberg integrals in this case to the local 
Tate integrals; also, the key is to use Godement sections. To be more precise, let $v\in\cV_\eta$ and 
$\xi_s\in\cV_{V_{\GL_2}}^{\GL_2}(\tau_s)$. Then the Rankin-Selberg integral $\Psi_{0,1}(v\ot\xi_s)$ is of the form
\[
\Psi_{0,1}(v\ot\xi_s)
=
\int_{F^\x}\eta(a)\xi_s(\jmath^{0,1}(a))d^\x a
\]
where the embedding $\jmath^{0,1}:\GL_1(F)\hookto\GL_2(F)$ is given by 
\[
\jmath^{0,1}(a)
=
\frac{1}{2}\pMX{1}{1}{1}{-1}\pMX{1}{}{}{a}\pMX{1}{1}{1}{-1}
\]
for $a\in F^\x$. As before, we may assume that $\xi_s=f^\varphi_s(-;\tau)$ for some $\varphi\in\cS(F^2)$. 
For the ease of notation, we put
\[
h_0=\frac{1}{2}\pMX{1}{1}{1}{-1}
\quad\text{and}\quad
\varphi'=\rho(h_0)\varphi
\]
where $\rho$ stands for the right translation action of $\GL_2(F)$ on the $\cS(F^2)$, that is, 
$\rho(h)\phi(x,y)=\phi((x,y)h)$, where $h\in\GL_2(F)$ and $\phi\in\cS(F^2)$. At this point, we further assume that 
\[
\varphi'(x,y)=\varphi_1(x)\varphi_2(y)
\]
for some Bruhat-Schwartz functions $\varphi_1,\varphi_2\in\cS(F)$.

\subsubsection{}
By \eqref{E:GJ section split}, we find that 
\begin{align*}
\Psi_{0,1}(v\ot\xi_s)
&=
\int_{F^\x}\eta(a)\xi_s(\jmath^{0,1}(a))d^\x a\\
&=
\int_{F^\x}\eta(a)f^\varphi_s\left(h^{-1}\pMX{1}{}{}{a}h; \tau\right)d^\x a\\
&=
\int_{F^\x}\int_{F^\x}
\eta\chi_1(a)|a|_F^s\,\varphi\left((0,t)h^{-1}\pMX{1}{}{}{a}h\right)\chi_1\chi_2(t)|t|_F^{2s}d^\x td^\x a\\
&=
\int_{F^\x}\int_{F^\x}
\eta\chi_1(a)|a|_F^s\,\varphi'(t,-ta)\chi_1\chi_2(t)|t|_F^{2s}d^\x td^\x a\\
&=
\eta\chi_1(-1)\int_{F^\x}\int_{F^\x}
\varphi'(t,a)\,\eta\chi_1(a)|a|^s_F\,\eta^{-1}\chi_2(t)|t|^s_Fd^\x td^\x a\\
&=
\eta\chi_1(-1)
\left(\int_{F^\x}\varphi_2(a)\eta\chi_1(a)|a|_F^sd^\x a\right)
\left(\int_{F^\x}\varphi_1(t)\eta^{-1}\chi_2(t)|t|_F^sd^\x t\right)
\end{align*}
which is essentially a product of two local Tate integrals.

\subsubsection{}
We proceed to compute $\Psi_{0,1}(v\ot A_{\psi,\delta}(w_{1,1},\tau,s)\xi_s)$. By \lmref{L:int on sec}, we have 
\[
 A_{\psi,\delta}(w_1,1,\chi,s)\xi_s
 =
  A_{\psi,\delta}(w_1,1,\chi,s)f^{\varphi}_s(-;\tau)
  =
  \chi_1\chi_2(-1)f^{\h{\varphi}}_{1-s}(-;\tau^*)
\]
and hence the above calculations (with $\tau$, $s$ and $\varphi$ being replaced by $\tau^*$, $1-s$ and $\h{\varphi}$, 
respectively) imply
\begin{equation}\label{E:dual RS int split}
\Psi_{0,1}(v\ot A_{\psi,\delta}(w_1,1,\chi,s)\xi_s)
=
\eta\chi_1(-1)\int_{F^\x}\int_{F^\x}
\h{\varphi}'(t,a)\,\eta\chi^{-1}_2(a)|a|^{1-s}_F\,\eta^{-1}\chi^{-1}_1(t)|t|^{1-s}_Fd^\x td^\x a
\end{equation}
where we put $\h{\varphi}'=\rho(h_0)\h{\varphi}$. To finishes the proof, we claim that 
\begin{equation}\label{E:FT split}
\h{\varphi}'(x,y)
=
|2|_F^{-1}\,\widehat{\varphi'}\left(-\frac{x}{2},-\frac{y}{2}\right)
=
\h{\varphi}_2\left(x\right)\h{\varphi}_1\left(-y\right).
\end{equation}
Here for $f\in\cS(F)$, we define the Fourier transform $\h{f}$ of $f$ to be 
\[
\h{f}(x)
=
\int_Ff(y)\psi(xy)dy
\]
where $dy$ is the Haar measure on $F$ that is self-dual with respect to $\psi$. We begin with verifying the first equality
in \eqref{E:FT split}, which is in fact a special case of the following identity:
\[
\widehat{\rho(h)\phi}=|\det(h)|^{-1}_F\rho(h')\h{\phi}
\]
for $h\in\GL_2(F)$ and $\phi\in\cS(F^2)$, where 
\[
h':=\pMX{}{-1}{1}{} {}^th^{-1}\pMX{}{1}{-1}{}.
\]
By definition, we have
\begin{align*}
\widehat{\rho(h)\phi}(x,y)
&=
\int_F\int_F \rho(h)\phi(z,w)\psi_2(zy-wx)dzdw\\
&=
\int_F\int_F \phi((z,w)h)\psi_2\left((z,w)\pMX{}{1}{-1}{}\begin{pmatrix}x\\y\end{pmatrix}\right)dzdw\\
&=
|\det(h)|_F^{-1}\int_F\int_F \phi((z,w))\psi_2\left((z,w)h^{-1}\pMX{}{1}{-1}{}\begin{pmatrix}x\\y\end{pmatrix}\right)dzdw\\
&=
|\det(h)|_F^{-1}\int_F\int_F \phi((z,w))\psi_2\left((z,w)\pMX{}{1}{-1}{} {}^th'\begin{pmatrix}x\\y\end{pmatrix}\right)dzdw\\
&=
|\det(h)|^{-1}_F\rho(h')\h{\phi}(x,y)
\end{align*}
as desired. For the second equality in \eqref{E:FT split}, we note that if $d_2z$ (resp. $dz$) is the Haar measure on $F$
that is self-dual with respect to $\psi_2$ (resp. $\psi$), then $d_2z=|2|^{1/2}_Fdz$. Now we compute
\begin{align*}
|2|_F^{-1}\,\widehat{\varphi'}\left(-\frac{x}{2},-\frac{y}{2}\right)
&=
|2|_F^{-1}\int_F\int_F
\varphi'(z,w)\psi_2\left(-\frac{zy}{2}+\frac{wx}{2}\right)d_2zd_2w\\
&=
\int_F\int_F
\varphi_1(z)\varphi_2(w)\psi\left(-zy+wx\right)dzdw\\
&=
\h{\varphi}_2(x)\h{\varphi}_1(-y)
\end{align*}
which verifies the claim.

\subsubsection{}
By \eqref{E:dual RS int split} and \eqref{E:FT split}, we get that 
\begin{align*}
\Psi_{0,1}(v\ot A_{\psi,\delta}(w_1,1,\chi,s)\xi_s)
&=
\eta\chi_1(-1)\int_{F^\x}\int_{F^\x}
\h{\varphi}'(t,a)\,\eta\chi^{-1}_2(a)|a|^{1-s}_F\,\eta^{-1}\chi^{-1}_1(t)|t|^{1-s}_Fd^\x td^\x a\\
&=
\chi_1\chi_2(-1)
\left(\int_{F^\x}\h{\varphi}_2(t)\eta^{-1}\chi^{-1}_1(t)|t|_F^{1-s}d^\x t\right)
\left(\int_{F^\x}\h{\varphi}_1(a)\eta\chi^{-1}_2(a)|a|_F^{1-s}d^\x a\right)\\
&=
\chi_1\chi_2(-1)\gamma^{{\rm WD}}(s,\eta\chi_1,\psi)\gamma^{{\rm WD}}(s,\eta^{-1}\chi_2,\psi)\\
&\quad\quad\quad\quad\quad\quad\quad
\x\left(\int_{F^\x}\varphi_2(a)\eta\chi_1(a)|a|_F^sd^\x a\right)
\left(\int_{F^\x}\varphi_1(t)\eta^{-1}\chi_2(t)|t|_F^sd^\x t\right)\\
&=
\eta\chi_2(-1)\gamma^{{\rm WD}}(s,\eta\chi_1,\psi)\gamma^{{\rm WD}}(s,\eta^{-1}\chi_2,\psi)
\Psi_{0,1}(v\ot\xi_s)
\end{align*}
after applying the result of Tate (cf. \cite{Tatethesis}, \cite[Section 3.1]{Bump1998}) again. 
This proves the identity \eqref{E:n=0,r=1,E=FxF}.\qed

\subsection{Case (iii)}
In this case, we have $n\ge r=1$ and $E=F\,\x\,F$, so that $\tau=\chi_1\bt\chi_2$ for some characters $\chi_1, \chi_2$ of 
$F^\x$ as in the case (ii). The embedding $\jmath_{n,1}:\GL_2(F)\hookto\GL_{2n+1}(F)$ in this case is given by 
\[
\jmath_{n,1}\left(\pMX{a}{b}{c}{d}\right)
=
\begin{pmatrix}
a&&b\\
&I_{2n-1}&\\
c&&d
\end{pmatrix}
\]
where $a,b,c,d$ are scalars. We are going to show the following identity:
\begin{equation}\label{E:n>r=1,E=FxF}
\Gamma_\delta(s,\pi\x\tau,\psi)
=
\omega_\pi(-1)\chi_1(-1)^n\chi_2(-1)^{n+1}
\gamma^{{\rm WD}}(s,\pi\x\chi_1,\psi)\gamma^{{\rm WD}}(s,\t{\pi}\x\chi_2,\psi).
\end{equation}
Again, we indicate that $\omega_{\tau}(\delta)|\delta|_E^{s-1/2}=\chi_2(-1)$. As expected, we will apply the functional 
equations \eqref{E:FE GL} for $\pi\x\chi_1$ and $\pi\x\chi_2$. However, since the non-degenerated character 
$\psi_{\U_{2n+1}}$ is not the one used in \cite{JPSS1983}, the functional equations require some modifications. 
For this, recall that $I'_N\in\GL_N(F)$ is the matrix defined inductively by 
\[
I'_1=(1)
\quad\text{and}\quad
I'_n
=
\pMX{I'_{n-1}}{}{}{(-1)^{n-1}}.
\]
We put 
\[
d
=
\begin{pmatrix}
2I_n&&\\
&1&\\
&&I'_n
\end{pmatrix}.
\]
Now if $\chi$ is a character of $F^\x$, and if we use the Whittaker model $\cW(\pi,\psi_{\U_{2n+1}})$ for $\pi$,
then the functional equation for $\pi\x\chi$ becomes
\begin{align}\label{E:FE GL'}
\begin{split}
\int_{F^\x}\int_{\M_{1,n}(F)}
&W\left(d
\begin{pmatrix}
&I_n&\\
&&I_n\\
a&x
\end{pmatrix}
d^{-1}\right)
\chi(a)|a|_F^{s-1}
dx d^\x a\\
&=
\gamma^{{\rm WD}}(s,\pi\x\chi,\psi)
\int_{F^\x}\int_{\M_{n-1,1}(F)}
W\left(
\begin{pmatrix}
a&&\\
x&I_{n-1}&\\
&&I_{n+1}
\end{pmatrix}
\right)
\chi(a)|a|^{s-n}
dxd^\x a
\end{split}
\end{align}
where $W\in\cW(\pi,\psi_{\U_{2n+1}})$.

\subsubsection{}
Let $v\in\cV_\pi$ and $\xi_s\in\cV_{V_{\GL_2}}^{\GL_2}(\tau_s)$. We have 
\begin{align}\label{E:id for n>r=1 1}
\begin{split}
\Psi_{n,1}(v\ot\xi_s)
&=
\int_{V_{\GL_2}(F)\backslash\GL_2(F)}\int_{\b{X}_{n,1}(F)}
W_v(\b{u}\jmath_{n,1}(h))\xi_s(h)d\b{u}dh\\
&\quad=
\int_{B_{\GL_2}(F)\backslash\GL_2(F)}\int_{\b{X}_{n,1}(F)}\int_{F^\x}\int_{F^\x}d^\x a d^\x bd\b{u}dh\\
&\quad\quad\quad\quad\quad\quad\quad\quad\quad\quad\quad
W_v\left(\b{u}\jmath_{n,1}\left(\pMX{a}{}{}{b}\right)\jmath_{n,1}(h)\right)\chi_1(a)\chi_2^{-1}(b)|ab^{-1}|_F^{s-1}\xi_s(h)
\end{split}
\end{align}
where the $dh$ integration should be understood in the sense of Iwasawa decomposition. To proceed, let 
\[
\b{u}
=
\begin{pmatrix}
1\\
x&I_{n-1}\\
&&1\\
&&&I_{n-1}\\
&&&y&1
\end{pmatrix}
\] 
and we compute 
\[
\b{u}\jmath_{n,1}\left(\pMX{a}{}{}{b}\right)
=
\begin{pmatrix}
a\\
ax&I_{n-1}\\
0&0&1\\
0&0&0&I_{n-1}\\
0&0&0&y&b
\end{pmatrix}
=
\begin{pmatrix}
a\\
ax&I_{n-1}\\
0&0&1\\
0&0&0&I_{n-1}\\
0&0&0&0&1
\end{pmatrix}
\begin{pmatrix}
1\\
0&I_{n-1}\\
0&0&1\\
0&0&0&I_{n-1}\\
0&0&0&y&b
\end{pmatrix}.
\]
Then the last integral in \eqref{E:id for n>r=1 1} becomes
\begin{align*}
&\int_{B_{\GL_2}(F)\backslash\GL_2(F)}\int_{\M_{1,n-1}(F)}\int_{\M_{n-1,1}(F)}\int_{F^\x}\int_{F^\x}
d^\x a d^\x bdxdydh\\
&\quad\quad\quad\quad\quad\quad\quad\quad\quad
W_v\left(
\begin{pmatrix}
a&\\
x&I_{n-1}\\
0&0&I_{n+1}
\end{pmatrix}
\begin{pmatrix}
I_{n+1}\\
0&I_{n-1}\\
0&y&b
\end{pmatrix}
\jmath_{n,1}(h)
\right)
\chi_1(a)|a|_F^{s-n}\chi_2^{-1}(b)|b|_F^{1-s}\xi_s(h)
\end{align*}
after changing the variable $ax\mapsto x$.

\subsubsection{}
At this moment, we can apply the functional equation \eqref{E:FE GL'} to obtain that 
\[
\gamma^{{\rm WD}}(s,\pi\x\chi_1,\psi)\Psi_{n,1}(v\ot\xi_s)
\]
is equal to
\begin{align*}
&\int_{B_{\GL_2}(F)\backslash\GL_2(F)}\int_{\M_{1,n-1}(F)}\int_{\M_{1,n}(F)}\int_{F^\x}\int_{F^\x}
d^\x a d^\x bdxdydh\\
&\quad\quad\quad\quad\quad\quad\quad\quad\quad
W_v\left(
d
\begin{pmatrix}
&I_n&\\
&&I_n\\
a&x
\end{pmatrix}
d^{-1}
\begin{pmatrix}
I_{n+1}\\
0&I_{n-1}\\
0&y&b
\end{pmatrix}
\jmath_{n,1}(h)
\right)
\chi_1(a)|a|_F^{s-1}\chi_2^{-1}(b)|b|_F^{1-s}\xi_s(h).
\end{align*}
Since 
\[
d
\begin{pmatrix}
&I_n&\\
&&I_n\\
a&x
\end{pmatrix}
d^{-1}
\begin{pmatrix}
I_{n+1}\\
0&I_{n-1}\\
0&y&b
\end{pmatrix}
=
d
\begin{pmatrix}
&I_n&\\
&&I_n\\
1&x
\end{pmatrix}
d^{-1}
\begin{pmatrix}
I_{n+1}\\
0&I_{n-1}\\
0&y&1
\end{pmatrix}\jmath_{n,1}\left(\pMX{a}{}{}{b}\right)
\]
and 
\[
d^{-1}
\begin{pmatrix}
I_{n+1}\\
0&I_{n-1}\\
0&y&1
\end{pmatrix}
d
=
\begin{pmatrix}
I_{n+1}\\
0&I_{n-1}\\
0&y'&1
\end{pmatrix}
\quad\text{where}\quad
y'=(-1)^{n-1}yI'_{n-1}
\]
the above integral becomes
\begin{align}\label{E:id for n>r=1 2}
\begin{split}
\int_{V_{\GL_2}(F)\backslash\GL_2(F)}\int_{\M_{1,n-1}(F)}&\int_{\M_{1,n}(F)}dxdydh\\
&W_v\left(
d
\begin{pmatrix}
&I_n&\\
&&I_n\\
1&x
\end{pmatrix}
\begin{pmatrix}
I_{n+1}\\
0&I_{n-1}\\
0&y&1
\end{pmatrix}
d^{-1}
\jmath_{n,1}(h)
\right)
\xi_s(h)
\end{split}
\end{align}
after changing the variable $y'\mapsto y$.

\subsubsection{}
Now, let us write 
\[
\begin{pmatrix}
&I_n&\\
&&I_n\\
1&x
\end{pmatrix}
=
\begin{pmatrix}
&I_n&\\
&&I_n\\
1&
\end{pmatrix}
\begin{pmatrix}
1&0&x\\
&I_n&0\\
&&I_n
\end{pmatrix}
=
\begin{pmatrix}
1&x'&0\\
&I_{2n-1}&0\\
&&1
\end{pmatrix}
\begin{pmatrix}
1&0&x_{2}\\
&I_{2n-1}&0\\
&&1
\end{pmatrix}
\]
and 
\[
\begin{pmatrix}
I_{n+1}\\
0&I_{n-1}\\
0&y&1
\end{pmatrix}
=
\begin{pmatrix}
1\\
0&I_{2n-1}\\
0&y'&1
\end{pmatrix}
\]
with $x'=(0\,\,x_1)\in\M_{1,2n-1}(F)$ and $y'=(0\,\,y)\in\M_{1,2n-1}(F)$, where we denote $x=(x_1\,\,x_2)$ with 
$x_2\in F$ stands for the last column of $x$. We have 
\[
\begin{pmatrix}
1&0&x_{2}\\
&I_{2n-1}&0\\
&&1
\end{pmatrix}
\begin{pmatrix}
1\\
0&I_{2n-1}\\
0&y'&1
\end{pmatrix}
\begin{pmatrix}
1&0&x_{2}\\
&I_{2n-1}&0\\
&&1
\end{pmatrix}^{-1}
=
\begin{pmatrix}
1&x_2y'&0\\
&I_{2n-1}&0\\
&&1
\end{pmatrix}
\begin{pmatrix}
1\\
0&I_{2n-1}\\
0&y'&1
\end{pmatrix}
\]
and 
\[
d
\begin{pmatrix}
1&0&x_{2}\\
&I_{2n-1}&0\\
&&1
\end{pmatrix}
d^{-1}
=
\begin{pmatrix}
1&0&x'_{2}\\
&I_{2n-1}&0\\
&&1
\end{pmatrix}
\quad\text{with}\quad
x'_2=(-1)^{n-1}2x_2.
\]
With these, the integral \eqref{E:id for n>r=1 2} can be written as
\begin{align*}
\begin{split}
&\int_{V_{\GL_2}(F)\backslash\GL_2(F)}\int_{\M_{1,n-1}(F)}\int_{\M_{1,n-1}(F)}\int_F dx_2dx_1dydh\\
&\quad\quad\quad\quad\quad\quad
W_v\left(
d
\begin{pmatrix}
&I_n&\\
&&I_n\\
1&
\end{pmatrix}
\begin{pmatrix}
1&x'+x_2y'&0\\
&I_{2n-1}&0\\
&&1
\end{pmatrix}
\begin{pmatrix}
1\\
0&I_{2n-1}\\
0&y'&1
\end{pmatrix}
d^{-1}
\begin{pmatrix}
1&0&x'_{2}\\
&I_{2n-1}&0\\
&&1
\end{pmatrix}
\jmath_{n,1}(h)
\right)
\xi_s(h).
\end{split}
\end{align*}
Then by first changing the variables $x'+x_2y'\mapsto x'$, $x'_2\mapsto x_2$ and then noting that 
\[
\begin{pmatrix}
1&0&x_{2}\\
&I_{2n-1}&0\\
&&1
\end{pmatrix}
=
\jmath_{n,1}\left(\pMX{1}{x_2}{}{1}\right)
\quad\text{and}\quad
\begin{pmatrix}
1&x'&0\\
&I_{2n-1}&0\\
&&1
\end{pmatrix}
\begin{pmatrix}
1\\
0&I_{2n-1}\\
0&y'&1
\end{pmatrix}
=
\begin{pmatrix}
1&x'&\\
&I_{2n-1}&\\
&y&1
\end{pmatrix}
\]
the above integral alters to
\begin{align}\label{E:id for n>r=1 3}
\begin{split}
|2|^{-1}_F\int_{\GL_2(F)}\int_{\M_{1,n-1}(F)}&\int_{\M_{1,n-1}(F)}
W_v\left(
d
\begin{pmatrix}
&I_n&\\
&&I_n\\
1&
\end{pmatrix}
\begin{pmatrix}
1&x'&\\
&I_{2n-1}&0\\
&y'&1
\end{pmatrix}
d^{-1}
\jmath_{n,1}(h)
\right)
\xi_s(h)
dx_1dydh.
\end{split}
\end{align}

\subsubsection{}
At this point, we make a change of the variable $h\mapsto w_{1,1}h$. Then since 
\[
d':
=
\jmath_{n,1}(w_{1,1})^{-1}d^{-1}\jmath_{n,1}(w_{1,1})
=
\begin{pmatrix}
(-1)^{n-1}\\
&2^{-1}I_{n-1}\\
&&1\\
&&&I'_{n-1}\\
&&&&2^{-1}
\end{pmatrix}
\]
and 
\[
\jmath_{n,1}(w_{1,1})^{-1}
\begin{pmatrix}
1&x'&\\
&I_{2n-1}&0\\
&y'&1
\end{pmatrix}
\jmath_{n,1}(w_{1,1})
=
\begin{pmatrix}
1&y'&\\
&I_{2n-1}&0\\
&x'&1
\end{pmatrix}
\]
The integral \eqref{E:id for n>r=1 3} becomes
\begin{align*}
\begin{split}
|2|^{-1}_F\int_{\GL_2(F)}\int_{\M_{1,n-1}(F)}&\int_{\M_{1,n-1}(F)}dx_1dydh\\
&W_v\left(
d
\begin{pmatrix}
&I_n&\\
&&I_n\\
1&
\end{pmatrix}
\jmath_{n,1}(w_{1,1})
\begin{pmatrix}
1&x'&\\
&I_{2n-1}&0\\
&y'&1
\end{pmatrix}
d'
\jmath_{n,1}(h)
\right)
\xi_s(w_{1,1}h).
\end{split}
\end{align*}
We then factor the $dh$ integration through $V_{\GL_2}(F)$ to obtain
\begin{align*}
\begin{split}
|2|^{-1}_F\int_{\GL_2(F)}&\int_{\M_{1,n-1}(F)}\int_{\M_{1,n-1}(F)}\int_Fdzdx_1dydh\\
&W_v\left(
d
\begin{pmatrix}
&I_n&\\
&&I_n\\
1&
\end{pmatrix}
\jmath_{n,1}(w_{1,1})
\begin{pmatrix}
1&x'&\\
&I_{2n-1}&0\\
&y'&1
\end{pmatrix}
d'
\begin{pmatrix}
1&0&z\\
&I_{2n-1}&0\\
&&1
\end{pmatrix}
\jmath_{n,1}(h)
\right)
\xi_s\left(w_{1,1}\pMX{1}{z}{}{1}h\right).
\end{split}
\end{align*}
Since 
\[
d'
\begin{pmatrix}
1&0&z\\
&I_{2n-1}&0\\
&&1
\end{pmatrix}
d'^{-1}
=
\begin{pmatrix}
1&0&z'\\
&I_{2n-1}&0\\
&&1
\end{pmatrix}
\quad\text{with}\quad
z'=(-1)^{n-1}2z
\]
and 
\[
\begin{pmatrix}
1&0&z'\\
&I_{2n-1}&0\\
&&1
\end{pmatrix}^{-1}
\begin{pmatrix}
1&x'&\\
&I_{2n-1}&0\\
&y'&1
\end{pmatrix}
\begin{pmatrix}
1&0&z'\\
&I_{2n-1}&0\\
&&1
\end{pmatrix}
=
\begin{pmatrix}
1&x'-z'y'&\\
&I_{2n-1}&0\\
&y'&1
\end{pmatrix}
\]
we see that the above integral is equal to
\begin{align}\label{E:id for n>r=1 4}
\begin{split}
|2|^{-1}_F&\int_{\GL_2(F)}\int_{\M_{1,n-1}(F)}\int_{\M_{1,n-1}(F)}\int_Fdzdx_1dydh\\
&\quad\quad
W_v\left(
d
\begin{pmatrix}
&I_n&\\
&&I_n\\
1&
\end{pmatrix}
\jmath_{n,1}(w_{1,1})
\begin{pmatrix}
1&0&z'\\
&I_{2n-1}&0\\
&&1
\end{pmatrix}
\begin{pmatrix}
1&x'&\\
&I_{2n-1}&0\\
&y'&1
\end{pmatrix}
d'
\jmath_{n,1}(h)
\right)
\xi_s\left(w_{1,1}\pMX{1}{z}{}{1}h\right).
\end{split}
\end{align}
after changing the variable $x-z'y'\mapsto x'$.

\subsubsection{}
Now, since 
\[
\jmath_{n,1}(w_{1,1})
\begin{pmatrix}
1&0&z'\\
&I_{2n-1}&0\\
&&1
\end{pmatrix}
\jmath_{n,1}(w_{1,1})^{-1}
=
\begin{pmatrix}
1\\
0&I_{2n-1}\\
z'&0&1
\end{pmatrix}
\]
and
\[
d
\begin{pmatrix}
&I_n&\\
&&I_n\\
1&
\end{pmatrix}
\begin{pmatrix}
1\\
0&I_{2n-1}\\
z'&0&1
\end{pmatrix}
\begin{pmatrix}
&I_n&\\
&&I_n\\
1&
\end{pmatrix}^{-1}
d^{-1}
=
d
\begin{pmatrix}
I_{2n-1}\\
0&1&z'\\
0&0&1
\end{pmatrix}
d^{-1}
=
\begin{pmatrix}
I_{2n-1}\\
0&1&-z'\\
0&0&1
\end{pmatrix}
\]
and also 
\[
W_v\left(
\begin{pmatrix}
I_{2n-1}\\
0&1&-z'\\
0&0&1
\end{pmatrix}
g
\right)
=
\psi^{-1}(-z')W_v(g)
=
\psi^{-1}\left((-1)^n2z\right)W_v(g)
\]
for $g\in\GL_{2n+1}(F)$, we see that integral \eqref{E:id for n>r=1 4} becomes
\begin{align*}
\begin{split}
|2|^{-1}_F\int_{\GL_2(F)}\int_{\M_{1,n-1}(F)}\int_{\M_{1,n-1}(F)}
&W_v\left(
d
\begin{pmatrix}
&I_n&\\
&&I_n\\
1&
\end{pmatrix}
\jmath_{n,1}(w_{1,1})
\begin{pmatrix}
1&x'&\\
&I_{2n-1}&0\\
&y'&1
\end{pmatrix}
d'
\jmath_{n,1}(h)
\right)\\
&\quad\quad\quad\quad\quad\quad
\int_F
\xi_s\left(w_{1,1}\pMX{1}{z}{}{1}h\right)\psi_2^{-1}\left((-1)^nz\right)
dzdx_1dydh.
\end{split}
\end{align*}
Then by changing the variable $(-1)^nz\mapsto z$ and noting that 
\[
\xi_s\left(w_{1,1}\pMX{1}{(-1)^nz}{}{1}h\right)
=
\chi_2(-1)^n\xi_s\left(w_{1,1}\pMX{1}{z}{}{1}\pMX{(-1)^n}{}{}{1}h\right)
\]
we finally get that $\gamma^{{\rm WD}}(s,\pi\x\chi_1,\psi)\Psi_{n,1}(v\ot\xi_s)$ is the same as 
\begin{align*}
\begin{split}
|2|^{-1}_F\chi_2(-1)^n\int_{\GL_2(F)}&\int_{\M_{1,n-1}(F)}\int_{\M_{1,n-1}(F)}\\
&W_v\left(
d
\begin{pmatrix}
&I_n&\\
&&I_n\\
1&
\end{pmatrix}
\jmath_{n,1}(w_{1,1})
\begin{pmatrix}
1&x'&\\
&I_{2n-1}&0\\
&y'&1
\end{pmatrix}
d'
\jmath_{n,1}\left(\pMX{(-1)^n}{}{}{1}h\right)
\right)\\
&\quad\quad\quad\quad\quad\quad\quad\quad\quad\quad\quad\quad\quad\quad\quad\quad\quad\quad
\int_F
\xi_s\left(w_{1,1}\pMX{1}{z}{}{1}h\right)\psi_2^{-1}\left(z\right)
dzdx_1dydh.
\end{split}
\end{align*}

\subsubsection{}
By replacing $\xi_s$ with $A_{\psi,\delta}(w_{1,1},\tau,s)\xi_s$, we find that 
$\gamma^{{\rm WD}}(1-s,\pi\x\chi^{-1}_2,\psi)\Psi_{n,1}(v\ot A_{\psi,\delta}(w_{1,1},\tau,s)\xi_s)$
is equal to
\begin{align*}
\begin{split}
|2|^{-1}_F\chi_1(-1)^n\int_{\GL_2(F)}&\int_{\M_{1,n-1}(F)}\int_{\M_{1,n-1}(F)}\\
&W_v\left(
d
\begin{pmatrix}
&I_n&\\
&&I_n\\
1&
\end{pmatrix}
\jmath_{n,1}(w_{1,1})
\begin{pmatrix}
1&x'&\\
&I_{2n-1}&0\\
&y'&1
\end{pmatrix}
d'
\jmath_{n,1}\left(\pMX{(-1)^n}{}{}{1}h\right)
\right)\\
&\quad\quad\quad\quad\quad\quad\quad\quad\quad\quad\quad\quad
\int_F
A_{\psi,\delta}(w_{1,1},\tau,s)\xi_s\left(w_{1,1}\pMX{1}{z}{}{1}h\right)\psi_2^{-1}\left(z\right)
dzdx_1dydh.
\end{split}
\end{align*}
Then since 
\[
\int_F
\xi_s\left(w_{1,1}\pMX{1}{z}{}{1}h\right)\psi_2^{-1}\left(z\right)dz
=
\int_F
A_\psi(w_{1,1},\tau,s)\xi_s\left(w_{1,1}\pMX{1}{z}{}{1}h\right)\psi_2^{-1}\left(z\right)dz
\]
we deduce that 
\[
\chi_2(-1)^n\gamma^{{\rm WD}}(s,\pi\x\chi_1,\psi)\Psi_{n,1}(v\ot\xi_s)
=
\chi_1(-1)^n\gamma^{{\rm WD}}(1-s,\pi\x\chi^{-1}_2,\psi)\Psi_{n,1}(v\ot A_{\psi,\delta}(w_{1,1},\tau,s)\xi_s).
\] 
From this and \eqref{E:gamma GL 1}, \eqref{E:gamma GL 2}, the identity \eqref{E:n>r=1,E=FxF} follows. 
This completes the proof of \propref{P:main'}.\qed

\section*{Appendix}\label{A}
The definition of the Rankin-Selberg gamma factors comes from the functional equation of the integrals, which, in turn, relies 
on the fact that these integrals define elements in certain Hom-spaces, the dimension of which is at most one. When the 
representations involved are irreducible, this follows from the results in the literature, as indicated below. In this appendix,
we will prove that when $F$ is non-archimedean, these Hom-spaces, which also defined for reducible representations, 
again satisfy this multiplicity-at-most-one property even when the representations involved are reducible.\\

To describe the Hom-spaces, let $Y_{n,r}\subset\U_{2n+1}$ and $Y^{n,r}\subset\U_{2r}$ be 
the unipotent subgroups given by
\[
Y_{n,r}(F)
=
\stt{
y
=
\begin{pmatrix}
I_r&0&0&b_2&0\\
a_1&z_1&x_1&c&b_1\\
&&1&x_2&0\\
&&&z_2&0\\
&&&a_2&I_r
\end{pmatrix}
\in\U_{2n+1}(F)
\mid
z_1, z_2\in V_{\GL_{n-r}}(K)
}
\]
if $n\ge r$, and
\[
Y^{n,r}(F)
=
\stt{
y
=
\begin{pmatrix}
I_{n+1}&b_2&&\\
&z_2&&\\
a_1&c&z_1&b_1\\
&a_2&&I_{n+1}
\end{pmatrix}
\in\U_{2r}(F)
\mid
z_1, z_2\in V_{\GL_{\ell}}(K)
}
\]
if $n< r$, where, as before, $\ell=r-n-1$ when $n<r$, and $K=E$ or $F$ according to $E$ is a filed or not. Except for the 
matrices $z_1, z_2$, the matrices $a_j, b_j, x_j$ and $c$ for $j=1,2$ are of implied sizes, and are such that the resulting 
matrix $y$ is contained in $\U_{2n+1}(F)$ or $\U_{2r}(F)$. In particular, when $E$ is a field, the pairs of matrices 
$(a_1, a_2)$, $(b_1, b_2)$, $(x_1, x_2)$ and $(z_1, z_2)$ are related.\\

Define a character $\psi_{Y_{n,r}}$ of $Y_{n,r}(F)$ by 
\[
\psi_{Y_{n,r}}(y)
=
\begin{cases}
\psi_{V_{\GL_{n-r}}}(z_1)\psi(2^{-1}(x_1)_{1,n-r})&\quad\text{if $E$ is a field};\\

\psi_{V_{\GL_{n-r}}}(z_1)\b{\psi}_{V_{\GL_{n-r}}}(z_2)\psi(2^{-1}(x_1)_{1,n-r}+(x_2)_{11})&\quad\text{if $E=F\,\x\, F$}.
\end{cases}
\]
On the other hand, let $\psi_{Y^{n,r}}$ be the character of $Y^{n,r}(F)$ given by 
\[
\psi_{Y^{n,r}}(y)
=
\begin{cases}
\psi_{V_{\GL_\ell}}(z_1)\psi((a_1)_{n+1,\ell}+(b_1)_{\ell,1})&\quad\text{if $E$ is a field};\\
\psi_{V_{\GL_\ell}}(z_1)\b{\psi}_{V_{\GL_\ell}}(z_2)\psi((a_1)_{\ell,n+1}+(b_1)_{\ell,1}-(a_2)_{11}-(b_2)_{n+1,1})
&\quad\text{if $E=F\,\x\, F$}.
\end{cases}
\]
The unipotent subgroup $Y_{n,r}(F)$ (resp. $Y^{n,r}(F)$) and its character $\psi_{Y_{n,r}}$ 
(resp. $\psi_{Y^{n,r}}$) are stable under the conjugation of $\jmath_{n,r}(\U_{2r}(F))$ (resp. $\jmath^{n,r}(\U_{2n+1}(F))$),
so that we can extend $\psi_{Y_{n,r}}$ (resp. $\psi_{Y^{n,r}}$) to the character of 
\[
\U_{2r}(F)\ltimes Y_{n,r}(F)\quad (\text{resp.}\,\,\U_{2n+1}(F)\ltimes Y^{n,r}(F)) 
\]
which we still denote by $\psi_{Y_{n,r}}$ (resp. $\psi_{Y^{n,r}}$).\\

Now let $\pi$ and $\tau$ be representations of $\U_{2n+1}(F)$ and $\G_r(F)$, respectively. Instead of assuming that 
they are irreducible and generic, we assume that they are of $Whittaker$ $type$, namely, the spaces of Whittaker functionals
\[
{\rm Hom}_{V_{\U_{2n+1}}(F)}\left(\pi,\psi_{\U_{2n+1}}\right)
\quad\text{and}\quad
{\rm Hom}_{V_{\G_r(F)}}\left(\tau,\psi_{\G_r}\right)
\]
of $\pi$ and $\tau$ are one-dimensional. Then the Rankin-Selberg integrals also defined for these representations, and 
converge absolutely in some right-half complex plane. Furthermore, when $F$ is non-archimedean, they also admit 
the meromorphic continuation, and in the domain of absolute convergence, become rational functions in $q^{-s}$.
In any case, when the integrals are absolute convergence, we have
\begin{equation}\label{E:equiv n>r}
\Psi_{n,r}(\pi(\jmath_{n,r}(h)y)v\ot\rho_{\tau,s}(h)\xi_s)
=
\psi_{Y_{n,r}}(y)\Psi_{n,r}(v\ot\xi_s)
\end{equation}
if $n\ge r$, where $(h,y)\in\U_{2r}(F)\ltimes Y_{n,r}(F)$ and 
\begin{equation}\label{E:equiv n<r}
\Psi_{n,r}(\pi(g)v\ot\rho_{\tau,s}(\jmath^{n,r}(g)y)\xi_s)
=
\psi_{Y^{n,r}}(y)\Psi_{n,r}(v\ot\xi_s)
\end{equation}
if $n<r$, where $(g,y)\in\U_{2n+1}(F)\ltimes Y^{n,r}(F)$.\\

Because of the above equivalence properties of the integrals, its natural to consider the following Hom-spaces 
\begin{equation}\label{E:hom-space n>r}
{\rm Hom}_{\U_{2r}(F)\ltimes Y_{n,r}(F)}(\pi\ot\rho,\psi_{Y_{n,r}})
\end{equation}
if $n\ge r$, and 
\begin{equation}\label{E:hom-space n<r}
{\rm Hom}_{\U_{2n+1}(F)\ltimes Y^{n,r}(F)}(\pi\ot\rho,\psi_{Y^{n,r}})
\end{equation}
if $n<r$. Here $\pi$ and $\rho$ are representations of $\U_{2n+1}(F)$ and $\U_{2r}(F)$, respectively, and we let
$Y_{n,r}(F)$ (resp. $Y^{n,r}(F)$) act trivially on $\rho$ (resp. $\pi$) when $n\ge r$ (resp. $n<r$). When $F$ is archimedean,
we understand that the underlying space of $\pi\ot\rho$ is the completed projective tensor product 
$\cV_\pi\widehat{\ot}\cV_\rho$, and elements in \eqref{E:hom-space n>r} and \eqref{E:hom-space n<r} are required to be 
continuous. Now we have the following result.

\begin{thmA}[\cite{JiangSunZhu2010}, \cite{SunZhu2012}, \cite{AGRS2010}, \cite{GanGrossPrasad2012}]\label{T:mult-one}
Suppose that $\pi$ and $\rho$ are irreducible. Then the spaces \eqref{E:hom-space n>r} and \eqref{E:hom-space n<r} 
have at most one-dimensional.
\end{thmA}

We remark that when $F$ is non-archimedean, \thmref{T:mult-one} for $r=n$ or $n+1$ was first established by 
Aizenbud-Gourevith-Rallis-Schiffmann in \cite{AGRS2010}; their result was later extended to arbitrary $n,r$ by 
Gan-Gross-Prasad in \cite[Section 15]{GanGrossPrasad2012}. When $F$ is archimedean, \thmref{T:mult-one} for $r=n$ or 
$n+1$ was due to Sun-Zhu in \cite{SunZhu2012}, and then due to Jiang-Sun-Zhu for arbitrary $n, r$ in 
\cite{JiangSunZhu2010}.\\ 

As pointed out in the beginning of this appendix, our aim is to remove the irreducibility assumption, at least when $F$ is 
non-archimedean. In this direction, we have the following result due to Morimoto-Soudry: 

\begin{thmA}[\cite{MorimotoSoudry2020}]
Suppose that $F$ is non-archimedean, and $\pi$, $\tau$ are of Whittaker type. Then the space \eqref{E:hom-space n<r} 
with $\rho_{\tau,s}$ being replaced by $\rho_{\tau,s}$ are one-dimensional except for countable many $s$.
\end{thmA}

We obtain a similar result when $n\ge r$.

\begin{thmA}
Suppose that $F$ is non-archimedean, and $\pi$, $\tau$ are of Whittaker type. Then the space \eqref{E:hom-space n>r} 
with $\rho_{\tau,s}$ being replaced by $\rho_{\tau,s}$ are one-dimensional except for countable many $s$.
\end{thmA}

\begin{proof}
When $E$ is a field, this result is obtained in the appendix of \cite{YCheng2}, so we assume that $E=F\,\x\,F$. The proof 
of this case is similar to the case when $E$ is a field; indeed, we only need to "double" everything. More concretely, a key 
to the proof when $E$ is a field is the existence of a finite length $P_{n+1}(F)$-filtration of certain Jacquet module of $\pi$.
Here $P_{n+1}\subset\GL_{n+1}$ is the "mirobolic subgroup". When $E=F\,\x\,F$, on the other hand, such a finite length
$P_{n+1}(F)$-filtration changes to a finite length $P_{n+1}(F)\,\x\,P_{n+1}(F)$-filtration; consequently, we only need to 
replace $P_{n+1}(F)$ by $P_{n+1}(F)\,\x\,P_{n+1}(F)$ in the rest of the arguments.\\

To describe the $P_{n+1}(F)\,\x\,P_{n+1}(F)$-filtration, let $Y\subset P\subset\GL_{2n+1}$ be the subgroups defined by 
\[
P(F)
=
\stt{
\begin{pmatrix}
a&x&c\\
&1&y\\
&&b
\end{pmatrix}
\mid 
a,b\in\GL_n(F),\,c\in\M_{n,n}(F),\,x\in\M_{n,1}(F)\,\,\text{and}\,\,y\in\M_{1,n}(F)
}
\]
and 
\[
Y(F)
=
\stt{
\begin{pmatrix}
I_n&0&c\\
&1&0\\
&&I_n
\end{pmatrix}
\mid 
c\in\M_{n,n}(F)}.
\]
Clearly, $Y$ is normal in $P$. Next, let $P'_{n+1}\subset\GL_{n+1}$ be the subgroup given by
\[
P'_{n+1}(F)
=
\stt{
\pMX{1}{y}{}{b}\mid b\in\GL_n(F)\,\,\text{and}\,\,y\in\M_{1,n}(F)
}.
\]
Observe that $P'_{n+1}$ is isomorphic to $P_{n+1}$ via the isomorphism sending $g$ to $g^*=J_{n+1}{}^tg^{-1}J_{n+1}$.
Now let $\beta$ be the following isomorphism:
\[
\beta: P(F)/Y(F)\longto P_{n+1}(F)\,\x\,P'_{n+1}(F);
\quad
\begin{pmatrix}
a&x&0\\
&1&y\\
&&b
\end{pmatrix}Y(F)
\mapsto
\left(
\pMX{a}{x}{}{1},\pMX{1}{y}{}{b}
\right).
\]
Then the Jacquet module $(\pi_Y, \cV_{\pi, Y})$ of $\pi$ with respect to $Y(F)$ admits a 
$P_{n+1}(F)\,\x\,P'_{n+1}(F)$-structure via $\beta$, namely, for $p\in P(F)$ and $v\in\cV_\pi$, we have 
\[
\pi_Y(\beta(pY(F)))\left(v+\cV_\pi(Y)\right)
=
\pi(p)v
\]
where $\cV_\pi(Y)\subset\cV_\pi$ is the subspace spanned by the elements of the form $\pi(y)w-w$ for $y\in Y(F)$ and 
$w\in\cV_\pi$.\\

The assumption that $\pi$ is of Whittaker type and the proof of \cite[Proposition 8.2]{GPSR1987} give the following finite 
sequence of $P_{n+1}(F)\,\x\,P'_{n+1}(F)$-modules
\begin{equation}\label{E:P_n+1}
0\subset \cV_0\subset \cV_1\subset\cdots\subset\cV_M=\cV_{\pi,Y}
\end{equation}
such that $\cV_{i-1}\backslash\cV_i$ is an irreducible $P_{n+1}(F)\,\x\,P'_{n+1}(F)$-module for $1\le i\le M$, 
\[
\cV_1\cong
{\rm ind}_{V_{\GL_{n+1}}(F)}^{P_{n+1}(F)}\psi_{\GL_{n+1}}
\bt
{\rm ind}_{V_{\GL_{n+1}}(F)}^{P'_{n+1}(F)}\psi^{-1}_{\GL_{n+1}}\]
and $\cV_{i+1}/\cV_i\ncong\cV_1$ for $1\le i\le M-1$. Here and below,  "ind" stands for the compact induction 
and the inductions are non-normalized.\\

At this point, let $\sR_r$ (resp. $\sR'_r$) be the subgroup of $P_{n+1}$ (resp. $P'_{n+1}$) defined by 
\[
\sR_r(F)
=
\stt{
e(x,z,y)
=
\begin{pmatrix}
I_r&0&0\\
x&z&y\\
0&0&1
\end{pmatrix}
\mid
x\in{\rm Mat}_{(n-r), r}(F), z\in V_{\GL_{n-r}}(F)\,\,\text{and}\,\, y\in\M_{1,n-r}(F)
}
\]
and respectively, 
\[
\sR'_r(F)
=
\stt{
e'(x',z',y')
=
\begin{pmatrix}
1&x'&0\\
0&z'&0\\
0&y'&1_r
\end{pmatrix}
\mid
x'\in{\rm Mat}_{1,n-r}(F), z'\in V_{\GL_{n-r}}(F)\,\,\text{and}\,\, y'\in\M_{r,n-r}(F)
}.
\]
Denote by $\psi_{\sR_r}$ (resp. $\psi_{\sR'_r}$) the character of $\sR_r(F)$ (resp, $\sR'_r(F)$) given by 
\[
\psi_{\sR_r}(e(x,y,z))
=
\psi_{\GL_{n-r}}(z)\psi(y_{1,n-r})
\]
and respectively, 
\[
\psi_{\sR'_r}(e'(x',y',z'))
=
\psi_{\GL_{n-r}}(z')\psi(x_{1,1}).
\]
Then similar arguments as in \cite[Appendix]{YCheng2} imply that the Hom-space 
\[
{\rm Hom}_{\U_{2r}(F)\ltimes Y_{n,r}(F)}(\pi\ot\rho_{\tau,s},\psi_{Y_{n,r}})
\]
is isomorphic to
\begin{equation}\label{E:Hom space}
{\rm Hom}_{P_{n+1}(F)\,\x\,P'_{n+1}(F)}
\left(
\pi_Y\ot
\left({\rm ind}^{P_{n+1}(F)}_{\GL_r(F)\ltimes\sR_r(F)}(\tau_{1,s+1}\boxtimes\psi_{\sR_r}^{-1})
\bt
{\rm ind}^{P_{n+1}(F)}_{\GL_r(F)\ltimes\sR'_r(F)}(\tau_{2,1-s}\boxtimes\psi_{\sR'_r})\right), 
\mathbb{C}\right)
\end{equation}
where we regard $\GL_r(F)$ as a subgroup of $P_{n+1}(F)$ or $P'_{n+1}(F)$ via the embeddings 
\[
a
\mapsto
\pMX{a}{}{}{I_{n-r+1}}
\quad
\text{or} 
\quad
a
\mapsto
\pMX{I_{n-r+1}}{}{}{a}
\]
respectively.\\

To analyze the Hom-space \eqref{E:Hom space}, we remind that the irreducible representations of $P'_{n+1}(F)$ are of the 
form
\[
{\rm ind}_{\cR'_m(F)}^{P'_{n+1}(F)}\sigma\boxtimes\psi^{-1}_{\GL_{n+1-m}}
\]
where $\cR'_m\subset P'_{n+1}$ is the subgroup given by
\[
\cR'_m(F)
=
\stt{
\pMX{z}{y}{}{b}\mid b\in\GL_m(F),\, y\in{\rm Mat}_{m, (n+1-m)}(F)\,\,\text{and}\,\, z\in V_{\GL_{n+1-m}}(F)
}
\]
for $0\le m\le n$, $\sigma$ is an irreducible representation of $\GL_m(F)$ and we extend the representation 
$\sigma\boxtimes\psi^{-1}_{\GL_{n+1-m}}$ of $\GL_m(F)\x V_{\GL_{n+1-m}}(F)$ to $\cR'_m(F)$ trivially across $b$. 
This can be deduced easily from the shape of irreducible representations of $P_{n+1}(F)$ (cf. \cite{BZ1977}).
Moreover, if $G$ and $G'$ are two $\ell$-groups in the sense of \cite{BZ1976}, and $\pi_i$, $\pi'_i$ are admissible 
representations of $G$, $G'$, respectively, for $i=1,2$, then
\[
{\rm Hom}_{G\x G'}\left((\pi_1\bt\pi_1')\ot(\pi_2\bt\pi'_2),\mathbb{C}\right)
\cong
{\rm Hom}_G(\pi_1\ot\pi_2,\mathbb{C})\ot{\rm Hom}_{G'}(\pi'_1\ot\pi'_2,\mathbb{C}).
\]
With these,  one can then use \eqref{E:P_n+1} and the arguments in \cite[pp. 55-57]{Soudry1993}, \cite[Appendix]{YCheng2} 
to obtain the desired result.
\end{proof}

\end{document}